\documentclass[11pt]{article}
\usepackage{graphicx}
\usepackage{amsmath}
\usepackage{amsfonts}
\usepackage[english]{babel}
\usepackage{amsthm}
\usepackage{latexsym,color}
\numberwithin{equation}{section}
\newtheorem{theorem}{Theorem}[section]

\newtheorem{proposition}[theorem]{Proposition}
\newtheorem{remark}[theorem]{Remark}
\def\neweq#1{\begin{equation}\label{#1}}
\def\endeq{\end{equation}}

\input texdraw

\newcommand{\R}{\mathbb{R}}

\newcommand{\eps}{\epsilon}

\usepackage{algorithmic,algorithm}
\usepackage{subfig}

\begin{document}


%
%

\title{On the inverse problem of detecting cardiac ischemias: \\ theoretical analysis and numerical reconstruction}


\author{Elena Beretta\thanks{Dipartimento di Matematica "F. Brioschi", Politecnico di Milano ({\tt elena.beretta@polimi.it})}
\and  Cecilia Cavaterra\thanks{Dipartimento di Matematica, Universit\`{a} degli Studi di Milano ({\tt cecilia.cavaterra@unimi.it})}
\and M.Cristina Cerutti\thanks{Dipartimento di Matematica "F. Brioschi", Politecnico di Milano" ({\tt cristina.cerutti@polimi.it})}
\and Andrea Manzoni\thanks{CMCS-MATHICSE-SB, Ecole Polytechnique F\'ed\'erale de Lausanne ({\tt andrea.manzoni@epfl.ch})}
\and Luca Ratti\thanks{Dipartimento di Matematica "F. Brioschi", Politecnico di Milano" ({\tt luca.ratti@polimi.it})}
}
 
\maketitle


\begin{abstract}
In this paper we develop theoretical analysis and numerical reconstruction techniques for the solution of an inverse boundary value problem dealing with the nonlinear, time-dependent monodomain equation, which models the evolution of the electric potential in the myocardial tissue.
The goal is the detection of  an inhomogeneity $\omega_\epsilon$ (where the coefficients of the equation are altered) located inside a domain $\Omega$  starting from observations of the potential on the boundary $\partial \Omega$. Such a problem is related to the detection of myocardial ischemic regions, characterized by severely reduced blood perfusion and consequent lack of electric conductivity.
In the first part of the paper we provide an asymptotic  formula for electric potential perturbations caused by internal conductivity inhomogeneities of low volume fraction, extending the results published in \cite{BCMP} to the case of three-dimensional, parabolic problems. In the second part we implement a reconstruction procedure based on the topological gradient of a suitable cost functional. Numerical results obtained on an idealized three-dimensional left ventricle geometry for different measurement settings assess the feasibility and robustness of the algorithm. 
\end{abstract}



\section{Introduction}

Mathematical and numerical models of computational electrophysiology can provide quantitative tools to describe electrical heart function and disfunction \cite{Quarteroni2016}, often complementing  imaging techniques (such as computed tomography and  magnetic resonance) for diagnostic and therapeutic purposes. In this context, detecting pathological conditions or reconstructing model features such as tissue conductivities from potential measurements yield to the solution of an inverse boundary value problem. Standard electrocardiographic techniques attempt to infer electrophysiological
processes in the heart from body surface measurements of the electrical potential, as in the case of electrocardiograms (ECGs), or body surface ECGs (also known as body potential maps). These measurements  can provide useful insights for the
reconstruction of the cardiac  electrical  activity  within the so-called  electrocardiographic imaging, by solving the well-known
{inverse problem of electrocardiography}\footnote{The {\em inverse problem of electrocardiography} aims at recovering the epicardial potential (that is, at the heart surface)
from body surface measurements \cite{MacLeod2010,ColliFranzone1978,Pavarino2014book}. Since the torso is considered as a passive conductor, such an inverse problem involves the linear steady diffusion model as direct problem. A step further, aiming at computing the potential inside the heart from the epicardial potential, has been considered, e.g., in \cite{Burger2010}.}. A much more invasive option to acquire potential measurements  is represented by non-contact electrodes inside a heart cavity to record endocardial potentials.

Here we focus on  the problem of detecting the position and the size of myocardial ischemias  from a single boundary measurement of the electric potential. Ischemia is a reversible precursor of heart infarction caused by partial occlusion of one or more coronary arteries, which supply blood to the heart. If this condition persists, myocardial cells die  and the ischemia eventually degenerates in infarction. For the time being, we consider an insulated heart model, neglecting the coupling with the torso; 
\textcolor{black}{this results in the inverse problem of detecting inhomogeneities for a nonlinear parabolic reaction-diffusion equation (in our case, the so-called monodomain equation) dealing with a single measurement of the  endocardial potential. Our long-term goal is indeed to deal with an inverse problem for the coupled heart-torso model, in order to detect ischemias from body surface measurements, such as those acquired on each patient with symptoms of cardiac disease through an ECG.
The problem we consider in this paper is a mathematical challenge itself, almost never considered before. Indeed, difficulties include the nonlinearity of both the direct and the inverse problem, as well as the lack of measurements at disposal. Indeed, even for the linear counterpart of the inverse problem, it has been shown in \cite{Isakov1} and \cite{Isakov2} that infinitely many measurements are needed to detect uniquely the unknown inclusions, and that the continuous dependence of the inclusion from the data is logarithmic \cite{DiCristoVessella2010}. Moreover, despite  the inverse problem of ischemia identification from measurements of surface potentials has been tackled in an
optimization framework for numerical purposes \cite{Nielsen2007,LN,Alvarez2012,Chavez2015}, a detailed mathematical analysis of this problem has never been performed.
To our knowledge, no theoretical investigation of inverse problems related with ischemia detection involving the monodomain and/or the bidomain model has been
carried out. On the other hand, recent results regarding both the analysis and the numerical approximation of this inverse problem in a much simpler stationary case
have been obtained in \cite{BCMP,BerettaManzoniRatti}.
In order to obtain rigorous theoretical results additional assumptions are needed, for instance by considering small-size conductivity inhomogeneities. We thus model ischemic regions as small inclusions  $\omega_{\eps}$  where the electric conductivity is significantly smaller than the one of healthy tissue and there is no ion transport.}

We establish a rigorous asymptotic expansion of the boundary potential perturbation due to the presence of the inclusion adapting to the parabolic nonlinear case  the approach introduced
by Capdeboscq and Vogelius in \cite{CV} for  the case of the linear conductivity equation.  The theory of detection of small conductivity inhomogeneities from boundary measurements via asymptotic techniques has been developed in the last three decades in the framework of Electric Impedence Tomography (see, e.g.,  \cite{AK,FV,CMV}). A similar approach has also been used in Thermal Imaging (see, e.g., \cite{Ammari2005}).
We use these results to set a reconstruction procedure for detecting the inclusion. To this aim, as in \cite{BerettaManzoniRatti},
we propose a reconstruction algorithm based on topological optimization, where a suitable quadratic functional is minimized to detect the position and the size of
the inclusion (see also \cite{art:cmm}).

Numerical results obtained on an idealized
left ventricle geometry 
assess the feasibility of the proposed procedure.
Several numerical test cases also show the robustness  of the reconstruction procedure with respect to measurement noise, unavoidable when dealing with real data. The modeling assumption on the small size of the inclusion, instrumental to the derivation of our theoretical results, is  verifed in practice in the case of residual ischemias after myocardial infarction. On the other hand, a fundamental task of ECG's imaging is  to detect the presence of  ischemias as precursor of heart infarction without any constraint on its size. For this reason, we also consider the case of the detection of larger size inclusions, for which the proposed algorithm still provides useful insights.

The paper is organized as follows. In Section 2 we describe the monodomain model of cardiac electrophysiology we are going to consider.
In Section 3 we show some suitable wellposedness results concerning the direct problems, in the unperturbed (background) and perturbed
cases. In Section 4 we prove useful energy estimates of the difference of the solutions of the two previous problems.
The asymptotic expansion formula is derived in Section 5 and the reconstruction algorithm in Section 6.
Numerical results are finally provided in Section 7. The appendix, Section 8, is devoted to a technical proof of a result needed in section 6.

\section{The monodomain model of cardiac electrophysiology}\label{sec:model}


The monodomain equation is a nonlinear parabolic reaction-diffusion PDE  for the transmembrane potential, providing a mathematical description of the macroscopic electric activity of the heart \cite{SLCNMT,Pavarino2014book}.
Throughout the paper we consider the following (background) initial and boundary value problem
\begin{equation}
\label{(P_0)} \ \
\begin{cases}
& \nu C_m u_t- {\rm div} (k_0\nabla u) + \nu f(u)  = 0,\ \ \ {\rm in} \ \ \Omega\times (0,T),\\
& \displaystyle\frac{\partial u}{\partial n}  = 0, \ \ \  {\rm on } \ \ \partial\Omega \times (0,T), \\
& u(0) = u_0, \ \ \ {\rm in} \ \ \Omega,
\hskip2truecm
\end{cases}
\end{equation}
where $\Omega \subset  {\bf R}^3$ is a bounded set with boundary $\partial\Omega$, and
$k_0 \in \mathbb{R}, k_0 >0$. Here $\Omega$ is the domain occupied by the ventricle, $u$ is the (transmembrane) electric potential, $f(u)$ is a nonlinear term modeling the ionic current flows across the membrane of cardiac cells,
$k_0$ is the conductivity tensor of the healthy tissue, $C_m >0$  and $\nu >0$ are two constant coefficients representing the membrane capacitance and the surface area-to-volume ratio, respectively.
For the sake of simplicity we deal with an insulated heart, namely we do not consider the effect of the surrounding torso, which behaves as a passive conductor.
The initial datum $u_0$ represents the initial activation of the tissue, arising from the propagation of the electrical impulse in the cardiac conduction system. This equation yields a macroscopic model of the cardiac
tissue, arising from the superposition of intra and extra cellular media, both assumed to occupy the whole heart volume (bidomain model), making the hypothesis that the extracellular and the intracellular conductivities are proportional quantities. Concerning the mathematical analysis of both the monodomain and the bidomain models, some results on the related direct problems have been obtained for instance in \cite{Bendahmane2006,Gerbeau2008,Bourgault2009,Pavarino2014book}.

We thus  assume a phenomenological model to describe  the effect of ionic currents through a nonlinear function of the potential.
We  neglect the coupling with the ODE system modeling the evolution of the so-called {\em gating variables}, which represent the amount of open channels per unit area of the cellular membrane and thus regulate the transmembrane
currents.


{In the case of a single gating variable $w$, a well-known option would be to replace  $f$ by $g =g(u,w)$ where
\[
g(u, w) = - \beta u(u-\alpha)(u-1) - w,
\]
and $w$ solves the following ODE initial value problem, $\forall \, x \in \Omega$,
\[
\frac{\partial w}{\partial t} = \rho(u- \gamma w) \ \  \mbox{in} \   (0,T) \, , \qquad {w}(0) = {w}_0,
\]
for suitable (constant) parameters $\beta$, $\alpha$, $\rho$, $\gamma$. This is the so-called  FitzHugh-Nagumo model for the ionic current, and the  gating variable $w$ is indeed a  recovery function allowing to take into
account the depolarization phase. See, e.g., \cite{Pavarino2014book} for more details. In our case, the model \eqref{(P_0)} is indeed widely used to characterize the large-scale propagation of the
front-like solution  in the cardiac excitable medium.}

%



As suggested in \cite[Sect. 4.2]{Pavarino2014book} and \cite[Sect. 2.2]{SLCNMT}, hereon we consider the cubic function
\begin{equation}
f(u) = A^2(u-u_1)(u-u_2)(u-u_3),  \quad u_i \in \mathbb{R}, \quad  u_1 < u_2 < u_3,
\label{H2}
\end{equation}
 where   $A>0$ is a parameter determining the rate of change of $u$ in the depolarization phase, and $u_{1} <  u_{2} < u_{3}$ are given constant values representing the resting, threshold and  peak potentials,
 respectively.  Possible values of the parameters are, e.g., $u_1 = -85 mV$, $u_{2} = - 65 mV$ and $u_3 = 40 mV$, $A=0.04$, see \cite{SLCNMT}.
 Note that both the sharpness of the wavefront and its propagation speed  strongly depend on the value of the parameter $A$.

Consider now a small inhomogeneity located in a measurable bounded domain $\omega_\varepsilon \subset \Omega$, such that there exist a
compact set $K_0$, with $\omega_\varepsilon \subset K_0 \subset\Omega$, and a constant $d_0 >0$ satisfying
\begin{equation}
{\rm dist}(\omega_\varepsilon, \Omega \backslash K_0) \geq d_0 >0.
\label{dist}
\end{equation}
Moreover, we assume
\begin{equation}
|\omega_\varepsilon| >0, \, \, \quad\quad \lim_{\varepsilon \to 0} |\omega_\varepsilon| = 0.
\label{smalldim}
\end{equation}
In the inhomogeneity $\omega_\varepsilon$ the conductivity coefficient and the nonlinearity take different values with respect the ones in
$\Omega \backslash \omega_\varepsilon$. The problem we consider is therefore
\begin{equation}
\label{perturbed} \ \
\begin{cases}
& \nu C_m u^\varepsilon_t- {\rm div} (k_\varepsilon\nabla u^\varepsilon) + \nu \chi_{\Omega \backslash \omega_\varepsilon} f(u^\varepsilon) = 0,
\ \ \ {\rm in} \ \ \Omega\times (0,T),\\
& \displaystyle {\partial u^\varepsilon \over \partial n}  = 0, \ \ \  {\rm on } \ \ \partial\Omega \times (0,T), \\
& u^\varepsilon(0) = u_0, \ \ \ {\rm in} \ \ \Omega,
\end{cases}
\hskip2truecm
\end{equation}
where $\chi_{D}$ stands for the characteristic function of a set $D \subset\mathbb{R}^3$. Here
\begin{equation}\label{eq:2_6}
k_\varepsilon = (k_0 - k_1) \chi_{\Omega \backslash \omega_\varepsilon} + k_1 =
\begin{cases}
&k_0 \quad \quad {\rm in} \ \  \Omega \backslash \omega_{\varepsilon}, \\
&k_1 \quad \quad {\rm in} \ \ \omega_\varepsilon, \\
\end{cases}
\end{equation}
with $k_0, k_1 \in \mathbb{R}, \, k_0 > k_1 >0$.

\section{Well posedness of the direct problem}

Problem \eqref{(P_0)}  thus describes the propagation of the initial activation  $u_0$ in an insulated heart portion (e.g., the left ventricle), and hereon will be referred to as the {\em background problem};
we devote Section  \ref{sec:3_1} to the analysis of its well-posedness.  The well-posedness of the perturbed problem modeling the presence of a small ischemic region in the domain will be instead analyzed in Section \ref{sec:3_2}.

\subsection{Well posedness of the background problem}\label{sec:3_1}

\setcounter{equation}{0}

%
For the sake of simplicity, throughout the paper we set
  $\nu = C_m = 1$ and we assume that
 \begin{equation}\label{eq:3_0}
 \Omega \in C^{2+\alpha}, \ \alpha \in (0,1),
 \end{equation}
\begin{equation}
u_0 \in C^{2+\alpha}(\overline\Omega),
\quad \quad u_1< u_0(x) <u_3 \quad \forall \, x \in \Omega, \quad \quad {\partial u_0 (x)\over \partial n} = 0 \quad  \forall \, x \in \partial\Omega.
\label{H3}
\end{equation}
Moreover, let us set
\begin{equation}
M_1:=\|f\|_{C([u_1,u_3])}, \quad \quad M_2 := \|f^\prime\|_{C([u_1,u_3])}.
\label{M1}
\end{equation}
The following well posedness result holds.
\begin{theorem}\label{regu}
Assume \eqref{H2}, \eqref{eq:3_0}, \eqref{H3}. Then problem \eqref{(P_0)} admits a unique solution $u\in C^{2+ \alpha, 1 + \alpha/2}(\overline\Omega \times [0,T])$
such that
\begin{equation}\label{boundu}
u_1 \leq u(x,t) \leq u_3, \quad (x,t) \in \overline\Omega \times [0,T],
\end{equation}
\begin{equation}\label{bounduc2}
\|u \|_{C^{2+ \alpha, 1 + \alpha/2}(\overline\Omega \times [0,T])}
\leq C,
\end{equation}
where  $C$ is a positive constant depending (at most) on $k_0, T, \Omega, M_1, M_2,\|u_0\|_{C^{2 + \alpha}(\overline\Omega)}$.
\end{theorem}
\begin{proof}
We omit the details of the proof since \eqref{boundu} can be easily obtained using the results in \cite[def. 3.1 and Thm. 4.1]{P}
and \eqref{bounduc2} by means of \cite[Thm. 5.1.17 (ii) and Thm. 5.1.20]{L}.



\end{proof}

\subsection{Well posedness of the perturbed problem}\label{sec:3_2}

The well-posedness of the perturbed problem \eqref{perturbed} is provided by the following theorem.
\begin{theorem}\label{regueps}
Assume \eqref{H2}, \eqref{eq:2_6}, \eqref{eq:3_0}, \eqref{H3}. Then problem \eqref{perturbed} admits a unique solution $u^\varepsilon$ such that
\begin{equation}
u^\varepsilon \in L^2(0,T; H^1(\Omega)) \cap C([0,T]; L^2(\Omega)), \quad u_t^\varepsilon \in L^2(0,T; (H^1(\Omega))^\prime) + L^{4/3}(\Omega \times (0,T)).
\label{regueps1}
\end{equation}
Moreover, $u^\varepsilon \in C^{\alpha, \alpha/2}(\overline\Omega \times [0,T])$ and the following estimate holds
\begin{equation}
\|u^\varepsilon\|_{C^{\alpha, \alpha/2}(\overline\Omega \times [0,T])}
\leq C,
\label{regueps2}
\end{equation}
where $C$ is a positive constant depending (at most) on $k_0, k_1, T, \Omega, \|u_0\|_{C^\alpha(\overline\Omega)}$ and $M_1$.
\end{theorem}

\begin{proof}
Throughout the proof $C$ will be as in the statement of the Theorem.

Recalling the definition of $f$, there exist $k \geq 0, \, \alpha_1 >0,\, \alpha_2 >0, \, \lambda >0$ such that
$$
\alpha_1u^4 - k \leq f(u)u \leq \alpha_2u^4 + k, \quad \quad f^\prime(u) \geq - \lambda.
$$
We formulate problem \eqref{perturbed}  in the weak form
\begin{equation}
\int_\Omega u^\varepsilon_t v + \int_\Omega k_\varepsilon \nabla u^\varepsilon \cdot \nabla v
+ \int_{\Omega} \chi_{\Omega \backslash\omega_\varepsilon} f(u^\varepsilon)v = 0 , \quad \forall \, v \in H^1(\Omega).
\label{wepsilon}
\end{equation}
Setting $\tilde f(u) = f(u) - u$, \eqref{wepsilon} becomes
\begin{equation}
\int_\Omega u^\varepsilon_t v + \int_\Omega k_\varepsilon \nabla u^\varepsilon \cdot \nabla v
+ \int_{\Omega} \chi_{\Omega \backslash\omega_\varepsilon} u^\varepsilon v
+ \int_{\Omega} \chi_{\Omega \backslash\omega_\varepsilon} \widetilde f(u^\varepsilon)v = 0 , \quad \forall \, v \in H^1(\Omega).
\label{wepsilon2}
\end{equation}
Observe that, thanks to following the Poincar\'{e} type inequality in \cite[formula (A.4)]{BCMP}
\begin{equation}\label{poincare}
\|z\|^2_{H^1(\Omega)} \leq S(\Omega) \left (\|\nabla z\|^2_{L^2(\Omega)}
+ \|z\|^2_{L^2(\Omega\backslash\omega_\varepsilon)} \right), \quad \forall \, z \in H^1(\Omega),
\end{equation}
the bilinear form
$a_\varepsilon(u^\varepsilon,v) = \left ( \int_\Omega k_\varepsilon \nabla u^\varepsilon \cdot \nabla v
+ \int_{\Omega \backslash\omega_\varepsilon} u^\varepsilon v\right )$ is coercive.
Indeed
\begin{equation}
a_\varepsilon(u^\varepsilon,u^\varepsilon) =  \int_\Omega k_\varepsilon |\nabla u^\varepsilon|^2 + \int_{\Omega \backslash\omega_\varepsilon} (u^\varepsilon)^2
\geq S\|u^\varepsilon\|^2_{H^1(\Omega)},
\label{aepscoer}
\end{equation}
where $S$ is a positive constant depending on $\Omega$ and $k_1$.

Through the classical Faedo-Galerkin approximation scheme it is possible to prove that problem \eqref{perturbed} admits a unique weak solution $u^\varepsilon$
satisfying \eqref{regueps1}.

In order to obtain further regularity for $u^\varepsilon$, let  $\{\phi_n\}$ be a sequence such that
\begin{equation}\nonumber
\phi_n \in C^1(\overline\Omega), \;0 \leq \phi_n(x) \leq 1, \; \forall \, x \in \overline\Omega, \;
\phi_n(x) = 1, \, \, \forall \, x \in \overline\Omega \backslash \omega_\varepsilon,\;{\rm and}\;
\phi_n \to \chi_{\Omega \setminus \omega_\varepsilon} \; {\rm in} \; L^\infty(\Omega),
\end{equation}
and formulate the approximating problems
\begin{equation}
\label{(P_n)} \ \
\begin{cases}
& u^n_t- {\rm div} (((k_0 - k_1)\phi_n + k_1)\nabla u^n) + \phi_n f(u^n) = 0,
\ \ \ {\rm in} \ \ \Omega\times (0,T),\\
& \displaystyle {\partial u^n \over \partial n}  = 0, \ \ \  {\rm on } \ \ \partial\Omega \times (0,T), \\
& u^n(0) = u_0, \ \ \ {\rm in} \ \ \Omega.
\end{cases}
\end{equation}
Using the same arguments as in the proof of Theorem \ref{regu}, we can prove that, $\forall \, n \in \mathbb{N}$, problem \eqref{(P_n)}
admits a unique solution $u^n$ such that
$$
u^n \in C(\overline\Omega \times [0,T]), \quad u_1 \leq u^n(x,t) \leq u_3, \quad (x,t) \in \overline\Omega \times [0,T].
$$
Moreover, by means again of a standard Faedo-Galerkin approximation scheme (for any $n$) we can prove that the solution to problem \eqref{(P_n)}
satisfies
$$
u^n \in L^2(0,T; H^1(\Omega)) \cap C([0,T]; L^2(\Omega)),
\quad u_t^n \in L^2(0,T; (H^1(\Omega))^\prime) + L^{4/3}(\Omega \times (0,T)),
$$
$$
\|u^n\|^2_{L^\infty(0,T; L^2(\Omega))} \leq C, \quad \quad
\|u^n\|^2_{L^2(0,T; H^1(\Omega))} \leq C,
$$
$$
\|u_t^n|^2_{L^{4/3}(0,T; (H^1(\Omega))^\prime)} \leq C, \quad\quad
\|\phi_n\widetilde f(u^n)\|^2_{L^{4/3}(\Omega \times (0,T))} \leq C,
$$
where $C$ are some positive constants independent of $n$.

An application of \cite[Thm. 8.1]{R} implies that, up to a subsequence,
$u^n \to \zeta$ strongly in $L^2(\Omega \times (0,T))$,
so that
$u^n \to \zeta$ a.e. in $\Omega \times (0,T)$ and
$\phi_n\widetilde f(u^n) \rightharpoonup \chi_{\omega_\varepsilon} \widetilde f(\zeta)\,\, {\rm in}\,\, L^{4/3}(\Omega \times (0,T))$.
Since problem \eqref{perturbed} has a unique solution (cf. \eqref{wepsilon}), we conclude that $\zeta=u^\varepsilon$
and satisfies
\begin{align}\nonumber
\|u^\varepsilon\|^2_{L^\infty(0,T; L^2(\Omega))} \leq C, \quad &\|u^\varepsilon\|^2_{L^2(0,T; H^1(\Omega))} \leq C,
\quad \|u_t^\varepsilon\|^2_{L^{4/3}(0,T; (H^1(\Omega))^\prime)} \leq C,\\
\label{boundueps}
&u_1 \leq u^\varepsilon (x,t)\leq u_3,\quad {\rm in}\, \, \overline\Omega \times [0,T].
\end{align}
Considering now the interior regularity result in \cite[Theorem 2.1]{D} (see also \cite{LSU})
and the regularity up to the boundary contained in \cite[Theorem 4.1]{D}, then we deduce \eqref{regueps2}.
\end{proof}

\section{Energy estimates for $\mathbf{u^\varepsilon - u}$ }

\setcounter{equation}{0}

In this section we prove some energy estimates for the difference between  $u^\varepsilon$ and $u$, solutions to problem \eqref{perturbed} and problem \eqref{(P_0)}, respectively, that are crucial to establish the asymptotic formula for $u^\varepsilon-u$ of Theorem 4 in Section 5.
\begin{proposition}
Assume \eqref{H2}, \eqref{eq:2_6}, \eqref{eq:3_0}, \eqref{H3}. Setting $w := u^{\varepsilon} - u$, then
\begin{equation}\label{estimate3}
 \|w\|_{L^\infty(0,T;L^2(\Omega))} \leq C|\omega_\varepsilon|^{1/2},
\end{equation}
\begin{equation}\label{estw}
\|w\|_{L^2(0,T;H^1(\Omega))} \leq C|\omega_\varepsilon|^{1/2}.
\end{equation}
Moreover, there exists $0<\beta<1$ such that
\begin{equation}\label{estw3}
\|w\|_{L^2(\Omega\times (0,T))} \leq C |\omega_\varepsilon|^{\frac{1}{2} + \beta}.
\end{equation}
Here $C$ stands for a positive constant depending (at most) on $k_0, k_1, \Omega, T, M_1, M_2, \|u_0\|_{C^{2+\alpha}(\overline\Omega)}$.
\end{proposition}
\begin{proof}
Throughout the proof $C$ will be as in the statement of the Theorem.

On account of the assumptions, Theorems \ref{regu} and \ref{regueps} hold. Then $w$ solves the problem
\begin{equation}\label{PD}
\begin{cases}
&\displaystyle w_t- {\rm div} (k_\varepsilon\nabla w)
+ \chi_{\Omega / \omega_\varepsilon} p_\varepsilon w
= - {\rm div} (\widetilde k \chi_{\omega_\varepsilon} \nabla u) + \chi_{\omega_\varepsilon} f(u), \ {\rm in} \ \Omega\times (0,T),\\
& \displaystyle {\partial w \over \partial n}  = 0, \  {\rm on } \  \partial\Omega \times (0,T), \\
& w(0) = 0,  \ {\rm in} \ \Omega,
\end{cases}
\end{equation}
where we have set $\widetilde k := k_0 - k_1 > 0$ and
\begin{equation}\label{peps}
p_\varepsilon w := f^\prime(z_\varepsilon)w = f(u^{\varepsilon}) - f(u),
\end{equation}
$z_\varepsilon$ being a value between $u^\varepsilon$ and $u$. By means of \eqref{boundu}, \eqref{boundueps}
and recalling \eqref{M1}, we have
\begin{equation}\label{M3}
u_1 \leq z_\varepsilon \leq u_3, \quad \quad |p_\varepsilon| = |f^\prime(z_\varepsilon)| \leq M_2.
\end{equation}
Multiplying the first equation by $w$ in \eqref{PD} by $w$  and integrating by parts over $\Omega$, we get
$$
\frac{1}{2} \frac{d}{dt}\int_\Omega w^2 + \int_\Omega k_\varepsilon|\nabla w|^2
+ \int_\Omega\chi_{\Omega / \omega_\varepsilon} p_\varepsilon w^2
= \int_\Omega \widetilde k \chi_{\omega_\varepsilon} \nabla u\nabla w + \int_\Omega \chi_{\omega_\varepsilon} f(u)w.
$$
Adding and subtracting $\displaystyle{\int_\Omega \chi_{\Omega\setminus\omega_\varepsilon} w^2}$ and applying \eqref{aepscoer} we obtain
$$
\frac{1}{2} \frac{d}{dt}\int_\Omega w^2 + S\|w\|^2_{H^1(\Omega)}
\leq \int_{\omega_\varepsilon} \widetilde k  \nabla u\nabla w + \int_{\omega_\varepsilon} f(u)w
- \int_\Omega\chi_{\Omega / \omega_\varepsilon} (p_\varepsilon - 1) w^2.
$$
Recalling \eqref{M1} and \eqref{M3}, thanks to Young's inequality we deduce
$$
\frac{1}{2} \frac{d}{dt}\int_\Omega w^2 + S\|w\|^2_{H^1(\Omega)}
\leq \widetilde k\left (\frac{\widetilde k}{2S}\int_{\omega_\varepsilon}  | \nabla u|^2 + \frac{S}{2\widetilde k}\int_{\Omega}|\nabla w|^2\right)
+ \frac{1}{2}\int_{\omega_\varepsilon} (f(u))^2 
+ \int_\Omega (M_2 + \frac32) w^2,
$$
so that
\begin{equation}\label{estimate1}
\frac{1}{2} \frac{d}{dt}\int_\Omega w^2 + \frac{S}{2}\|w\|^2_{H^1(\Omega)}
\leq \frac{(\widetilde k)^2}{2S}\int_{\omega_\varepsilon}  | \nabla u|^2
+ \frac{1}{2}\int_{\omega_\varepsilon} M_1^2
+  (M_2 + \frac32) \int_\Omega w^2,
\end{equation}
and finally, see \eqref{bounduc2},
$$
 \frac{d}{dt}\|w(t)\|^2_{L^2(\Omega)}
\leq C \left (|\omega_\varepsilon| +  \|w(t)\|^2_{L^2(\Omega)}\right ).
$$
Recalling $w(0) = 0$, an application of Gronwall's Lemma implies
\begin{equation}\label{estimate2}
 \|w(t)\|^2_{L^2(\Omega)} \leq C|\omega_\varepsilon|, \quad t \in (0,T),
\end{equation}
so that \eqref{estimate3} follows.
Integrating now inequality \eqref{estimate1} on $(0,T)$ we get
\begin{equation*}
\int_\Omega w^2(T) + C\int_0^T\|w(t)\|^2_{H^1(\Omega)}dt
\leq C \left (|\omega_\varepsilon| +  \int_0^T\|w(t)\|^2_{L^2(\Omega)}dt\right),
\end{equation*}
and a combination with \eqref{estimate2} gives \eqref{estw}.

In order to obtain the more refined estimate \eqref{estw3}, observe that $w$ also
solves problem
\begin{equation}\label{PD2}
\begin{cases}
&\displaystyle w_t- {\rm div} (k_0\nabla w)
+ \chi_{\Omega / \omega_\varepsilon} p_\varepsilon w
= - {\rm div} (\widetilde k \chi_{\omega_\varepsilon} \nabla u^\varepsilon) + \chi_{\omega_\varepsilon} f(u), \ {\rm in} \ \Omega\times (0,T),\\
& \displaystyle {\partial w \over \partial n}  = 0, \  {\rm on } \  \partial\Omega \times (0,T), \\
& w(0) = 0,  \ {\rm in} \ \Omega.
\end{cases}
\end{equation}
Let's now introduce the auxiliary function $\overline w$, solution to the adjoint problem
\begin{equation}\label{PA}
\begin{cases}
&\displaystyle {\overline w}_t + {\rm div} (k_0\nabla \overline w)
- \chi_{\Omega / \omega_\varepsilon} p_\varepsilon \overline w
= - w, \ {\rm in} \ \Omega\times (0,T),\\
& \displaystyle {\partial \overline w \over \partial n}  = 0, \  {\rm on } \  \partial\Omega \times (0,T), \\
& \overline w(T) = 0,  \ {\rm in} \ \Omega.
\end{cases}
\end{equation}
By the change of variable $t \to T - t$, problem \eqref{PA} is equivalent to
\begin{equation}\label{PA2}
\begin{cases}
&\displaystyle z_t - {\rm div} (k_0\nabla z)
+ \chi_{\Omega / \omega_\varepsilon} \hat p_\varepsilon z
= \hat w, \ {\rm in} \ \Omega\times (0,T),\\
& \displaystyle {\partial z \over \partial n}  = 0, \  {\rm on } \  \partial\Omega \times (0,T), \\
& z (0) = 0,  \ {\rm in} \ \Omega,
\end{cases}
\end{equation}
where we have set $z(x,t) = \overline w(x, T - t), \, \hat p_\varepsilon (x,t) =  p_\varepsilon (x,T - t), \, \hat w (x,t)= w(x, T-t)$.

Since  $|\chi_{\Omega / \omega_\varepsilon} \hat p_\varepsilon|$ is bounded in $\Omega \times (0,T)$ and $w \in C^{\alpha, \alpha/2}(\overline\Omega\times [0,T])$,
standard arguments show that problem \eqref{PA2} admits a unique solution $z$ such that (see \cite[Ch.4]{LSU})
$$z \in W^{2,1}_2(\Omega \times (0,T)) := \left\{ z \in L^2(\Omega \times (0,T)) \, | \, z \in H^1(0,T; L^2(\Omega)) \cap L^2(0,T;H^2(\Omega))\right\}.$$

\textcolor{black}{Hereon, up to equation \eqref{est6z}, all the equations depend on $t$ and are valid for every $t \in (0,T)$; however, we will omit  this dependence for the sake of notation.}

Moreover, multiplying the first equation in \eqref{PA2} by $z$ and integrating over $\Omega$, we get
$$
\frac{1}{2}\frac{d}{dt}\int_\Omega z^2 +  k_0\int_\Omega|\nabla z|^2 + k_0\int_\Omega z^2
= \int_\Omega \hat w z -\int_\Omega \chi_{\Omega / \omega_\varepsilon} \hat p_\varepsilon z^2 + k_0\int_\Omega z^2.
$$
By means of Young's inequality and recalling \eqref{M3}, we have
\begin{equation}\label{est1z}
\frac{1}{2}\frac{d}{dt}\| z\|_{L^2(\Omega)}^2 +  \frac{k_0}{2}\|z\|^2_{H^1(\Omega)}
\leq \frac{1}{2k_0}\|\hat w\|_{L^2(\Omega)}^2  + (M_2 + k_0) \| z\|_{L^2(\Omega)}^2,
\end{equation}
and then
\begin{equation}\label{est2z}
\frac{d}{dt}\| z\|_{L^2(\Omega)}^2
\leq \frac{1}{k_0}\|\hat w\|_{L^2(\Omega)}^2 + 2(M_2 + k_0) \| z\|_{L^2(\Omega)}^2.
\end{equation}
Recalling $z(x,0)= 0$, an application of Gronwall's Lemma gives
\begin{equation}\label{est3z}
\| z \|_{L^2(\Omega)}^2
\leq C\|\hat w\|_{L^2(\Omega)}^2.
\end{equation}
Let's now multiply the first equation in \eqref{PA2} by $z_t$ and integrate over $\Omega$. We get
$$
\int_\Omega z_t^2 +  \frac{k_0}{2}\frac{d}{dt}\int_\Omega |\nabla z|^2
= \int_\Omega \hat w z_t -\int_\Omega \chi_{\Omega / \omega_\varepsilon} \hat p_\varepsilon z z_t.
$$
An application of Young's inequality gives
$$
\frac{1}{2}\int_\Omega z_t^2 +  \frac{k_0}{2}\frac{d}{dt}\int_\Omega|\nabla z|^2
\leq \int_\Omega (\hat w)^2 +\int_\Omega \chi_{\Omega / \omega_\varepsilon} (\hat p_\varepsilon)^2 z^2,
$$
and then
\begin{equation}\label{est6z}
\frac{1}{2}\|z_t\|_{L^2(\Omega)}^2 +  \frac{k_0}{2}\frac{d}{dt}\|\nabla z\|_{L^2(\Omega)}^2
\leq \|\hat w\|_{L^2(\Omega)}^2 +M_2^2\| z\|_{L^2(\Omega)}^2.
\end{equation}
Combining \eqref{est6z} and \eqref{est3z},
integrating in time on $(0,t)$, and using $\nabla z(0) = 0$ we deduce
$$
\|\nabla z(t)\|_{L^2(\Omega)}^2 \leq C \|\hat w\|_{L^2(\Omega \times (0,t))}^2, \quad t \in (0,T),
$$
so that
\begin{equation}\label{est6zbis}
\|z\|_{L^\infty(0,T;H^1(\Omega))}^2 \leq C \|\hat w\|_{L^2(\Omega \times (0,T))}^2.
\end{equation}
The same computations also gives

\begin{equation}\label{est7z}
\|z_t\|_{L^2(\Omega \times (0,T))}^2
\leq C\|\hat w\|_{L^2(\Omega \times (0,T))}^2.
\end{equation}
Then, an application of standard elliptic regularity results to problem \eqref{PA2} implies (see \cite{GT})
\begin{equation}\label{est7zbis}
\|z\|_{L^2(0,T;H^2(\Omega)}^2
\leq C\|\hat w\|_{L^2(\Omega \times (0,T))}^2.
\end{equation}
Recalling the definition of $z$ and $\hat w$, by estimates \eqref{est6zbis} and \eqref{est7zbis}
we get
\begin{equation}\label{est8z}
\|\overline w \|^2_{L^\infty(0,T;H^1(\Omega))} + \|\overline w\|_{L^2(0,T;H^2(\Omega))}^2
\leq C\|w\|_{L^2(\Omega \times (0,T))}^2,
\end{equation}
Finally, we want to prove that there exists $p >2$ such that
\begin{equation}\label{est9zbis}
\|\overline w\|_{L^p(\Omega\times (0,T))} + \|\nabla \overline w\|_{L^p(\Omega\times (0,T))} \leq C\|w\|_{L^2(\Omega \times (0,T))}.
\end{equation}
To this aim, on account of \eqref{est8z} and Sobolev immersion theorems, we deduce
\begin{equation}\label{est10z}
\|\overline w \|^2_{L^6(\Omega \times (0,T))} \leq C\|\overline w \|^2_{L^\infty(0,T;H^1(\Omega))}
\leq C\|w\|_{L^2(\Omega \times (0,T))}^2.
\end{equation}
Moreover, again from \eqref{est8z} we have
\begin{equation}\label{est11z}
\nabla \overline w \in L^\infty(0,T; L^2(\Omega)) \cap L^2(0,T;L^6(\Omega)).
\end{equation}
From well-known interpolation estimates (cf. \cite{N}) we infer
\begin{equation} \label{est12z}
\|\nabla \overline w \|^{10/3}_{L^{10/3}(\Omega \times (0,T))}
\leq C\|\nabla \overline w \|^{2}_{L^2(0,T;L^6(\Omega))} \|\nabla \overline w \|^{4/3}_{L^{4/3}(0,T;L^2(\Omega))}
\end{equation}
and therefore, using \eqref{est8z},
\begin{equation} \label{est13z}
\|\nabla \overline w \|^{10/3}_{L^{10/3}(\Omega \times (0,T))}
\leq C\| w \|^{2}_{L^2(\Omega \times (0,T))} \| w \|^{4/3}_{L^2(\Omega \times (0,T))} \leq C \| w \|^{10/3}_{L^2(\Omega \times (0,T))},
\end{equation}
so that \eqref{est9zbis} holds for any $p \in (2, \frac{10}{3}]$.

Let us now multiply the evolution equation in \eqref{PD2} by $\overline w$ and the evolution equation in \eqref{PA} by $w$, respectively.
Integrating on $\Omega$ we obtain
\begin{equation}\label{woverw1}
\int_\Omega w_t \overline w + k_0 \int_\Omega \nabla w \cdot  \nabla \overline w
+ \int_\Omega \chi_{\Omega / \omega_\varepsilon} p_\varepsilon w \overline w
=  \widetilde k \int_{\omega_\varepsilon} \nabla u^\varepsilon \nabla \overline w  + \int _{\omega_\varepsilon} f(u) \overline w,
\end{equation}
\begin{equation}\label{woverw2}
\int_\Omega \overline w_t  w -  k_0 \int_\Omega \nabla \overline w \cdot  \nabla  w
- \int_\Omega \chi_{\Omega / \omega_\varepsilon} p_\varepsilon \overline w  w
= -  \int _\Omega w^2.
\end{equation}
Summing up \eqref{woverw1} and \eqref{woverw2} we obtain 
$$
\int_\Omega (w_t \overline w + \overline w_t  w) =
\widetilde k \int_{\omega_\varepsilon} \nabla u^\varepsilon \nabla \overline w  + \int _{\omega_\varepsilon} f(u) \overline w
 -  \int _\Omega w^2,
$$
subsequently, an integration in time on $(0,T)$ gives
\begin{equation}\label{woverw3}
\int_0^T \int _\Omega w^2
= -\int_0^T\int_\Omega (w_t \overline w + \overline w_t  w)
+ \widetilde k \int_0^T\int_{\omega_\varepsilon} \nabla u^\varepsilon \nabla \overline w  + \int_0^T\int _{\omega_\varepsilon} f(u) \overline w.
\end{equation}
Recalling the conditions at time $t=0$ for $w$ and at time $t=T$ for $\overline w$, we get
\[
\begin{array}{rl}
\displaystyle \int_0^T\int_\Omega (w_t \overline w + \overline w_t  w) = \int_\Omega \int_0^T \! \left (w_t \overline w +\overline w_t  w \right) & \\
& \hspace{-5cm} \displaystyle  = \int_\Omega\Big ((w\overline w)(T) - (w\overline w)(0) - \int_0^T (w \overline w_t  + \overline w_t  w) \Big)
=0
\end{array}
\]
So that \eqref{woverw3} becomes
\begin{equation}\label{woverw4}
\int_0^T \int _\Omega w^2
= \widetilde k \int_0^T\int_{\omega_\varepsilon} \nabla u^\varepsilon \nabla \overline w  + \int_0^T\int _{\omega_\varepsilon} f(u) \overline w.
\end{equation}
Using now H\"{o}lder inequality we deduce
$$\|w\|_{L^2(\Omega\times (0,T))}^2
\leq \|\nabla u^\varepsilon\|_{L^{q}(\omega_\varepsilon \times (0,T))}\|\nabla \overline w\|_{L^{p}(\omega_\varepsilon \times (0,T))}
+\|f(u)\|_{L^{q}(\omega_\varepsilon \times (0,T))} \|\overline w\|_{L^{p}(\omega_\varepsilon \times (0,T))},
$$
where we may choose for instance $p=10/3$ and $q = 10/7$.

By means of \eqref{est9zbis} and  \eqref{M1},  from the previous inequality we get
$$
\|w\|^2_{L^2(\Omega\times (0,T))} \leq C\| w\|_{L^{2}(\Omega \times (0,T))}\left (\|\nabla u^\varepsilon\|_{L^{q}(\omega_\varepsilon \times (0,T))}
+ |\omega_\varepsilon|^{\frac{1}{q}} \right),
$$
and therefore
\begin{equation}\label{estw2}
\|w\|_{L^2(\Omega\times (0,T))} \leq C\left (\|\nabla u^\varepsilon\|_{L^{q}(\omega_\varepsilon \times (0,T))}
+ |\omega_\varepsilon|^{\frac{1}{q}} \right ).
\end{equation}
Thanks to \eqref{bounduc2} we also have
\[
\begin{array}{rl}
\|\nabla u^\varepsilon\|_{L^{q}(\omega_\varepsilon \times (0,T))}  & \leq
\|\nabla u^\varepsilon - \nabla u\|_{L^{q}(\omega_\varepsilon \times (0,T))} + \|\nabla u\|_{L^{q}(\omega_\varepsilon \times (0,T))} \\
& \leq \|\nabla w \|_{L^{q}(\omega_\varepsilon \times (0,T))} + C|\omega_\varepsilon|^{\frac{1}{q}}.
\end{array}
\]
Finally, using again H\"{o}lder inequality and \eqref{estw}, and recalling that $q \in [10/7,)2$, we obtain
\begin{align}\nonumber
\|\nabla w \|_{L^{q}(\omega_\varepsilon \times (0,T))} &\leq
\left ( \int_0^T\left (\int_{\omega_\varepsilon}|\nabla w (t)|^{q \frac{2}{p^\prime}}\right)^{\frac{q}{2}}
\left (\int_{\omega_\varepsilon} 1 \right)^{\frac{2-q}{2}}\right )^{\frac{1}{q}}\\ \nonumber
&\leq |\omega_\varepsilon|^{\frac{1}{q} - \frac{1}{2}}\left (\int_0^T\|\nabla w(t)\|^{q}_{L^2(\Omega)} \right )^{\frac{1}{q}}
\leq  |\omega_\varepsilon|^{\frac{1}{q} - \frac{1}{2}}\|\nabla w\|_{L^{q}(0,T; L^2\Omega))} \\ \nonumber
&\leq C(\Omega) |\omega_\varepsilon|^{\frac{1}{q} - \frac{1}{2}}\|\nabla w\|_{L^2(\Omega \times (0,T))}
\leq C |\omega_\varepsilon|^{\frac{1}{q}}.
\end{align}
Combining the previous estimate with \eqref{estw2}, since $\frac{1}{q}\in(\frac{1}{2},\frac{7}{10}]$ we can conclude that \eqref{estw3} holds with $\beta\in(0,\frac{1}{5}]$.
\end{proof}

\section{The asymptotic formula}

\setcounter{equation}{0}

In this section we derive and prove an asymptotic representation formula for $w = u_\varepsilon - u$ in analogy with \cite{BCMP} and \cite{CV}.
Let $\Phi =\Phi(x,t)$ be any solution of
\begin{equation}\label{(PPhi)}
\begin{cases}
& \Phi_t + k_0\Delta \Phi  - f^\prime(u)\Phi  = 0,\ \ \ {\rm in} \ \ \Omega\times (0,T),\\
& \Phi(T) = 0, \ \ \ {\rm in} \ \ \Omega.
\hskip2truecm
\end{cases}
\end{equation}

Our main result is the following
\begin{theorem}\label{expansion}
Assume \eqref{H2}, \eqref{eq:2_6}, \eqref{eq:3_0}, \eqref{H3}. Let $u^\varepsilon$ and $u$ be the solutions to \eqref{perturbed} and \eqref{(P_0)}
and $\Phi$ a solution to \eqref{(PPhi)}, respectively. Then, there exist a sequence $\omega_{\varepsilon_n}$ satisfying \eqref{dist} and \eqref{smalldim}
with $|\omega_{\varepsilon_n}|\rightarrow 0$, a regular Borel measure $\mu$ and a symmetric matrix $\mathcal{M}$ with elements
$\mathcal{M}_{ij} \in L^2(\Omega, d\mu)$ such that, for $\varepsilon \to 0$,
\begin{align}\label{wPhi8}
\int_0^T\int_{\partial\Omega} k_0\frac{\partial\Phi}{\partial n}(u^\varepsilon - u)
= |\omega_{\varepsilon_n}| \left\{ \int_0^T\int_\Omega \tilde k \mathcal{M} \nabla u\cdot  \nabla \Phi d\mu
+  \int_0^T\int_\Omega  f(u)\Phi d\mu + o(1) \right\}.
\end{align}
\end{theorem}
To prove Theorem \ref{expansion}, we need to state some preliminary results. Let  $v_\varepsilon^{(j)}$ and $v^{(j)}$ be the variational solutions
(depending only on $x \in \Omega$) to the problems
\begin{equation}\label{PV0eps}
(PV_\varepsilon)
\left\{\begin{array}{ll}
&{\rm div} (k_\varepsilon\nabla v_\varepsilon^{(j)})  = 0,\ \  {\rm in}  \ \Omega,\cr
&\frac{\partial v_\varepsilon^{(j)}}{\partial n}  = n_j, \ \   {\rm on }  \ \partial\Omega, \cr
&\int_{\partial\Omega} v_\varepsilon^{(j)} = 0,
\end{array}
\right.
\quad
(PV_0) \
\left\{\begin{array}{ll}
&{\rm div} (k_0\nabla v^{(j)})  = 0,\  \ {\rm in}  \ \Omega,\cr
&\frac{\partial v^{(j)}}{\partial n}  = n_j, \ \   {\rm on }  \ \partial\Omega, \cr
&\int_{\partial\Omega} v^{(j)} =0,
\end{array}
\right.
\end{equation}
$n_j$ being the $j-th$ coordinate of the outward normal to $\partial\Omega$. It can be easily verified that
\begin{equation} \label{vj}
v^{(j)} = x_j - \frac{1}{|\partial\Omega|} \int_{\partial\Omega}x_j.
\end{equation}
The following results hold
\begin{proposition}
Let $v_\varepsilon^{(j)}$ and $v^{(j)}$ solutions to \eqref{PV0eps}, then there exists $C(\Omega)>0$ such that
\begin{equation}\label{estv1}
\|v_\varepsilon^{(j)} - v^{(j)}\|_{H^1(\Omega)} \leq C(\Omega)|\omega_\varepsilon|^{\frac12}.
\end{equation}
Moreover, for some $\eta \in (0, \frac12)$, there exists $C(\Omega, \eta)>0$ such that
\begin{equation}\label{estv2}
\|v_\varepsilon^{(j)} - v^{(j)}\|_{L^2(\Omega)} \leq C(\Omega, \eta)|\omega_\varepsilon|^{\frac12 + \eta}.
\end{equation}
\end{proposition}
\begin{proof}
See Lemma 1 in \cite{CV}.
\end{proof}
\begin{proposition}
Let $u$ and $u_\varepsilon$ be the solutions to problems \eqref{(P_0)} and $\eqref{perturbed}$, respectively.
Consider $v_\varepsilon^{(j)}$ and  $v^{(j)}$ as in \eqref{PV0eps}.
Then, for any $\Phi \in C^1(\overline\Omega \times [0,T])$ with $\Phi(x,T) = 0$, the folllowing holds as $\varepsilon \to 0$,
\begin{equation} \label{I0}
\int_0^T\int_{\Omega}\frac{1}{|\omega_\varepsilon|} \chi_{\omega_\varepsilon} \nabla u\cdot \nabla v_\varepsilon^{(j)} \Phi dxdt =
\int_0^T\int_{\Omega}\frac{1}{|\omega_\varepsilon|} \chi_{\omega_\varepsilon}\nabla u_\varepsilon \cdot \nabla v^{(j)} \Phi dxdt + o(1).
\end{equation}
\end{proposition}
\begin{proof}
We follow the ideas in \cite{BCMP} and \cite{CV}.
Since $w = u_\varepsilon - u$, then we obtain the identity
\begin{align}\nonumber
\int_\Omega k_0\nabla w\cdot \nabla \left (v^{(j)}\Phi \right) &= \int_\Omega k_0\nabla w\cdot \nabla v^{(j)}\Phi
+ \int_\Omega k_0\nabla w\cdot \nabla \Phi v^{(j)}\\ \label{V1}
&= - \int_\Omega k_0 w \nabla v^{(j)} \cdot \nabla \Phi + \int_{\partial\Omega} k_0 w n_j\Phi
+\int_\Omega k_0\nabla w\cdot \nabla \Phi v^{(j)}.
\end{align}
Moreover, we have
$$\int_0^T\!\int_\Omega k_\varepsilon\nabla w\cdot \nabla \left (v_\varepsilon^{(j)}\Phi \right)
= \int_0^T\!\int_\Omega \left (  k_\varepsilon\nabla w\cdot \nabla v_\varepsilon^{(j)}\Phi
+  k_\varepsilon\nabla w\cdot \nabla \Phi v^{(j)}
+  k_\varepsilon\nabla w\cdot \nabla \Phi (v^{(j)}_\varepsilon - v^{(j)})\right )$$
$$= \int_0^T \int_\Omega  k_\varepsilon\nabla w\cdot \nabla v_\varepsilon^{(j)}\Phi
+ \int_\Omega k_0\nabla w\cdot \nabla \Phi v^{(j)}
+  \int_\Omega  (k_\varepsilon - k_0)\nabla w\cdot \nabla \Phi v^{(j)} + \int_\Omega k_\varepsilon\nabla w\cdot \nabla \Phi (v^{(j)}_\varepsilon - v^{(j)})
$$
$$= \int_0^T \Big( - \int_\Omega k_\varepsilon w \nabla v_\varepsilon^{(j)}\cdot \nabla \Phi
+ \int_{\partial\Omega}k_0 w n_j\Phi
+ \int_\Omega k_0\nabla w\cdot \nabla \Phi v^{(j)}$$
$$+ \int_\Omega (k_\varepsilon - k_0)\nabla w\cdot \nabla \Phi v^{(j)}
+ \int_\Omega k_\varepsilon\nabla w\cdot \nabla \Phi (v^{(j)}_\varepsilon - v^{(j)}) \Big)$$
$$= \int_0^T \Big( - \int_\Omega k_\varepsilon w \nabla v^{(j)}\cdot \nabla \Phi
+ \int_{\partial\Omega}k_0 w n_j\Phi
+ \int_\Omega k_0\nabla w\cdot \nabla \Phi v^{(j)}$$
$$+ \int_{\omega_\varepsilon} (k_1 - k_0)\nabla w\cdot \nabla \Phi v^{(j)} + \int_\Omega k_\varepsilon\nabla w\cdot \nabla \Phi (v^{(j)}_\varepsilon - v^{(j)})
- \int_\Omega k_\varepsilon w \nabla (v^{(j)}_\varepsilon - v^{(j)})\cdot \nabla \Phi \Big).
$$
A combination with \eqref{V1} gives
$$\int_0^T\int_\Omega k_\varepsilon\nabla w\cdot \nabla (v_\varepsilon^{(j)}\Phi )
= \int_0^T \Big( \int_\Omega k_0\nabla w\cdot \nabla  (v^{(j)}\Phi )  +  \int_\Omega(k_0 - k_\varepsilon) w \nabla v^{(j)}\cdot \nabla \Phi$$
$$+ \int_{\omega_\varepsilon} (k_1 - k_0)\nabla w\cdot \nabla \Phi v^{(j)} + \int_\Omega k_\varepsilon\nabla w\cdot \nabla \Phi (v^{(j)}_\varepsilon - v^{(j)})
- \int_\Omega k_\varepsilon w\cdot \nabla (v^{(j)}_\varepsilon - v^{(j)})\cdot \nabla \Phi \Big)$$
$$= \int_0^T \Big( \int_\Omega k_0\nabla w\cdot \nabla \left (v^{(j)}\Phi \right) + \int_{\omega_\varepsilon} (k_1 - k_0)\nabla w\cdot \nabla \Phi v^{(j)}
$$
$$+  \int_{\omega_\varepsilon}(k_0 - k_1) w \nabla v^{(j)}\cdot \nabla \Phi + \int_\Omega k_\varepsilon\nabla w\cdot \nabla \Phi (v^{(j)}_\varepsilon - v^{(j)})
- \int_\Omega k_\varepsilon w \nabla (v^{(j)}_\varepsilon - v^{(j)})\cdot \nabla \Phi \Big).
$$
Then, on account of \eqref{estimate3}, \eqref{estw}, \eqref{estv1}, \eqref{estv2} and Schwarz inequality, we get
\begin{align}\label{I1}
\int_0^T\int_\Omega k_\varepsilon\nabla w\cdot \nabla  (v_\varepsilon^{(j)}\Phi )
= \int_0^T \left( \int_\Omega k_0\nabla w\cdot \nabla (v^{(j)}\Phi ) - \int_{\omega_\varepsilon} \widetilde k\nabla w\cdot \nabla \Phi v^{(j)}\right )
 +o(|\omega_\varepsilon|).
\end{align}
Let us consider now problem \eqref{PD}. Multiplying both sides of the first equation by $v_\varepsilon^{(j)}\Phi$, and integrating by parts
on $\Omega\times (0,T)$, we obtain
\begin{align}\nonumber
\int_0^T\int_\Omega w_tv_\varepsilon^{(j)}\Phi + \int_0^T\int_\Omega k_\varepsilon\nabla w \cdot \nabla (v_\varepsilon^{(j)}\Phi)
+ \int_0^T\int_\Omega\chi_{\Omega / \omega_\varepsilon} p_\varepsilon w v_\varepsilon^{(j)}\Phi\\ \label{I2}
= \int_0^T\int_{\omega_\varepsilon} \widetilde k \nabla u\cdot\nabla(v_\varepsilon^{(j)}\Phi) + \int_0^T\int_{\omega_\varepsilon} f(u)v_\varepsilon^{(j)}\Phi.
\end{align}
On the other hand, multiplying the first equation in \eqref{PD2} by $v^{(j)}\Phi$ and integrating by parts on $\Omega\times (0,T)$, we get
\begin{align}\nonumber
\int_0^T\int_\Omega w_t v^{(j)}\Phi + \int_0^T\int_\Omega k_0\nabla w \cdot\nabla(v^{(j)}\Phi)
+ \int_0^T\int_\Omega\chi_{\Omega / \omega_\varepsilon} p_\varepsilon w v^{(j)}\Phi\\ \label{I3}
=  \int_0^T\int_{\omega_\varepsilon} \widetilde k \nabla u^\varepsilon \cdot\nabla(v^{(j)}\Phi) + \int_0^T\int_{\omega_\varepsilon} f(u) v^{(j)}\Phi.
\end{align}
A combination of \eqref{I1}, \eqref{I2} and \eqref{I3} gives, for $\varepsilon \to 0$,
$$
\int_0^T\int_{\omega_\varepsilon} \widetilde k \nabla u\cdot\nabla(v_\varepsilon^{(j)}\Phi) + \int_0^T\int_{\omega_\varepsilon} f(u)v_\varepsilon^{(j)}\Phi
- \int_0^T\int_\Omega w_tv_\varepsilon^{(j)}\Phi
- \int_0^T\int_\Omega\chi_{\Omega / \omega_\varepsilon} p_\varepsilon w v_\varepsilon^{(j)}\Phi$$
$$=  \int_0^T\int_{\omega_\varepsilon} \widetilde k \nabla u^\varepsilon \cdot\nabla(v^{(j)}\Phi) + \int_0^T\int_{\omega_\varepsilon} f(u) v^{(j)}\Phi
- \int_0^T\int_\Omega w_t v^{(j)}\Phi
- \int_0^T\int_\Omega\chi_{\Omega / \omega_\varepsilon} p_\varepsilon w v^{(j)}\Phi$$
$$- \int_0^T\int_{\omega_\varepsilon} \tilde k\nabla w\cdot \nabla \Phi v^{(j)}
 +o(|\omega_\varepsilon|),
$$
from which we deduce
$$
\int_0^T\int_{\omega_\varepsilon} \widetilde k \nabla u\cdot\nabla(v_\varepsilon^{(j)}\Phi)
=  \int_0^T\int_\Omega w_t(v_\varepsilon^{(j)} - v^{(j)})\Phi
+ \widetilde k \int_0^T\int_{\omega_\varepsilon} \Big ( \nabla u^\varepsilon \cdot\nabla(v^{(j)}\Phi)
- \nabla u^\varepsilon \cdot \nabla \Phi v^{(j)}\Big )$$
$$+ \int_0^T\int_{\omega_\varepsilon} \widetilde k\nabla u\cdot \nabla \Phi v^{(j)}
- \int_0^T\int_\Omega\chi_{\Omega / \omega_\varepsilon} p_\varepsilon w (v_\varepsilon^{(j)} - v^{(j)})\Phi
+ \int_0^T\int_{\omega_\varepsilon} f(u) (v^{(j)} - v_\varepsilon^{(j)})\Phi
+ o(|\omega_\varepsilon|).
$$
By means of  \eqref{estw2}, \eqref{estw3}, \eqref{estv1} and \eqref{estv2}, and recalling also \eqref{M1} and \eqref{M3},
an application of the H\"{o}lder inequality both in space and time gives, for $\varepsilon \to 0$,
$$
\widetilde k \int_0^T\int_{\omega_\varepsilon}  \nabla u\cdot\nabla(v_\varepsilon^{(j)}\Phi)
=  \int_0^T\int_\Omega w_t(v_\varepsilon^{(j)} - v^{(j)})\Phi
+ \widetilde k \int_0^T\int_{\omega_\varepsilon} \big ( \nabla u^\varepsilon \cdot\nabla v^{(j)}\Phi
+ \nabla u \cdot\nabla\Phi v^{(j)}\big )
 +o(|\omega_\varepsilon|)
$$
and then
\begin{align}\nonumber
\widetilde k \int_0^T\int_{\omega_\varepsilon}  \nabla u\cdot\nabla v_\varepsilon^{(j)}\Phi
=  \int_0^T\int_\Omega w_t(v_\varepsilon^{(j)} - v^{(j)})\Phi \\ \nonumber
+ \widetilde k \int_0^T\int_{\omega_\varepsilon} \nabla u^\varepsilon \cdot\nabla v^{(j)}\Phi
+ \widetilde k \int_0^T\int_{\omega_\varepsilon} \nabla u \cdot\nabla\Phi (v^{(j)}- v_\varepsilon^{(j)})
+o(|\omega_\varepsilon|) \\ \label{I6}
=  \int_0^T\int_\Omega w_t(v_\varepsilon^{(j)} - v^{(j)})\Phi
+ \widetilde k \int_0^T\int_{\omega_\varepsilon} \nabla u^\varepsilon \cdot\nabla v^{(j)}\Phi
+o(|\omega_\varepsilon|).
\end{align}
Consider the first term in the last line of \eqref{I6}. Integrating by parts in time and recalling that
$\Phi(T) = 0$, $w(0) = 0$, $(v^{(j)}- v_\varepsilon^{(j)})_t = 0$, we finally have
(cf. also \eqref{estw3}, \eqref{estv2}), for $\varepsilon \to 0$,
\begin{eqnarray}\label{I7}
& \displaystyle \int_0^T\int_\Omega w_t(v_\varepsilon^{(j)} - v^{(j)})\Phi =
\int_\Omega [w(v_\varepsilon^{(j)} - v^{(j)})\Phi](T)
- \int_\Omega [w(v_\varepsilon^{(j)} - v^{(j)})\Phi](0) \\ \nonumber
&-  \displaystyle  \int_0^T\int_\Omega w(v_\varepsilon^{(j)} - v^{(j)})_t\Phi  -  \int_0^T\int_\Omega w(v_\varepsilon^{(j)} - v^{(j)})\Phi _t
=  -  \int_0^T\int_\Omega w(v_\varepsilon^{(j)} - v^{(j)})\Phi _t = o(|\omega_\varepsilon|).
\end{eqnarray}
Combining \eqref{I6} and \eqref{I7} we get
\begin{align}
\widetilde k \int_0^T\int_{\omega_\varepsilon}  \nabla u\cdot\nabla v_\varepsilon^{(j)}\Phi
=  \widetilde k \int_0^T\int_{\omega_\varepsilon} \nabla u^\varepsilon \cdot\nabla v^{(j)}\Phi
+o(|\omega_\varepsilon|), \quad \varepsilon \to 0,
\end{align}
then formula \eqref{I0} is true.
\end{proof}

{\it Proof of Theorem \ref{expansion}.}
Following \cite[Sec.3]{CV}, there exist a regular Borel measure $\mu$, a symmetric matrix $\mathcal{M}$ with elements
$\mathcal{M}_{ij} \in L^2(\Omega, d\mu)$, a sequence $\omega_{\varepsilon_n}$ with $|\omega_{\varepsilon_n}| \to 0$ such that
\begin{equation}\label{L1}
|\omega_{\varepsilon_n}|^{-1}\chi_{\omega_{\varepsilon_n}} dx \to d\mu,
\quad \quad |\omega_{\varepsilon_n}|^{-1}\chi_{\omega_{\varepsilon_n}} \frac{\partial v_{\varepsilon_n}^{(j)}}{\partial x_i} dx \to \mathcal{M}_{ij} d\mu,
\end{equation}
in the weak$^*$ topology of $C^0(\overline\Omega)$.
On account of  \eqref{bounduc2}, we deduce also
\begin{equation}\label{L2}
|\omega_{\varepsilon_n}|^{-1}\chi_{\omega_{\varepsilon_n}} \frac{\partial u(t)}{\partial x_i}\frac{\partial v_{\varepsilon_n}^{(j)}}{\partial x_i} dx
\to \mathcal{M}_{ij}\frac{\partial u(t)}{\partial x_i} d\mu, \quad \forall \, t \in (0,T),
\end{equation}
in the weak$^*$ topology of $C^0(\overline\Omega)$.
Moreover, recalling \eqref{bounduc2}, \eqref{estw3} and \eqref{vj}, we get
\begin{align}\label{L3}
\left | \int_0^T\int_\Omega\frac{\chi_{\omega_{\varepsilon_n}}}{|\omega_{\varepsilon_n}|} \frac{\partial u^{\varepsilon_n}}{\partial x_i}
\frac{\partial v^{(j)}}{\partial x_i} \right |
\leq C,
\end{align}
where $C$ is independent of $\varepsilon_n$.
Hence
\begin{equation}\label{L4}
|\omega_{\varepsilon_n}|^{-1}\chi_{\omega_{\varepsilon_n}} \frac{\partial u^{\varepsilon_n}}{\partial x_i} \frac{\partial v^{(j)}}{\partial x_i} dxdt \to d\nu_j
\end{equation}
in the weak$^*$ topology of $C^0(\overline\Omega \times [0,T])$.
Combining \eqref{I0}, \eqref{L2} and \eqref{L4} we obtain
\begin{equation}\label{L5}
d\nu_j = \mathcal{M}_{ij}\frac{\partial u(t)}{\partial x_i} d\mu, \quad \forall \, t \in (0,T).
\end{equation}

Now multiply the first equation in \eqref{(PPhi)} by $w$ and the first equation in \eqref{PD2} by $\Phi$ on $\Omega\times (0,T)$.


Then, integrating by parts, we get
$$
\int_0^T\int_\Omega\Phi_tw + \int_0^T\int_\Omega k_0\nabla \Phi \cdot \nabla w  - \int_0^T\int_\Omega f^\prime(u)\Phi w
+\int_0^T\int_{\partial\Omega} k_0\frac{\partial\Phi}{\partial n}w= 0,
$$
$$
\int_0^T\int_\Omega w_t\Phi  + \int_0^T\int_\Omega k_0\nabla w \cdot \nabla \Phi + \int_0^T\int_\Omega\chi_{\Omega / \omega_\varepsilon} p_\varepsilon w \Phi
=   \int_0^T\int_{\omega_\varepsilon} \widetilde k \nabla u^\varepsilon \cdot \nabla \Phi +  \int_0^T\int_{\omega_\varepsilon} f(u)\Phi.
$$
Summing up the two previous equations, we have
\begin{align}\nonumber
\int_0^T\int_\Omega (w_t\Phi + \Phi_tw) - \int_0^T\int_\Omega f^\prime(u)\Phi w
+\int_0^T\int_{\partial\Omega} k_0\frac{\partial\Phi}{\partial n}w + \int_0^T\int_\Omega\chi_{\Omega / \omega_\varepsilon} p_\varepsilon w \Phi
\\ \label{wPhi3}
=   \int_0^T\int_{\omega_\varepsilon} \widetilde k \nabla u^\varepsilon \cdot \nabla \Phi +  \int_0^T\int_{\omega_\varepsilon} f(u)\Phi.
\end{align}
Observe that the following identities hold
\begin{equation}
\int_0^T\int_\Omega (w_t\Phi + \Phi_tw) = \int_\Omega\left ( \Phi(T)w(T) - \Phi(0)w(0)\right ) - \int_0^T\int_\Omega \Phi w_t +   \int_0^T\int_\Omega \Phi w_t =0,
\end{equation}
and then, from \eqref{wPhi3} we infer
\begin{equation} \label{wPhi4}
\int_0^T\int_\Omega \big ( \chi_{\Omega / \omega_\varepsilon} p_\varepsilon w \Phi - f^\prime(u)\Phi w \big)
+\int_0^T\int_{\partial\Omega} k_0\frac{\partial\Phi}{\partial n}w
=   \int_0^T\int_{\omega_\varepsilon} \widetilde k \nabla u^\varepsilon \cdot  \nabla \Phi +  \int_0^T\int_{\omega_\varepsilon} f(u)\Phi.
\end{equation}
Moreover, on account of \eqref{estw3}, we have
\begin{align}\nonumber
\int_0^T\int_\Omega \big ( \chi_{\Omega / \omega_\varepsilon} p_\varepsilon w \Phi - f^\prime(u)\Phi w \big )
&=\int_0^T\int_\Omega\big ( \chi_{\Omega / \omega_\varepsilon} p_\varepsilon w \Phi - \chi_{\Omega / \omega_\varepsilon} f^\prime(u)\Phi w \big )
- \int_0^T\int_{\omega_\varepsilon} f^\prime(u)\Phi w\\ \label{wPhi5}
&= \int_0^T\int_\Omega\chi_{\Omega / \omega_\varepsilon} (p_\varepsilon -  f^\prime(u)) w \Phi + o(|\omega_\varepsilon|) = o(|\omega_\varepsilon|).
\end{align}
The last equality in \eqref{wPhi5} is a consequence of the regularity of $f$ (see \eqref{peps}, from which $|p_\varepsilon - f^\prime(u)| \leq C|w|$ follows)
and \eqref{estw3}.
Combining \eqref{wPhi4} and \eqref{wPhi5}
we obtain
$$
\int_0^T\int_{\partial\Omega} k_0\frac{\partial\Phi}{\partial n}w
= |\omega_\varepsilon|\int_0^T\int_\Omega \widetilde k |\omega_\varepsilon|^{-1}\chi_{\omega_{\varepsilon}}\nabla u^\varepsilon \cdot  \nabla \Phi
+ |\omega_\varepsilon| \int_0^T\int_\Omega \chi_{\omega_\varepsilon}|\omega_\varepsilon|^{-1} f(u)\Phi
+ o(|\omega_{\varepsilon_n}|).
$$
And finally, by means of \eqref{L1}, \eqref{L4} and \eqref{L5}, the formula \eqref{wPhi8} holds. \hfill $\Box$
\begin{remark}
We would like to emphasize that, with minor changes, the asymptotic expansion extends  to  the case of piecewise smooth anisotropic conductivities
of the form
\begin{equation}\label{anisotropy1}
\mathbb{K}_{\varepsilon} = \left\{
 \begin{array}{rl}
 \mathbb{K}_0 & \mbox{in } \ \Omega \setminus \omega_{\varepsilon} \\
 \mathbb{K}_1 & \mbox{in } \  \omega_{\varepsilon}
 \end{array}
 \right.
\end{equation}
where $\mathbb{K}_0, \mathbb{K}_1\in C^{\infty}(\Omega)$ are symmetric matrix valued functions satisfying
\[
\alpha_0|{\bf\xi}|^2\leq{\bf\xi}^T\mathbb{K}_0(x){\bf\xi}\leq \beta_0 |{\bf\xi}|^2, \quad
\alpha_1|\xi|^2\leq{\bf\xi}^T\mathbb{K}_1(x){\bf\xi}\leq \beta_1|{\bf\xi}|^2, \quad \forall \, {\bf\xi}\in\mathbb{R}^3, \forall \, x \in\Omega,
\]
with $0<\alpha_1<\beta_1< \alpha_0 < \beta_0$. Then, the asymptotic formula \eqref{wPhi8} becomes
\[
\int_0^T\int_{\partial\Omega} \mathbb{K}_0\nabla \Phi \cdot n (u^\varepsilon - u) =
|\omega_\varepsilon|\int_0^T\int_{\Omega}\Big( \mathcal{M}_{i\,j}(\mathbb{K}_0-\mathbb{K}_1)_{ik}\frac{\partial u}{\partial x_k}\frac{\partial \Phi}{\partial x_j}
+f(u)\Phi \Big )d\mu  +o(|\omega_\varepsilon|)\,
\]
where $\Phi$ solves
\begin{equation}\label{anisotropy2}
		\begin{cases}
			& \Phi_t + {\rm div} (\mathbb{K}_0 \nabla \Phi) - f^\prime(u)\Phi  = 0,\ \ \ {\rm in} \ \ \Omega\times (0,T),\\
			& \Phi(T) = 0, \ \ \ {\rm in} \ \ \Omega,
			\hskip2truecm
		\end{cases}
	\end{equation}
and $u$ is the background solution of
\begin{equation}\label{anisotropy3}
		\begin{cases}
			& u_t - {\rm div} (\mathbb{K}_0 \nabla u) + f(u)  = 0,\ \ \ {\rm in} \ \ \Omega\times (0,T),\\
            & \mathbb{K}_0\nabla u \cdot n = 0, \ \ \ {\rm on} \ \ \partial \Omega \times (0,T), \\
			& u(0) = 0, \ \ \ {\rm in} \ \ \Omega.
			\hskip2truecm
		\end{cases}
	\end{equation}
The matrix $\mathcal{M}$ is called  the polarization tensor associated to the inhomogeneity $\omega_\varepsilon$.
Indeed, all the results of the previous sections can be extended to the case of constant anisotropic coefficients using for instance
the regularity results contained in \cite{LSU}.
\end{remark}

\section{A reconstruction algorithm}

\setcounter{equation}{0}

We now use the asymptotic expansion derived in the previous section to set a reconstruction procedure
for the inverse problem of detecting a spherical inhomogeneity $\omega_{\varepsilon}$ from boundary measurements of the potential.
Following the approach of \cite{BerettaManzoniRatti},
\cite{art:cmm}, but taking into account the time-dependence of the problem, we introduce the mismatch functional
\begin{equation}
J(\omega_\varepsilon) = \frac{1}{2} \int_0^T \int_{\partial \Omega}(u^\varepsilon - u_{meas})^2 ,
\label{eq:J}
\end{equation}
being $u^\varepsilon$ the solution of the perturbed problem \eqref{perturbed} in presence of an inclusion $\omega_{\varepsilon}$ satisfying hypotheses
\eqref{dist}, \eqref{smalldim}.
It is possible to reformulate the inverse problem in terms of the following minimization problem
\begin{equation}
	J(\omega_\varepsilon) \rightarrow \min
\label{eq:minprob}
\end{equation}
among all the small inclusions, well separated from the boundary. We introduce the following additional assumption on the exact inclusion
\begin{equation}
\omega_\varepsilon = z + \varepsilon B = \{x \in \Omega \text{ s.t. } x = z+\varepsilon b,\ b \in B\},
\label{eq:omegaez}
\end{equation}
being $z \in \Omega$ and $B$ an open, bounded, regular set containing the origin. We remark that we prescribe the geometry of the inclusion to be fixed
throughout the whole observation time.
The restriction of the functional $J$ to the class of inclusions satisfying \eqref{eq:omegaez} is denoted by $j(\varepsilon;z)$.
We can now define the \textit{topological gradient} $G: \Omega \rightarrow \R$ as the first order term appearing in the asymptotic expansion of the cost
functional with respect to $\varepsilon$, namely
\[
	j(\varepsilon;z) = j(0) + |\omega_\varepsilon| G(z) + o(|\omega_\varepsilon|), \quad \varepsilon \rightarrow 0,
\]
where $j(0) = \int_0^T \int_{\partial \Omega}(u - u_{meas})^2 $ and $u$ is the solution of the unperturbed problem \eqref{(P_0)}.
Under the assumption that the exact inclusion has small size and satisfies hypothesis
\eqref{eq:omegaez}, a reconstruction procedure consists in identifying the point $\bar{z} \in \Omega$ where the topological gradient $G$ attaints its minimum.
Indeed, the cost functional achieves the smallest value when it is evaluated in the center of the exact inclusion. Thanks to the hypothesis of small size, we expect the reduction of the cost functional $j$ to be correctly described by the first order term $G$, up to a reminder which is negligible with respect to $\varepsilon$.
\par
Nevertheless, in order to define a reconstruction algorithm, we need to efficiently evaluate the topological gradient $G$. According to the definition,
\[
	G(z) = \lim_{\varepsilon \rightarrow 0} \frac{j(\varepsilon;z)-j(0)}{|\omega_\varepsilon|}.
\]
Evaluating $G$ in a single point $z\in \Omega$ would require to solve the direct problem several times in presence of inclusions centered at $z$
with decreasing volume.
This procedure can be indeed avoided thanks to a useful \textit{representation formula} that can be deduced from the asymptotic expansion \eqref{wPhi8}.
To show this we need the following preliminary Proposition the proof of which is given in the Appendix.
\begin{proposition}\label{boundaryNorm}
	Consider the problem
	\begin{equation}\label{IR1}
		\begin{cases}
			& \Phi_t + k_0\Delta \Phi - f^\prime(u)\Phi  = 0,\ \ \ {\rm in} \ \ \Omega\times (0,T),\\
			& \displaystyle\frac{\partial \Phi}{\partial n}  = u^\varepsilon - u, \ \ \  {\rm on } \ \ \partial\Omega \times (0,T), \\
			& \Phi(T) = 0, \ \ \ {\rm in} \ \ \Omega.
			\hskip2truecm
		\end{cases}
	\end{equation}
	Given a compact set $K \subset \Omega$ such that $d(K,\partial \Omega) \geq d_0 > 0$ the following estimate holds
	\begin{equation}
		\|\Phi\|_{L^1(0,T;W^{1,\infty}(K))} \leq C \|u^\varepsilon - u \|_{L^2(0,T;L^2(\partial \Omega))}.
	\label{eq:infEst}
	\end{equation}
\end{proposition}
On account of Proposition \ref{boundaryNorm}, we deduce the following representation of the topological gradient
\begin{proposition}
The topological gradient of the cost functional $j(\varepsilon, z)$ can be expressed by
\begin{equation}
	G(z) = \int_0^T  \left( \widetilde k \mathcal{M} \nabla u(z) \cdot \nabla W(z)  +  f(u(z))W(z) \right),
\label{eq:repForm}
\end{equation}
where $W$ is the solution of the \textit{adjoint problem}:
\begin{equation}
\begin{cases}
& W_t + k_0\Delta W  - f^\prime(u)W  = 0,\ \ \ {\rm in} \ \ \Omega\times (0,T),\\
& \displaystyle \frac{\partial W}{\partial n} = u - u_{meas}, \ \ \ {\rm on} \ \ \partial \Omega \times (0,T), \\
& W(T) = 0, \ \ \ {\rm in} \ \ \Omega.
\end{cases}
\label{eq:adjoint}
\end{equation}
	\label{prop:repForm}
\end{proposition}
\begin{proof}
Consider the difference
\begin{equation}
\begin{aligned}
j(\varepsilon;z) - j(0)
&= \frac{1}{2} \|u^\varepsilon - u_{meas} \|_{L^2(0,T;L^2(\partial \Omega))}^2 - \frac{1}{2} \|u - u_{meas} \|_{L^2(0,T;L^2(\partial \Omega))}^2 \\
& =  \int_0^T \int_{\partial \Omega} (u^\varepsilon - u)(u - u_{meas})dt  + \frac{1}{2} \|u^\varepsilon - u \|_{L^2(0,T;L^2(\partial \Omega))}^2.
\end{aligned}
\label{eq:expansion}
\end{equation}
According to \eqref{wPhi8} and to the definition of the adjoint problem \eqref{eq:adjoint}, we can express
\[
	\int_0^T \int_{\partial \Omega} (u^\varepsilon - u)(u - u_{meas})dt =
|\omega_\varepsilon| \left\{ \int_0^T\int_\Omega \widetilde k \mathcal{M} \nabla u\cdot  \nabla W d\mu
+  \int_0^T\int_\Omega  f(u)W d\mu + o(1) \right\}.
\]
Since we assume \eqref{eq:omegaez}, the measure $\mu$ associated to the inclusion is the Dirac mass $\delta_z$ centered in point $z$ (see \cite{CV}).
Hence
\begin{equation}
\begin{aligned}
	\int_0^T \int_{\partial \Omega} (u^\varepsilon - u)(u - u_{meas})dt = &
|\omega_\varepsilon| \left\{ \int_0^T \widetilde k \mathcal{M} \nabla u(z)\cdot  \nabla W(z)  \right. \\
& + \left. \int_0^T f(u(z))W(z) \right\} + o(|\omega_\varepsilon|).
\end{aligned}
\label{eq:part1}
\end{equation}
Moreover, by \eqref{wPhi8}, the second term in the left-hand side of \eqref{eq:expansion} can be expressed as
\[
	\int_0^T \int_{\partial \Omega} (u^\varepsilon - u)(u^\varepsilon - u)dt =
|\omega_\varepsilon| \left\{ \int_0^T\widetilde k \mathcal{M} \nabla u(z)\cdot  \nabla \Phi(z) +  \int_0^T f(u(z))\Phi(z) \right\} + o(|\omega_\varepsilon|),
\]
where $\Phi$ is the solution to \eqref{IR1}. Thanks to regularity results on $u$ (see Theorem \ref{regu}) and using Proposition \ref{boundaryNorm} with
$K = \Omega_{d_0} = \{x \in \Omega \textit{ s.t. }d(x,\partial \Omega) \geq d_0\}$, we obtain
\begin{equation}
\begin{aligned}
&\int_0^T \int_{\partial \Omega} (u^\varepsilon - u)(u^\varepsilon - u)dt \leq C |\omega_\varepsilon| \left\{ \int_0^T | \nabla \Phi(z)|
+  \int_0^T |\Phi(z)| \right\} + o(|\omega_\varepsilon|) \\
& \leq C|\omega_\varepsilon| \|u^\varepsilon - u \|_{L^2(0,T,L^2(\partial \Omega))} + o(|\omega_\varepsilon|)
 \leq C|\omega_\varepsilon| \|u^\varepsilon - u \|_{L^2(0,T,H^1(\Omega))} + o(|\omega_\varepsilon|) \\
&\leq C|\omega_\varepsilon|^{\frac{3}{2}} + o(|\omega_\varepsilon|) = o(|\omega_\varepsilon|).
\end{aligned}
\label{eq:part2}
\end{equation}
Replacing \eqref{eq:part1} and \eqref{eq:part2} in \eqref{eq:expansion}, we finally get
\[
	j(\varepsilon;z) - j(0)
= |\omega_\varepsilon| \left\{ \int_0^T \widetilde k \mathcal{M} \nabla u(z)\cdot  \nabla W(z) +  \int_0^T f(u(z))W(z) \right\} + o(|\omega_\varepsilon|).
\]
\end{proof}
Thanks to the representation formula \eqref{eq:repForm}, evaluating the topological gradient of the cost functional requires just the solution
of two initial and boundary value problems. This yields the definition of
a \textit{one-shot algorithm} for the identification of the center of a small inclusion satisfying hypotesis \eqref{eq:omegaez} (see Algorithm \ref{al:top}).

\begin{algorithm}[h!]
\begin{algorithmic}
	\REQUIRE $u_0(x,0) \, \forall \, x \in \Omega$, $u_{meas}(x,t) \ \forall \, x \in \partial \Omega$, $t \in (0,T)$.
	\ENSURE approximated centre of the inclusion, $\bar{z}$
	\begin{enumerate}
		\item compute $u$ by solving $(P_0)$;
		\item compute $W$ by solving (PA);
		\item determine $G$ according to \eqref{eq:repForm};
		\item find $\bar{z}$ s.t. $G(\bar{z}) \leq G(z) \quad \forall \, z \in \Omega$.
	\end{enumerate}
 \end{algorithmic}
\caption{Reconstruction of a single inclusion of small dimensions}
\label{al:top}
\end{algorithm}

Inspired by the electrophysiological application, we consider moreover the possibility to have {\em partial boundary measurements}, i.e.
the case where the support of the function $u_{meas}$ is not the whole
boundary $\partial \Omega$ but only a subset $\Gamma \subset \partial\Omega$. In this case, it is possible to formulate a slightly different optimization problem,
in which we aim at minimizing the mismatch between the measured and the perturbed data just on the portion $\Gamma$ of the boundary. The same reconstruction algorithm can be devised
for this problem, by simply changing the expression of the Neumann condition of the adjoint problem.

\section{Numerical results}
\setcounter{equation}{0}

In order to implement Algorithm 1 for the detection of inclusions, it is necessary to approximate the solution of the background problem \eqref{(P_0)} and
the adjoint problem \eqref{eq:adjoint}. Moreover, when considering synthetic data $u_{meas}$, we must be able to compute the solution to the perturbed problem
\eqref{perturbed} in presence of the exact inclusion. We rely on the Galerkin finite element method for the numerical approximation of these problems.
The one-shot procedure makes the reconstruction algorithm very efficient, only requiring the solution of an adjoint problem for each acquired measurement over
the time interval, without entailing any iterative (e.g. descent) method for numerical optimization.

\subsection{Finite Element approximation of  initial and boundary value problems}

The background problem \eqref{(P_0)} can be cast in weak form as follows

\textit{$\forall \ t\in(0,T)$, find $ u(t) \in V = H^1(\Omega)$ such that $u(0) = u_0$ and}
\begin{equation}
	\int_{\Omega} u_t(t)v + \int_{\Omega} k_1 \nabla u(t)\cdot \nabla v + \int_{\Omega}f(u(t))v = 0, \qquad \forall \, v \in V.
\label{eq:weakU}
\end{equation}
By introducing a finite-dimensional subspace $V_h$ of $V$, $dim(V_h) = N_h < \infty$, the Galerkin (semi-discretized in space) formulation of problem
\eqref{eq:weakU} reads

\textit{$\forall \, t\in(0,T)$, find $u_h(t) \in V_h$ such that $u_h(0) = u_{h,0}$ and}
\begin{equation}
	((u_h)_t(t),v_h) + b(u_h(t),v_h) + F(u_h(t),v_h) = 0, \qquad \forall \, v_h \in V_h,
\label{eq:semidiscreteU}
\end{equation}
where $(\cdot,\cdot)$ is the inner product in $L^2(\Omega)$, $b(u,v) = \int_{\Omega}k_1 \nabla u \cdot \nabla v$, $F(u,v)=\int_{\Omega} f(u)v$, $f$ is defined as in \eqref{H2} and $u_{h,0}$ is the $H^1$-projection of $u_0$ onto $V_h$.

To obtain a full discretization of the problem, we introduce a finite difference approximation in time. According to the strategy reported in \cite{Pavarino2014book}, \cite{colli2004parallel},
we rely on a semi-implicit scheme which allows an efficient treatment of the nonlinear terms.
Let us consider an uniform partition $\{t^n\}_{n = 0}^N$ of the time interval $[0,T]$ of step $\tau=\frac{T}{N}$ s.t. $t^0 = 0, \ t^N=T$. Then, the fully discrete
formulation of \eqref{(P_0)} is given by

\textit{$\forall \,n = 0, \ldots N-1$, find $ u_h^{n+1} \in V_h$ such that $u_h^0 = u_{0,h}$ and}
\begin{equation}
	(u^{n+1}_h,v_h) - (u^n_h,v_h) + \tau b(u^{n+1}_h,v_h) + \tau F(u^n_h,v_h) = 0, \qquad \forall \, v_h \in V_h.
\label{eq:discreteU}
\end{equation}
\par With the same discretization strategy one may describe a numerical scheme for the approximate solution of the perturbed problem, using the weak form reported in \eqref{wepsilon} and introducing the forms
\[
	b_\varepsilon(u,v) = \int_{\Omega} k_\varepsilon \nabla u \cdot \nabla v, \qquad
	F_\varepsilon(u,v) = \int_{\Omega} \chi_{\Omega \setminus \omega_\varepsilon} f(u)v.
\]
The adjoint problem, instead, requires the introduction of the form $dF(u,v;w) = \int_{\Omega} f'(w)uv$, which is bilinear with respect to $u$ and $v$. Thanks
to the linearity of the adjoint problem, we can consider a fully implicit Crank-Nicolson scheme

\textit{$\forall \, n = 0, \ldots N-1$, find $ w_h^{n} \in V = H^1(\Omega)$ such that $w^N_h = 0$ and}
{
\begin{equation}
\begin{aligned}
	&(w^{n+1}_h,v_h) - (w^n_h,v_h) +\frac{\tau}{2} \left( b(w^{n+1}_h,v_h) + b(w^{n}_h,v_h) + \right. \\
	&\left. dF(w^{n+1}_h,v_h;u^{n+1}_h) + dF(w^n_h,v_h;u^{n}_h) \right)= \\
	&\frac{\tau}{2} \left( \int_{\partial \Omega} (u_h^{n+1} - u_{meas}(t^{n+1}))v_h + \int_{\partial \Omega} (u_h^n - u_{meas}(t^n))v_h  \right), \quad \forall \, v_h \in V_h.
\label{eq:discreteW}
\end{aligned}
\end{equation}
}
The existence and the uniqueness of the solutions of the fully-discrete problems \eqref{eq:discreteU} and \eqref{eq:discreteW} follow by the well-posedness of the continuous problems, since $V_h$ is a subspace of $H^1(\Omega)$.
For further details on the stability and of the convergence of the proposed schemes we refer to \cite{fernandez2010decoupled}, \cite{sanfelici2002convergence} and \cite{Pavarino2014book}.
\par
The numerical setup for the simulation is represented in Figure \ref{fig:setup}. We consider an idealized geometry of the left ventricle (which has been object of several studies, see e.g. \cite{Pavarino2014book},
\cite{colli2004parallel}), and define a tetrahedral tesselation $\mathcal{T}_h$ of the domain. The discrete space $V_h$ is the P1-Finite Element space over $\mathcal{T}_h$, i.e. the space of the continuous functions
over $\Omega$ which are linear polynomials when restricted on each element $T \in \mathcal{T}_h $. The mesh we use for all the reported results consists of 24924 tetrahedric elements and $N_h=5639$ nodes.
We report also the anisotropic structure considered in all the recontruction tests, according to \cite{rossi2014thermodynamically} and \cite{Pavarino2014book}. The conductivity matrix $\mathbb{K}_0$ for the monodomain equation is given by $\mathbb{K}_0(x) = \mathbb{K}^e(x)(\mathbb{K}^e(x)+ \mathbb{K}^i(x))^{-1} \mathbb{K}^i(x)$, where both $\mathbb{K}^i$ and $\mathbb{K}^e$ are orthotropic tensors with three constant positive real eigenvalues, namely
\[
\begin{aligned}
\mathbb{K}^e(x) &= k_{f}^e \vec{e_f}(x)\otimes\vec{e_f}(x) + k_{t}^e \vec{e_t}(x)\otimes\vec{e_t}(x) + k_{r}^e \vec{e_r}(x)\otimes\vec{e_r}(x)
\\
\mathbb{K}^i(x) &= k_{f}^i \vec{e_f}(x)\otimes\vec{e_f}(x) + k_{t}^i \vec{e_t}(x)\otimes\vec{e_t}(x) + k_{r}^i \vec{e_r}(x)\otimes\vec{e_r}(x)
\end{aligned}
\]
The eigenvectors $\vec{e_f}$, $\vec{e_t}$ and $\vec{e_r}$ are associated to the three principal directions of conductivity in the heart tissue: respectively, the fiber centerline, the tangent direction to the heart sheets and the transmural direction (normal to the sheets).
\begin{figure}[h!]
			\centering \vspace{-0.5cm}
			\subfloat[Domain]{
		    	\includegraphics[width=0.3\textwidth]{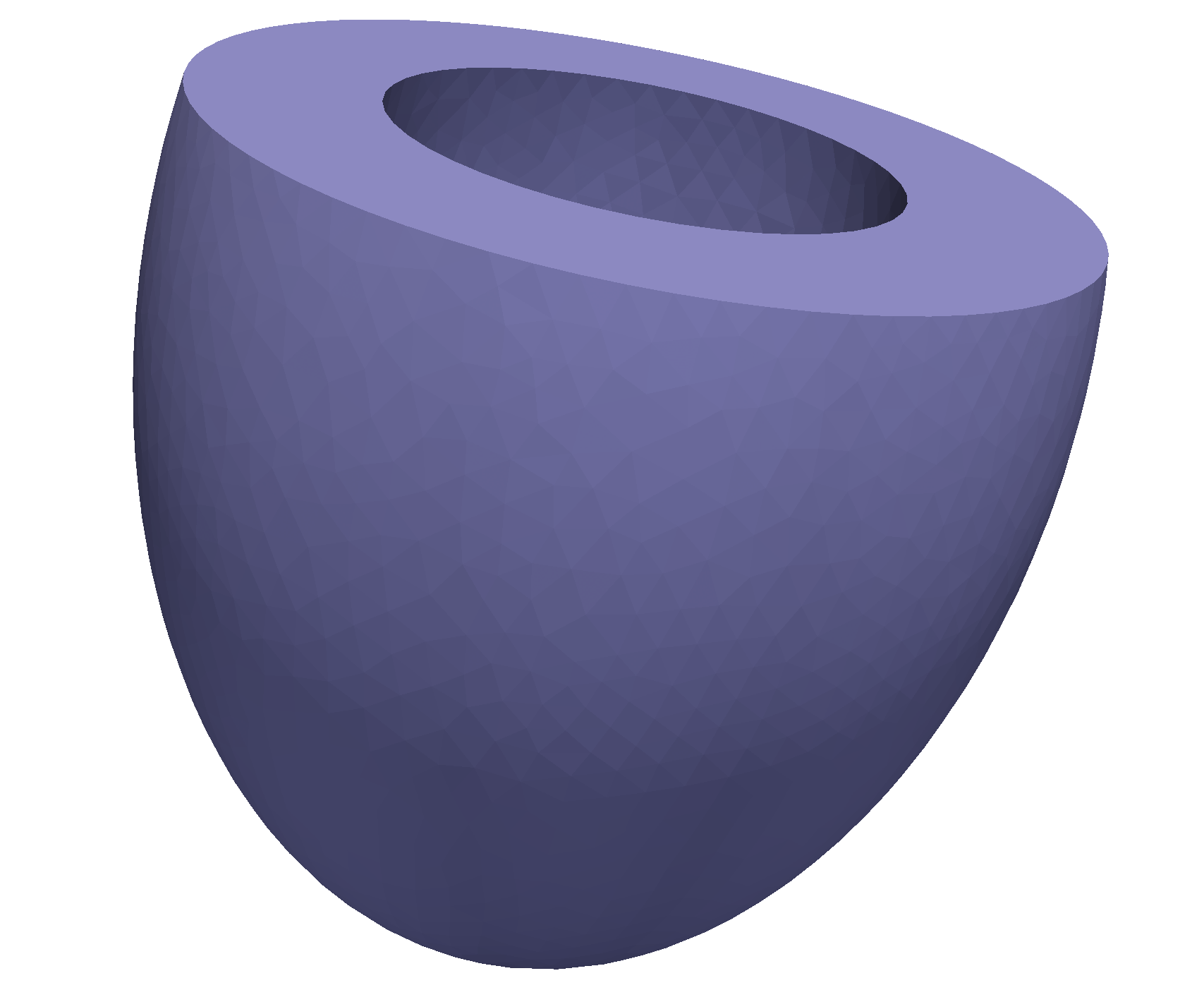}
			}
			\subfloat[Mesh (section)]{
		    	\includegraphics[width=0.3\textwidth]{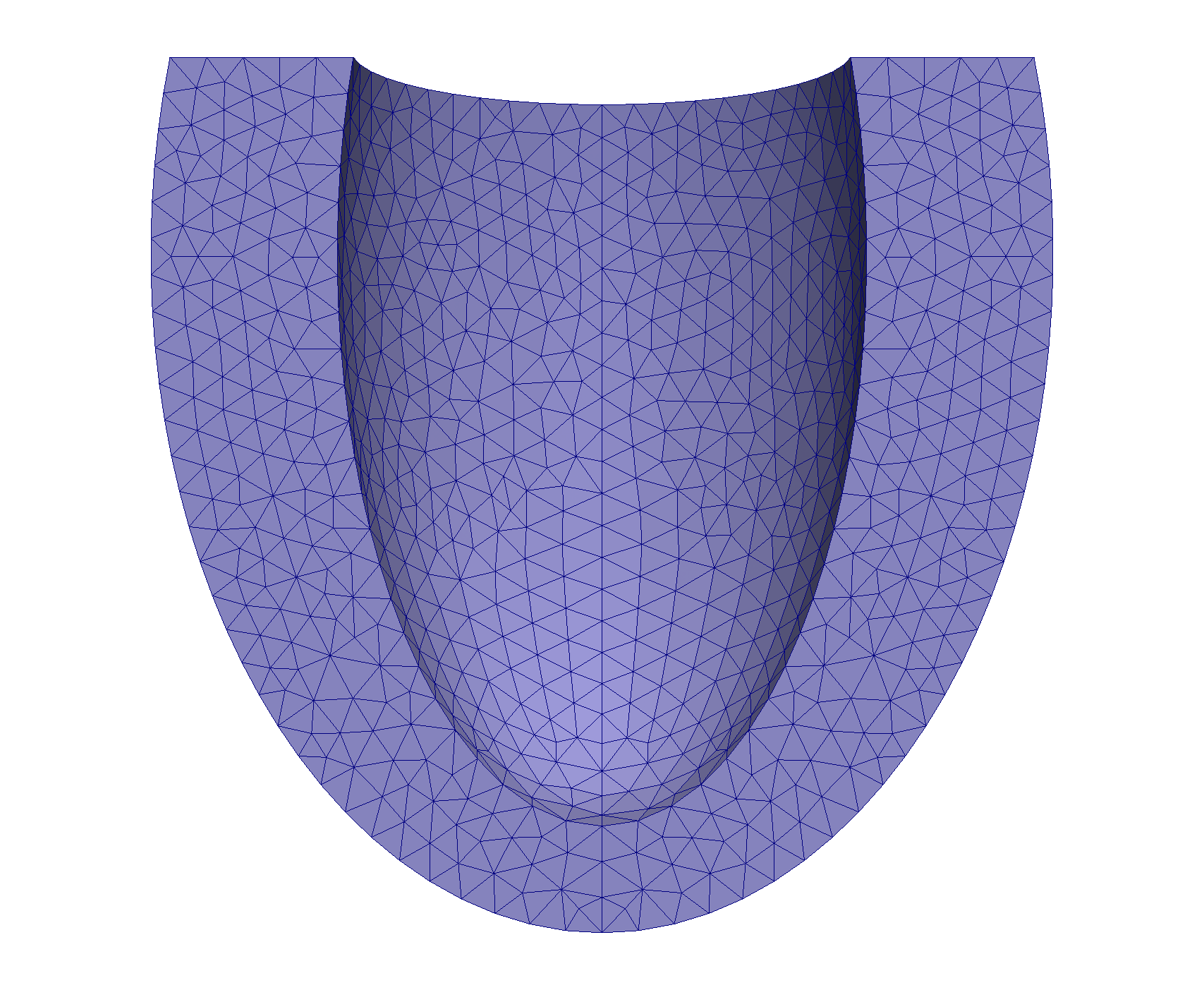}
			}
			\subfloat[Fiber directions]{
		    	\includegraphics[width=0.3\textwidth]{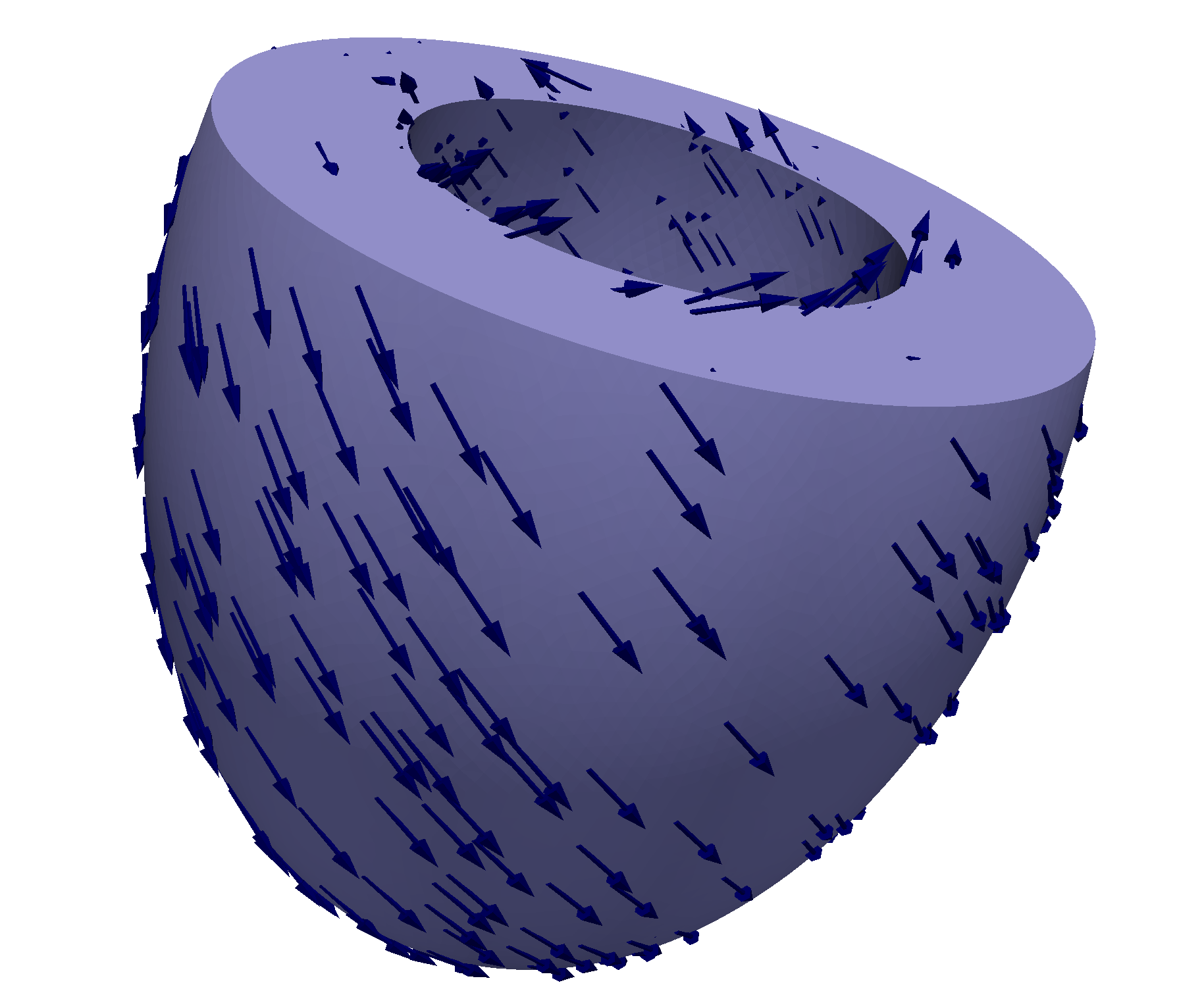}
			}
	\caption{Setup of numerical test cases}
	\label{fig:setup}
	\end{figure}
	\par
For the direct problem simulations, we consider the formulation reported in \eqref{(P_0)}, specifying realistic values for the parameters $C_m$ and $\nu$. We have rescaled the values of the coefficients $u_1$, $u_2$, $u_3$ and $A^2$ in order to simulate the electric potential in the adimensional range $[0,1]$. The rescaling is performed by the transformation $\tilde{u} = (\alpha+u)/\beta$, where $\alpha = 0.085 mV$ and $\beta = 0.125 mV$, whereas for the sake of simplicity we will still denote by u the rescaled variable $\tilde{u}$. We consider the initial datum $u_0$ to be positive on a band of the endocardium, representing the initial stimulus provided by the heart conducting system. The most important parameters, considered in accordance with \cite{gerbeau2015reduced}, \cite{SLCNMT}, are reported in Table \ref{tab:coef}.

\begin{table}[h!]
\centering
\begin{tabular}[t]{|c|c|c|c|c|c|c|c|c|c|c|c|}
\hline
$\nu$ & $C_m$ & $A^2$ & $u_1$ & $u_2$ & $u_3$ & $k_{f}^i$ & $k_{t}^i$ & $k_{r}^i$ & $k_{f}^e$ & $k_{t}^e$ & $k_{r}^e$\\
\hline
$500 \frac{m}{A}$ & $0.1\frac{mA\ ms}{cm^2}$ & 0.2 & 0 & 0.15 & 1 & 3 & 1 & 0.315 & 2 & 1.65 & 1.351\\
\hline
\end{tabular}
\caption{Physical coefficients}
\label{tab:coef}
\end{table}

In Figure \ref{fig:back} we report the solution of the discrete background problem \eqref{eq:discreteU} at different time instants, comparing the isotropic and the anisotropic cases.
\begin{figure}[t!]
			\centering \vspace{-0.5cm}
			\subfloat[Isotropic case, t = 0.2 T]{
		    	\includegraphics[width=0.3\textwidth]{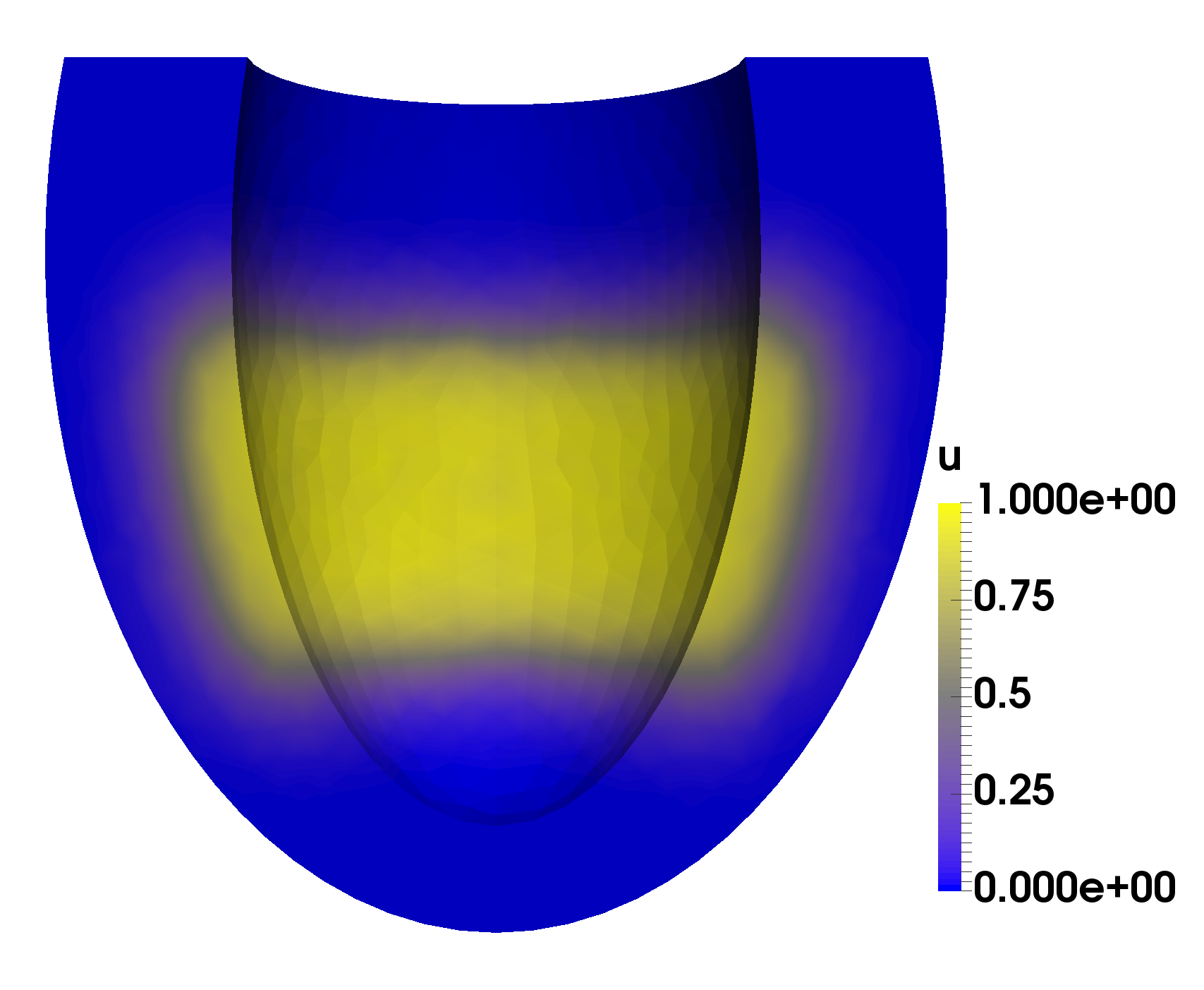}
			}
			\subfloat[Isotropic case, t = 0.5 T]{
		    	\includegraphics[width=0.3\textwidth]{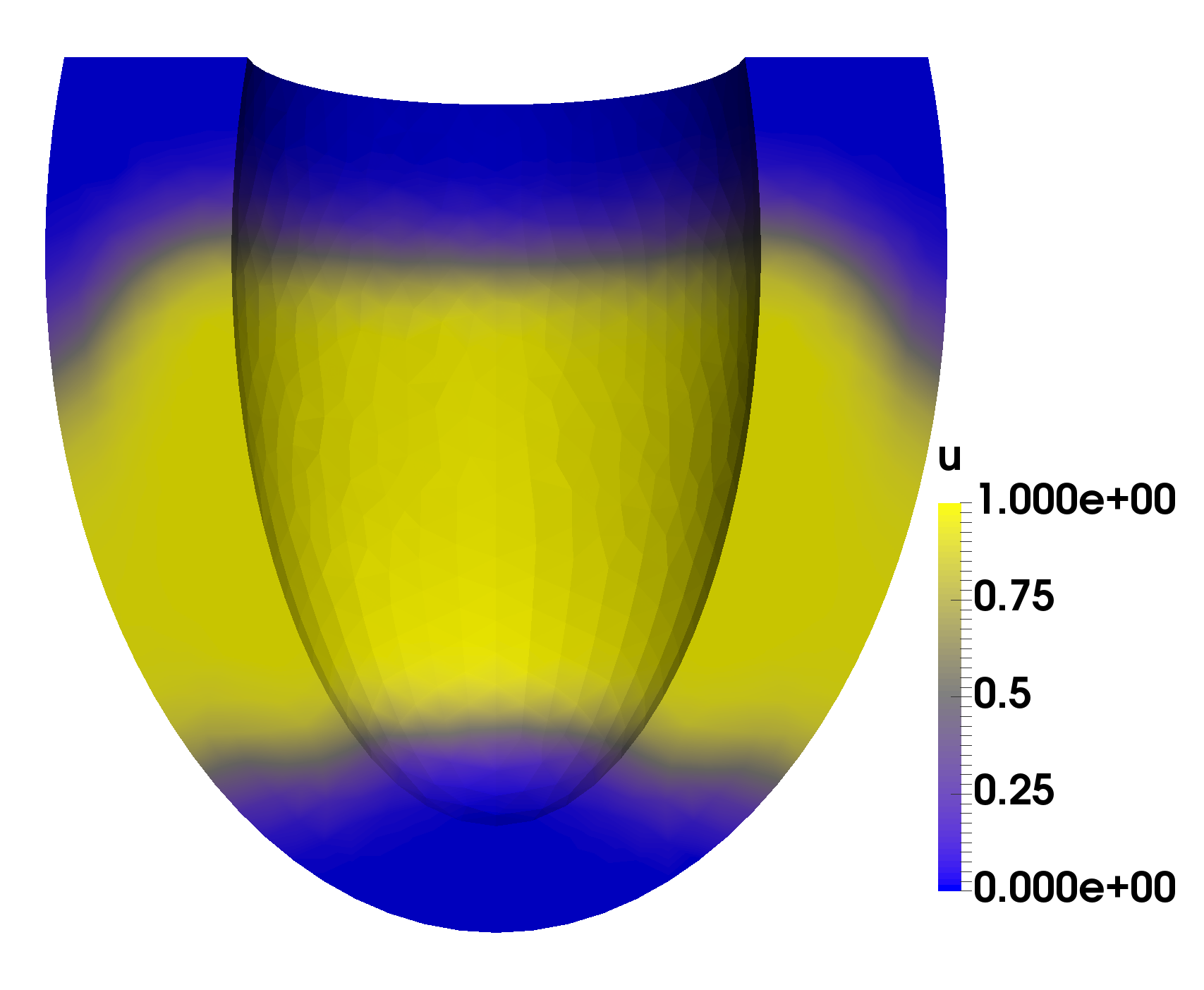}
			}
			\subfloat[Isotropic case, t = 0.8 T]{
		    	\includegraphics[width=0.3\textwidth]{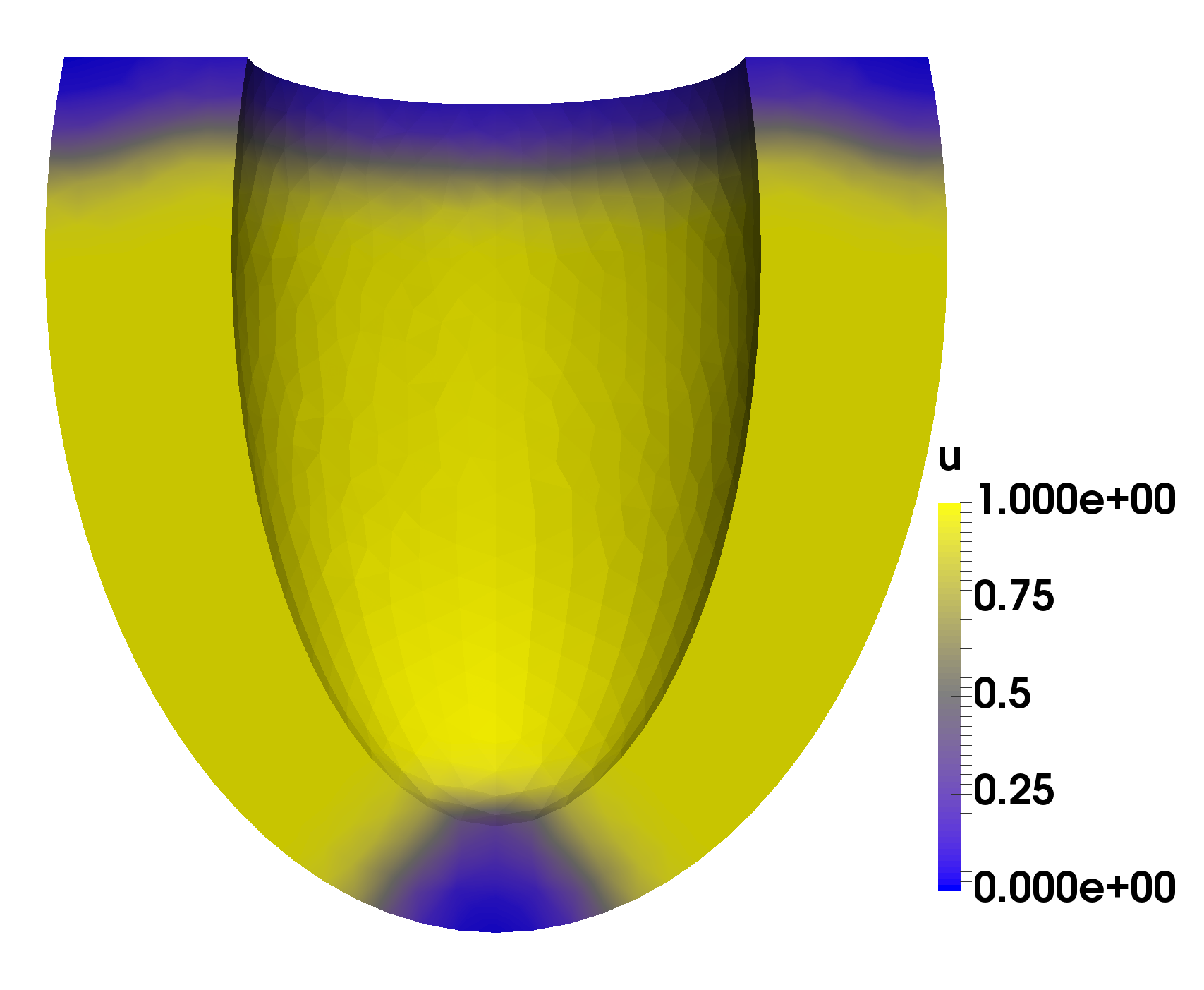}
			}
			\\
			\subfloat[Anisotropic case, t = 0.2 T]{
		    	\includegraphics[width=0.3\textwidth]{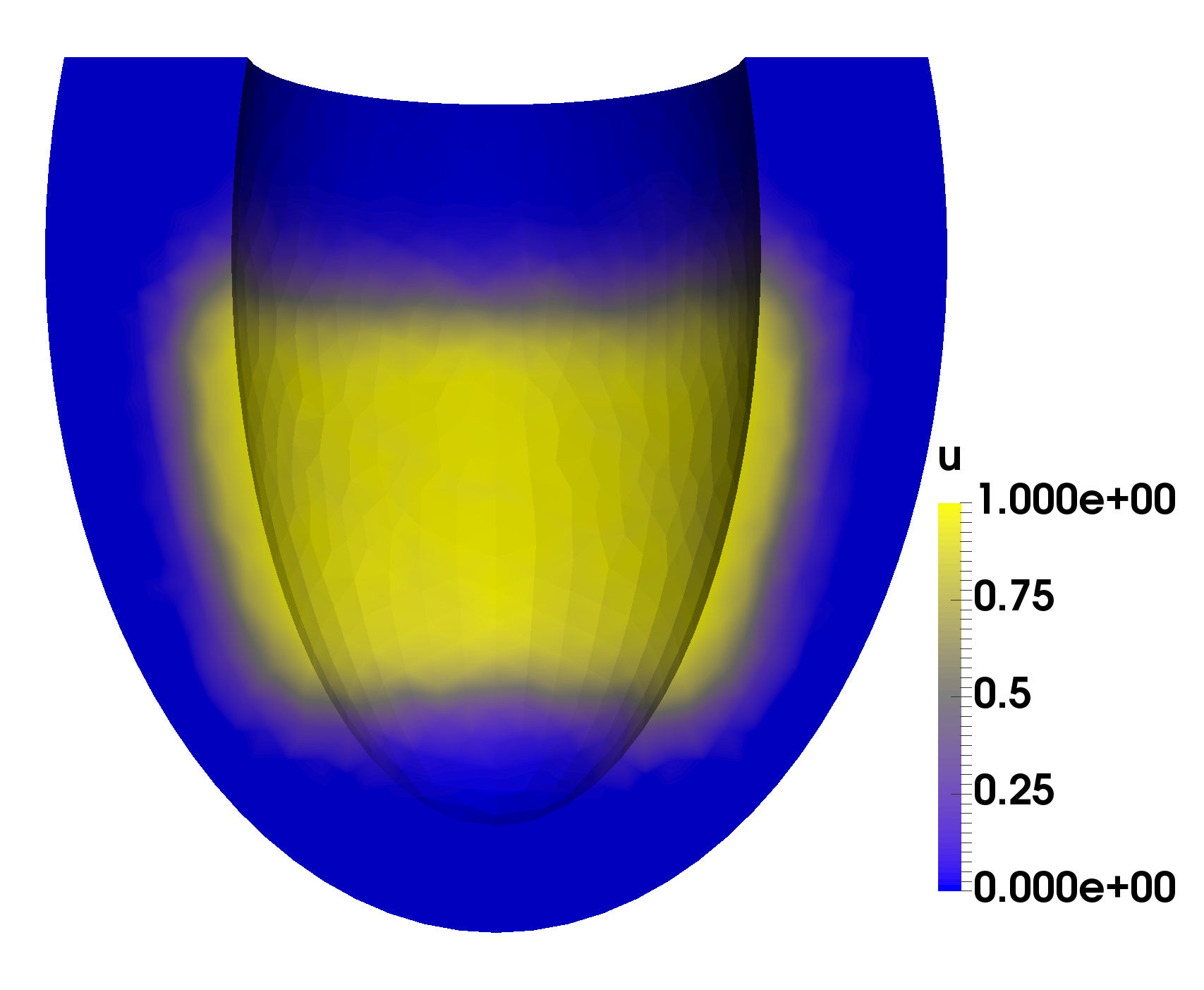}
			}
			\subfloat[Anisotropic case, t = 0.5 T]{
		    	\includegraphics[width=0.3\textwidth]{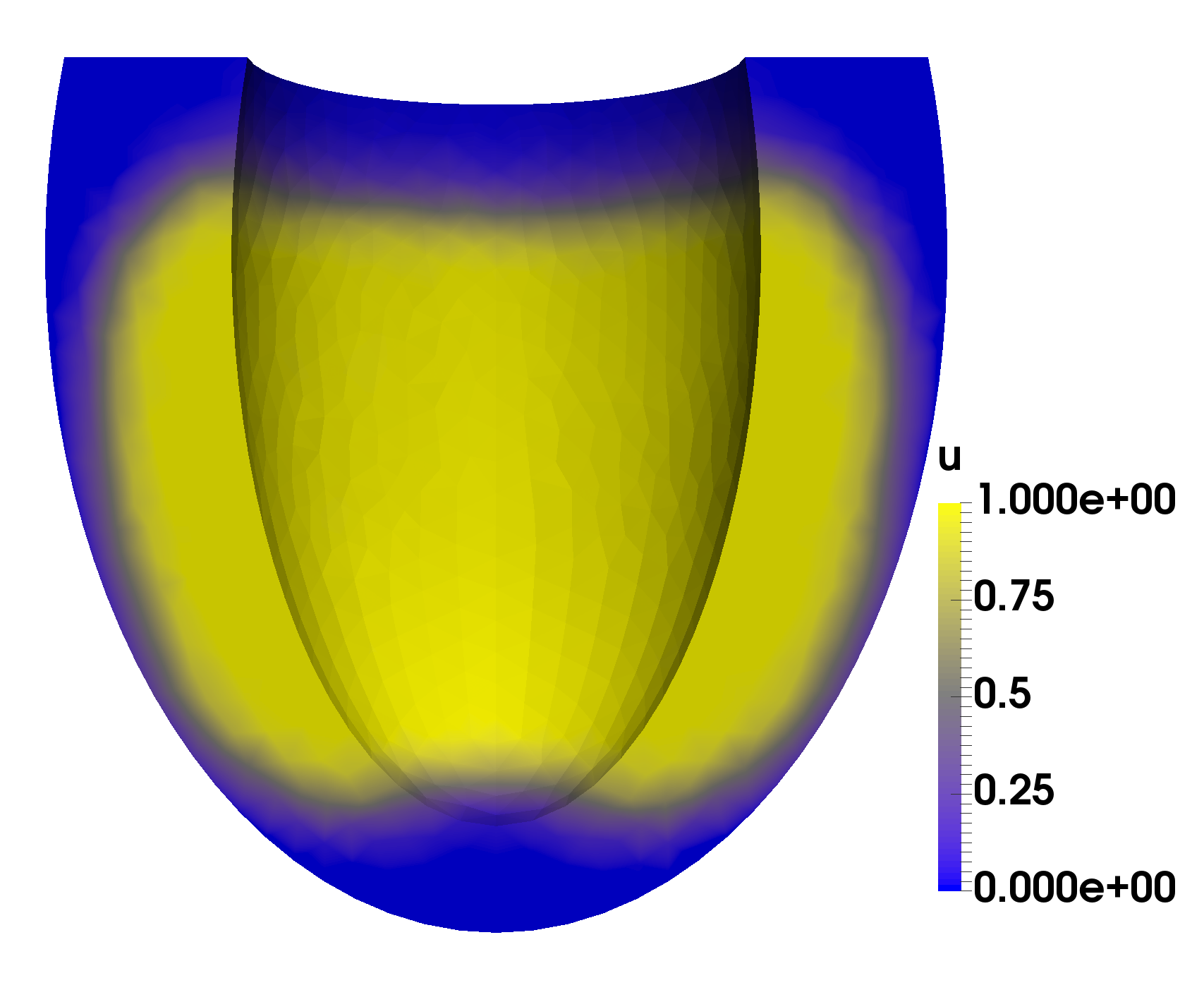}
			}
			\subfloat[Anisotropic case, t = 0.8 T]{
		    	\includegraphics[width=0.3\textwidth]{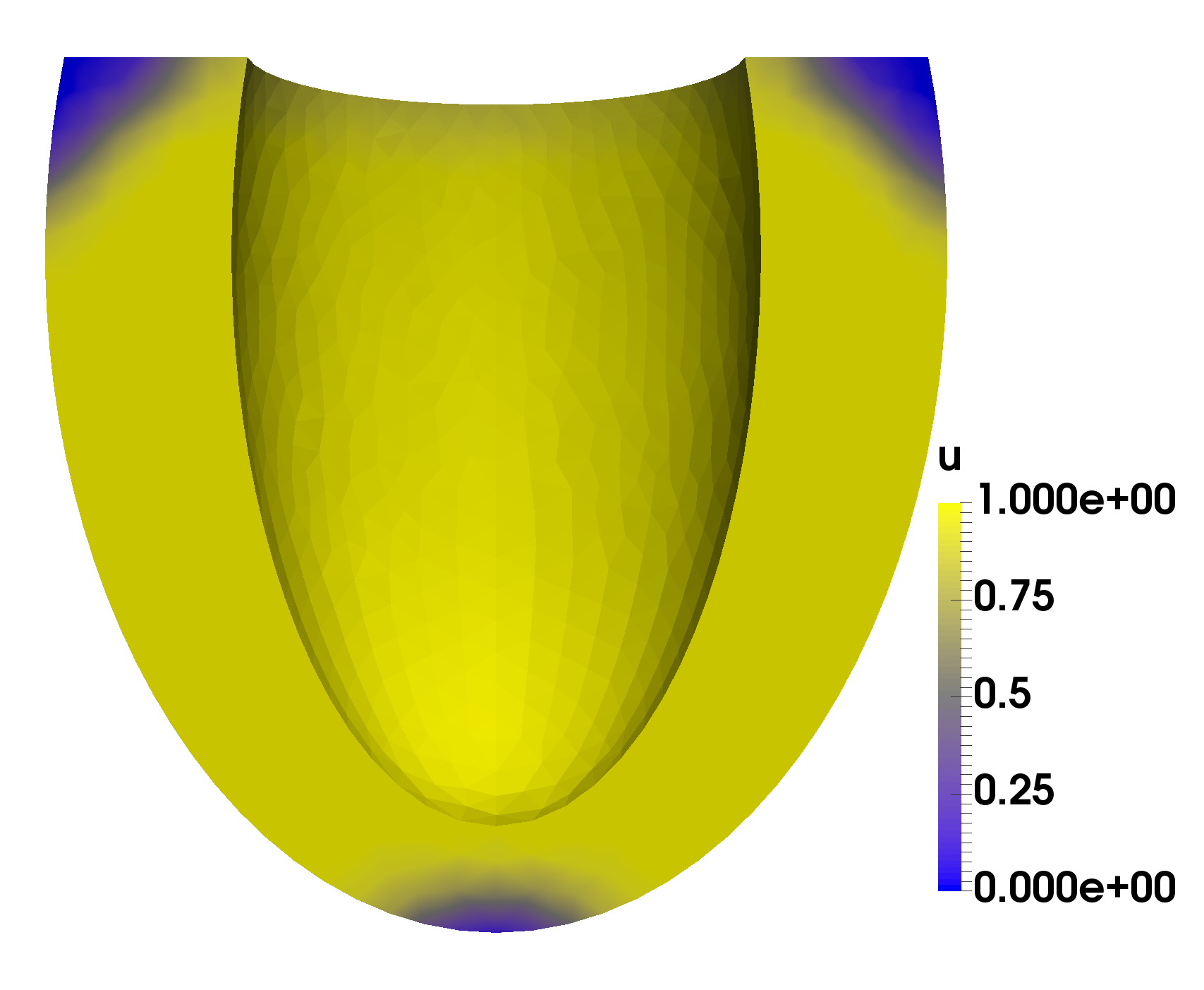}
			}
	\caption{Background problem simulation}
	\label{fig:back}
	\end{figure}

\subsection{Reconstruction of small inclusions}
We now tackle the problem of reconstructing the position of a small inhomogeneity using the knowledge of the electric potential of the tissue on a portion $\Gamma$
of the boundary of the domain. In particular, we assume that $u_{meas}$ is known on the endocardium, i.e. the inner surface of the heart cavity. We generate
synthetic data on a more refined mesh and test the effectiveness of Algorithm 1 in the reconstruction of a small spherical inclusion in different positions.
In Figure \ref{fig:reconstructions} 
we report the value assumed by the topological gradient, and superimpose the exact inclusions: we observe a negative region in proximity of the position of the real inclusion.
The algorithm precisely identifies the region where the inclusion is present, whereas the minimum may in general be found along the endocardium also when the center of the real inclusion is not located on the heart surface.
Nevertheless, due to the thinness of the domain the reconstructed position is close to the real one.
\begin{figure}[b!]
			 \vspace{-0.1cm}
			\centerline{
 		    	\includegraphics[width=0.3\textwidth]{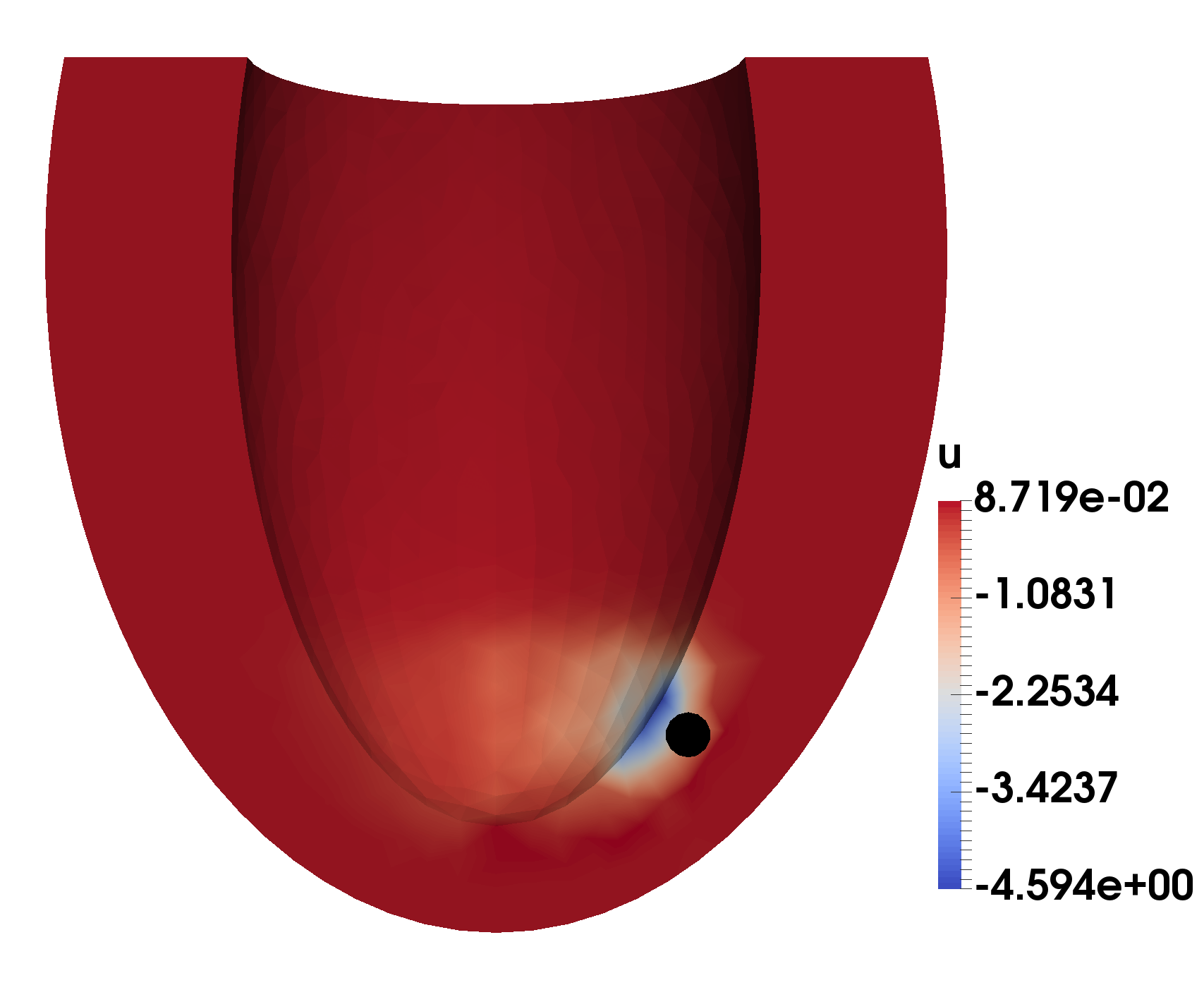}
		    	\includegraphics[width=0.3\textwidth]{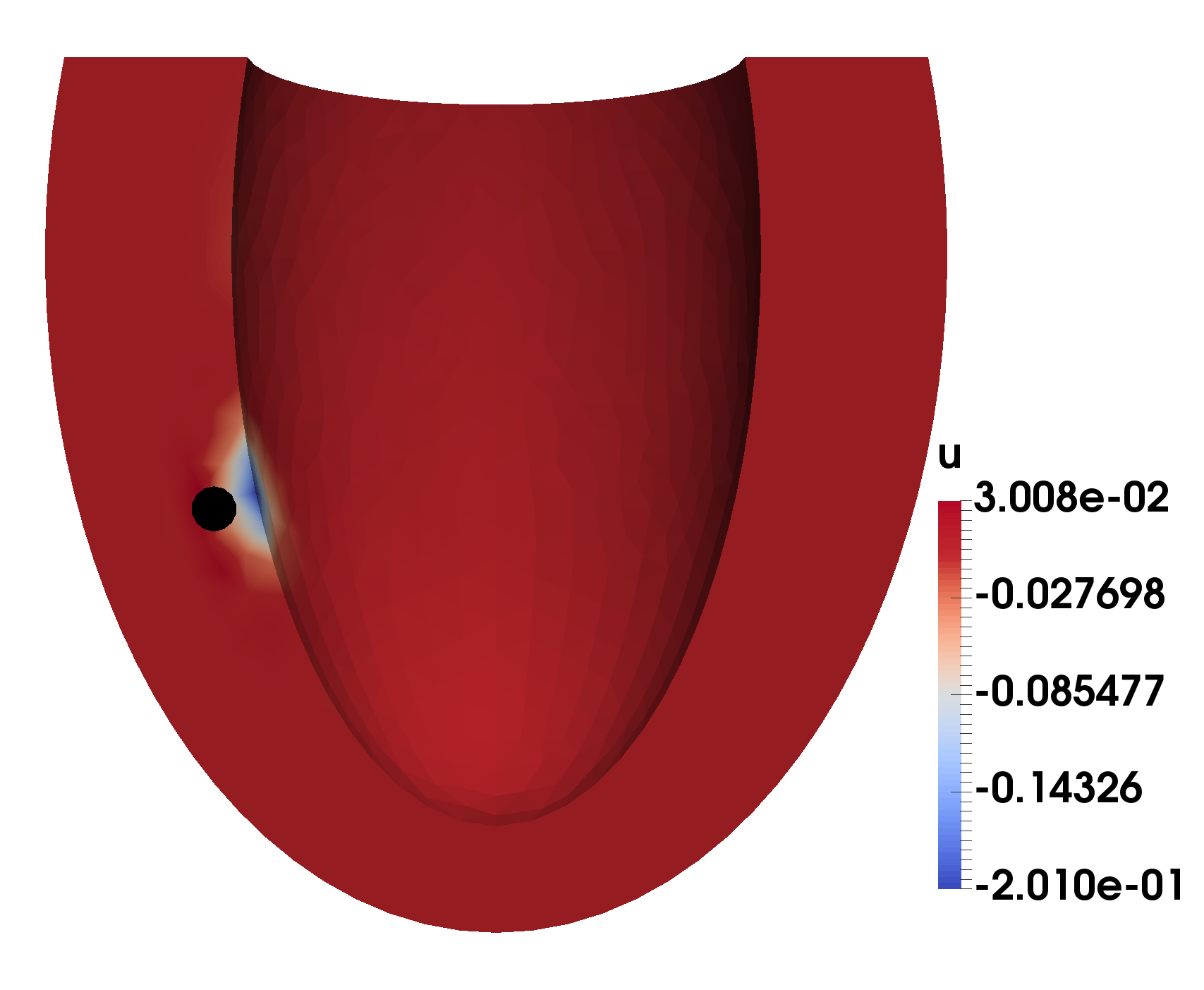}
		    	\includegraphics[width=0.3\textwidth]{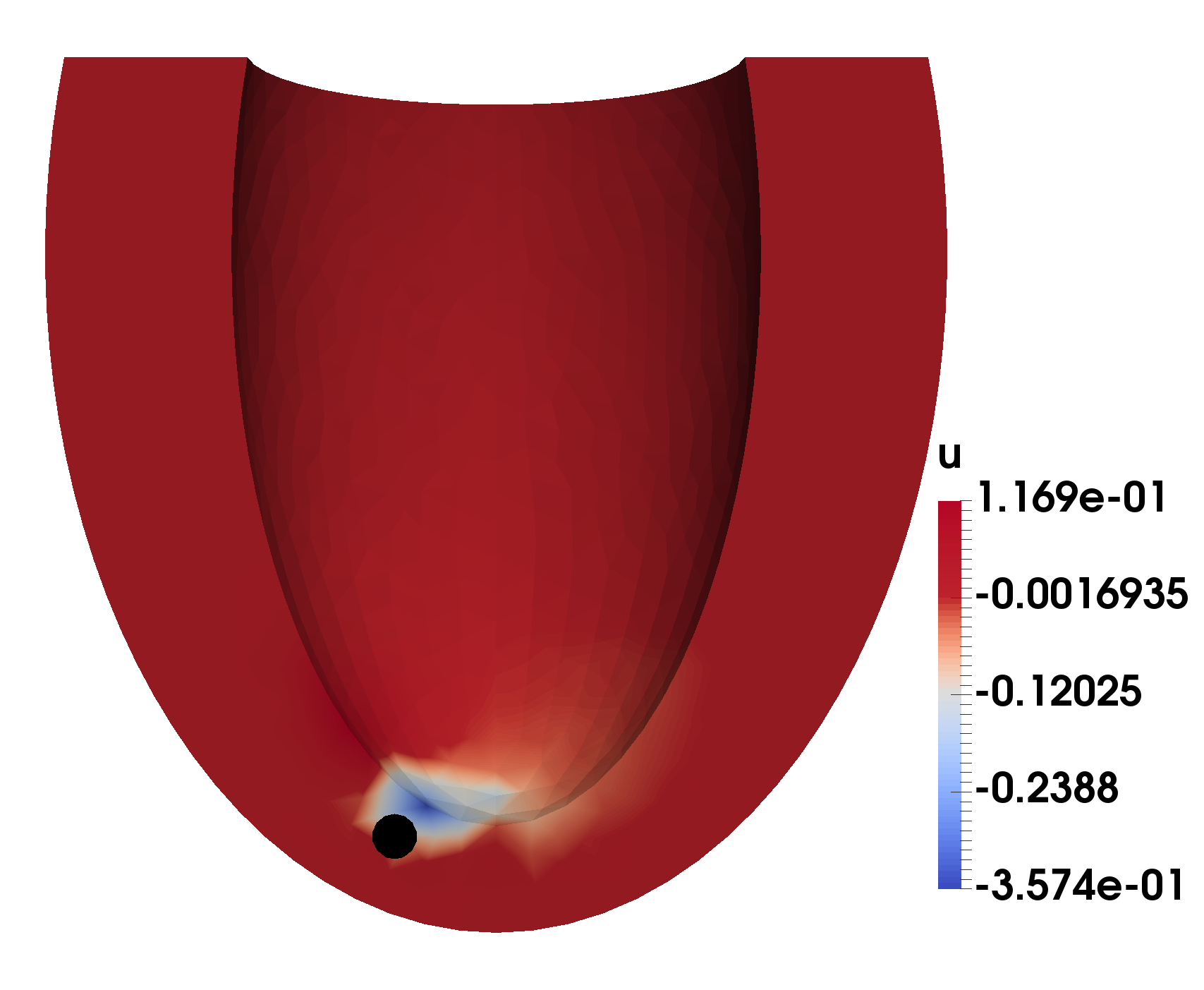}
			}
%
			\centerline{
			\includegraphics[width=0.3\textwidth]{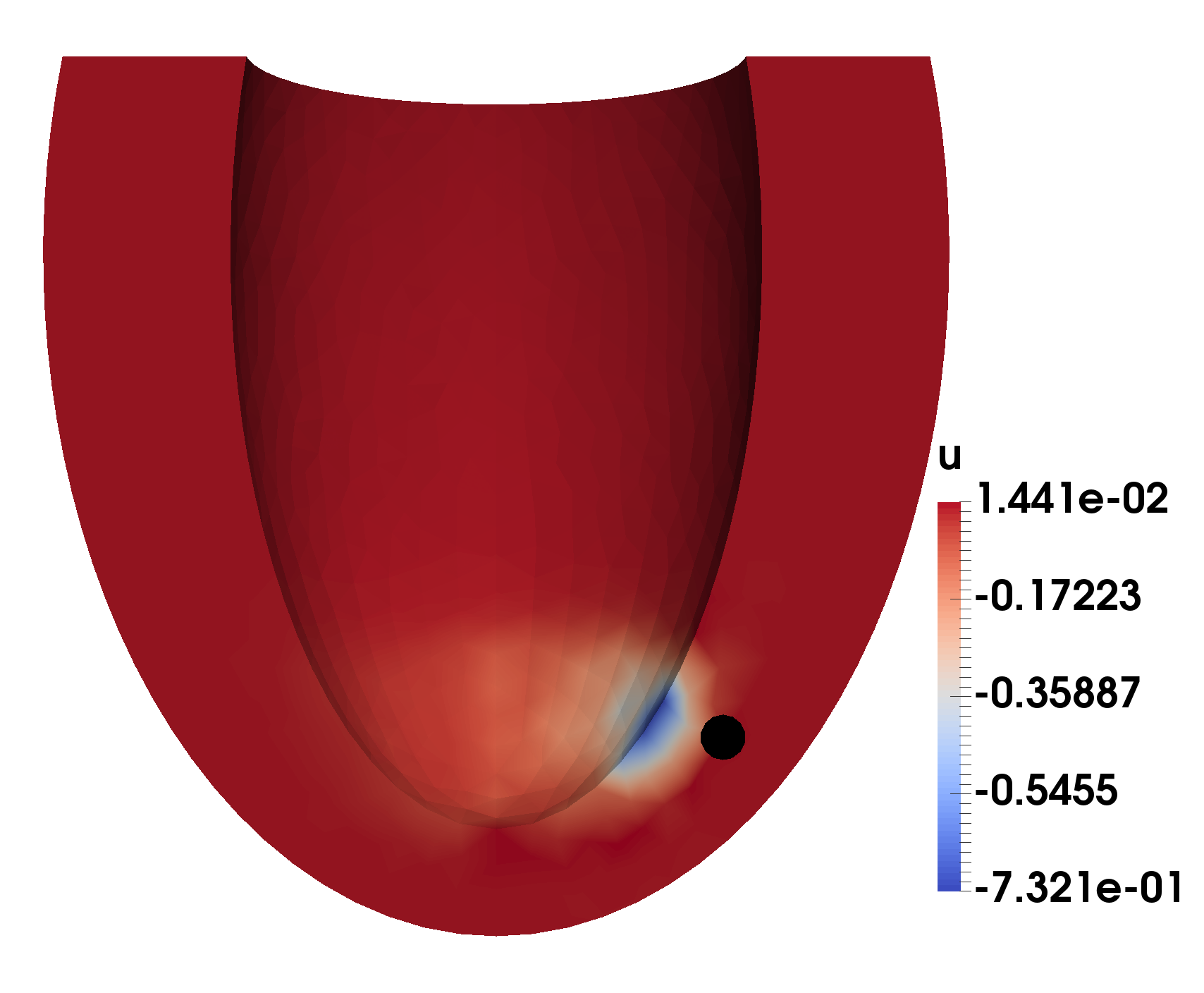}
			\includegraphics[width=0.3\textwidth]{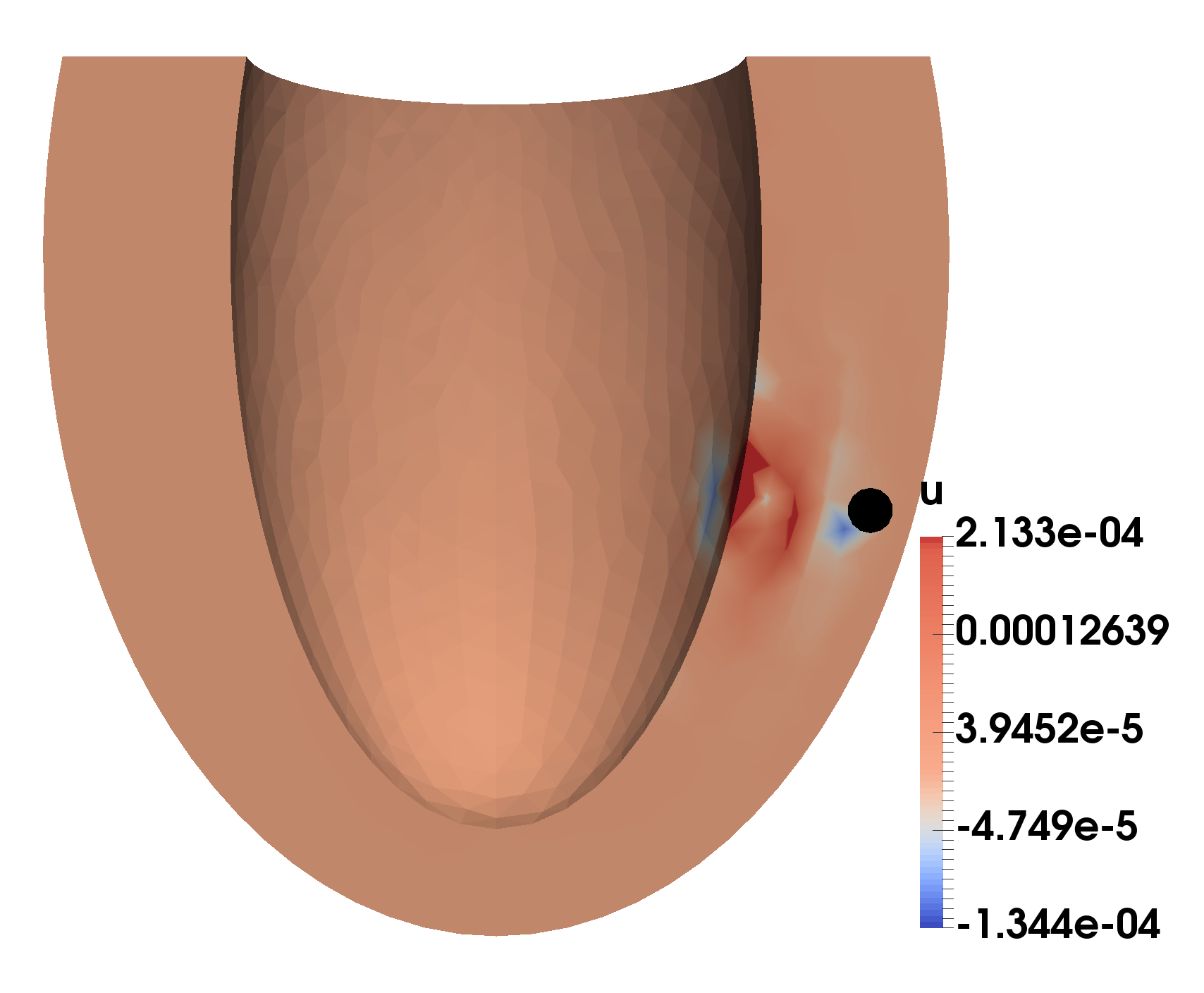}
			\includegraphics[width=0.3\textwidth]{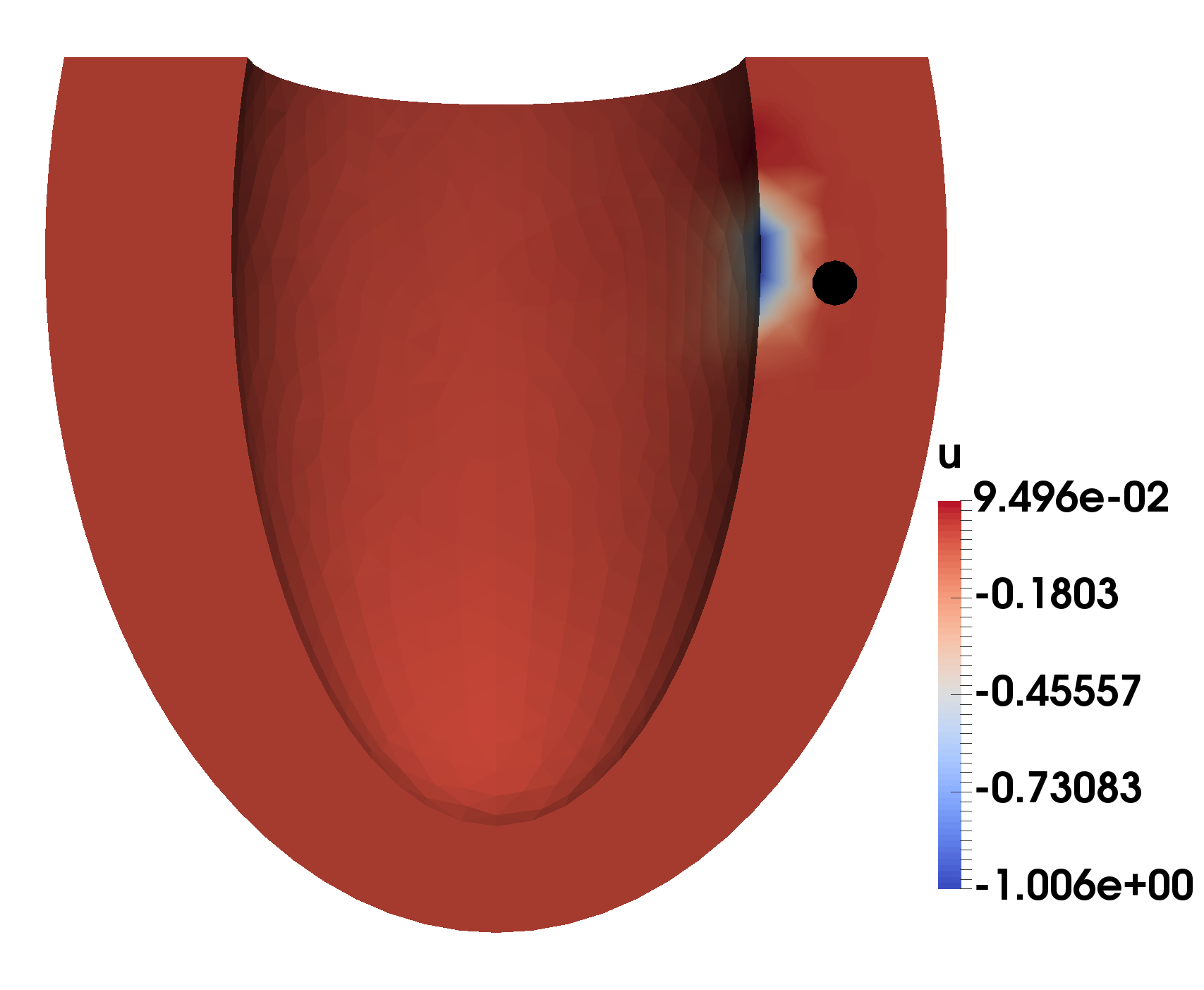}
			}
	\caption{Reconstruction: value of the topological gradient in different configurations}
	\label{fig:reconstructions}
	\vspace{-0.05cm}
	\end{figure}
\par
This slight loss in accuracy seems to be an intrinsic limit of the topological gradient strategy applied to the considered problem. We point out that the reconstruction is performed by relying on a single measurement
acquired on the boundary, which is indeed a constraint imposed by the physical problem. Hence, all the techniques relying on the introduction of several measurements to increase the quality of the reconstruction are impracticable.
A different strategy, as proposed in several works addressing the steady-state case, may consist in introducing a modification to the cost functional $J$. In \cite{ammari2012stability} and many related works the authors
introduce a cost functional inherited from imaging techniques, whereas in \cite{art:cmm}, \cite{park2012topological} different strategies involving the Kohn-Vogelius functional or similar ones are explored. Nevertheless,
the nonlinearity of the direct problem considered in this work prevents the possibility to apply these techniques, since the analitycal expressions of the fundamental solution, single and double layer potentials would not be available.

\subsection{Reconstruction in presence of experimental noise}
We test the stability of the algorithm in presence of experimental noise on the measured data $u_{meas}$. We consider different noise levels, according to the formula
\[
	\widetilde{u_{meas}}(x,t) = u_{meas}(x,t) + p \eta(x,t),
\]
where $\eta(x,t)$, for each point $x$ and instant $t$, is a Gaussian random variable with zero mean and standard deviation equal to $u_3-u_1$, whereas $p \in [0,1]$ is the noise level.
In Figure \ref{fig:error} the results of the reconstruction with different noise levels are compared. The algorithm shows to be highly stable with respect to high rates of noise, with increasing accuracy as the noise level reduces.
\begin{figure}[h!]
			\centering \vspace{-0.15cm}
			\subfloat[noise: $1\%$]{
		    	\includegraphics[width=0.3\textwidth]{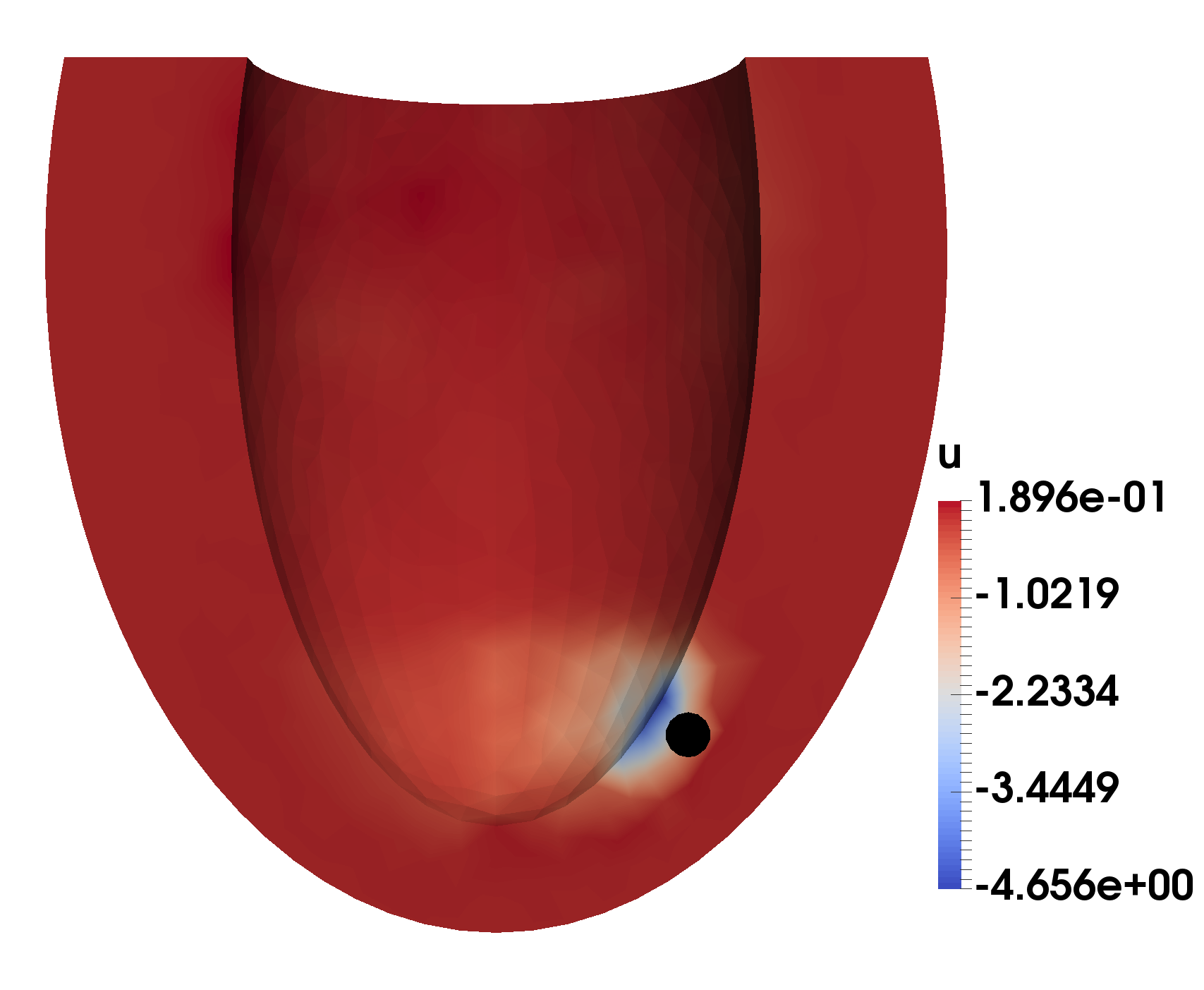}
			}
			\subfloat[noise: $5\%$]{
		    	\includegraphics[width=0.3\textwidth]{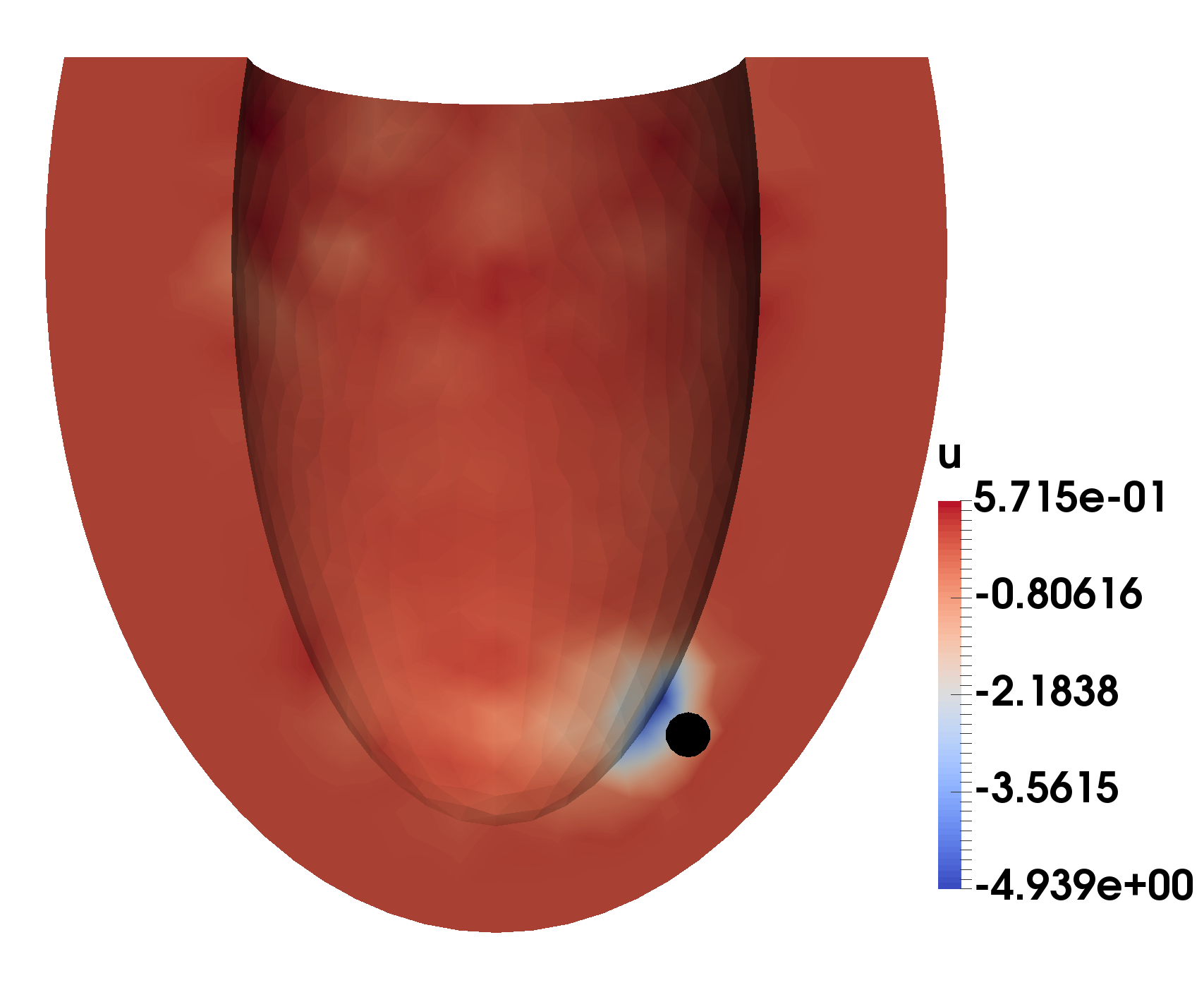}
			}
			\subfloat[noise: $10\%$]{
		    	\includegraphics[width=0.3\textwidth]{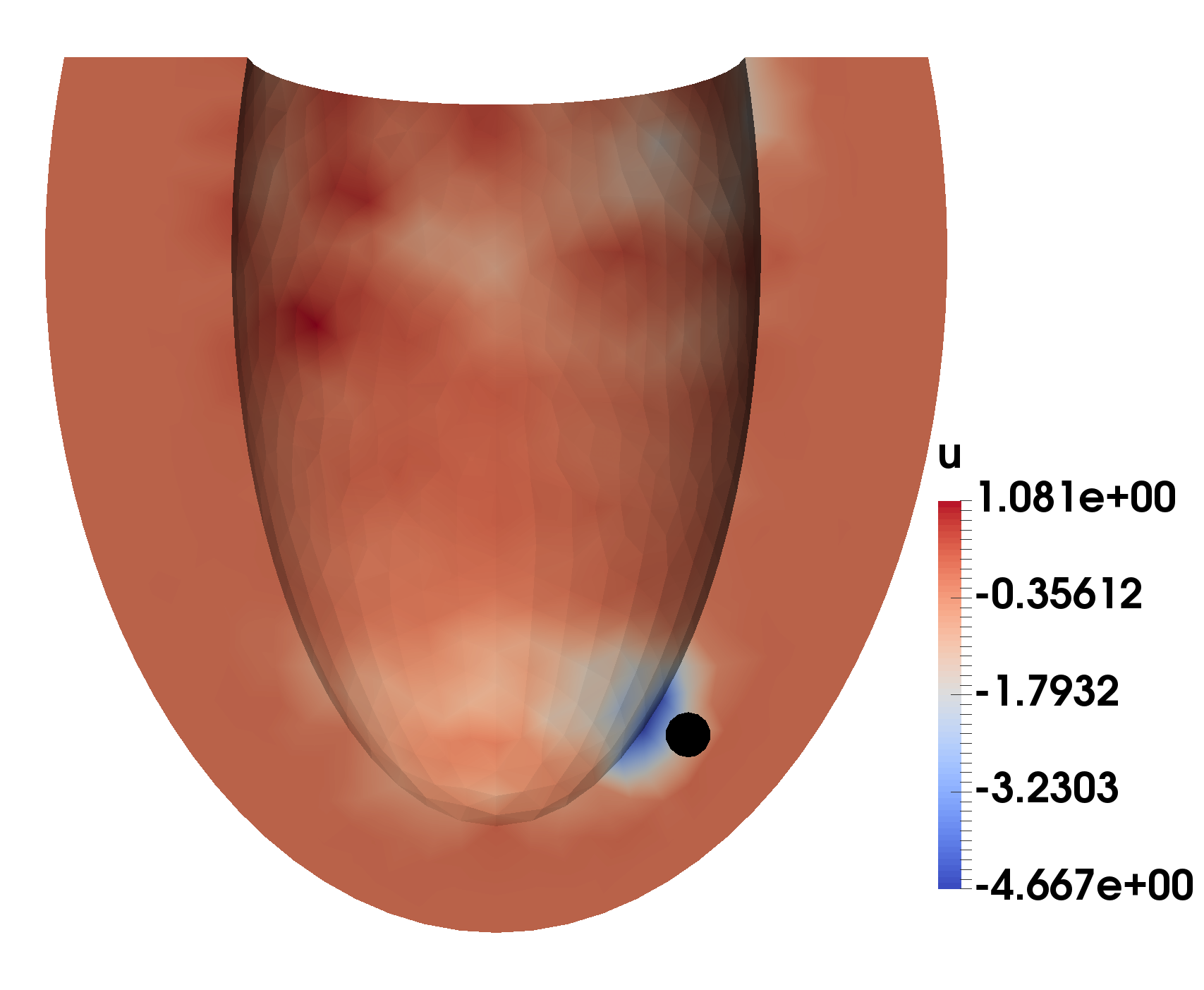}
			}
	\caption{Reconstruction results}
	\label{fig:error} \vspace{-0.1cm}
	\end{figure}

\subsection{Reconstruction from partial discrete data}
A further test case we have performed deals with the reconstruction of the position of small inclusions starting from the knowledge of partial data. We are interested in assessing the effectiveness of our algorithm when the electric potential is measured only on a discrete set of points on the endocardium, possibly simulating the procedure of intracavitary electric measurements. Figure \ref{fig:partial} shows that the algorithm is able to detect the region where the small ischemia is located from the knowledge of the potential on $N_{p} = 246, 61, 15$ different points. The
position of the reconstructed inclusion is slightly affected by the reduction of sampling points; nevertheless, reliable reconstructions can be obtained even with a very small number of points.

\begin{figure}[h!]
			\centering \vspace{-0.5cm}
			\subfloat[$N_{p}=246$]{
		    	\includegraphics[width=0.3\textwidth]{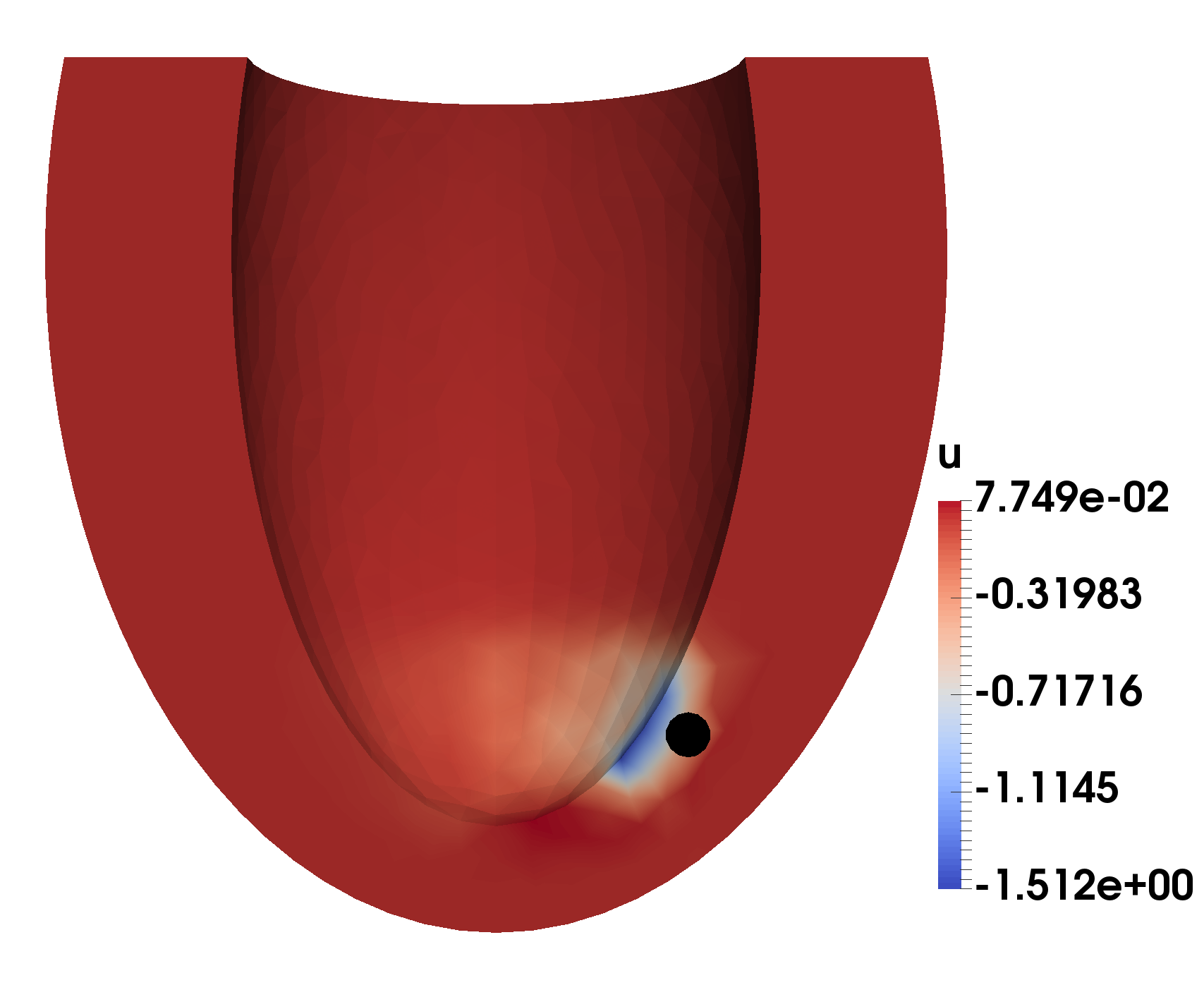}
			}
			\subfloat[$N_{p}=61$]{
		    	\includegraphics[width=0.3\textwidth]{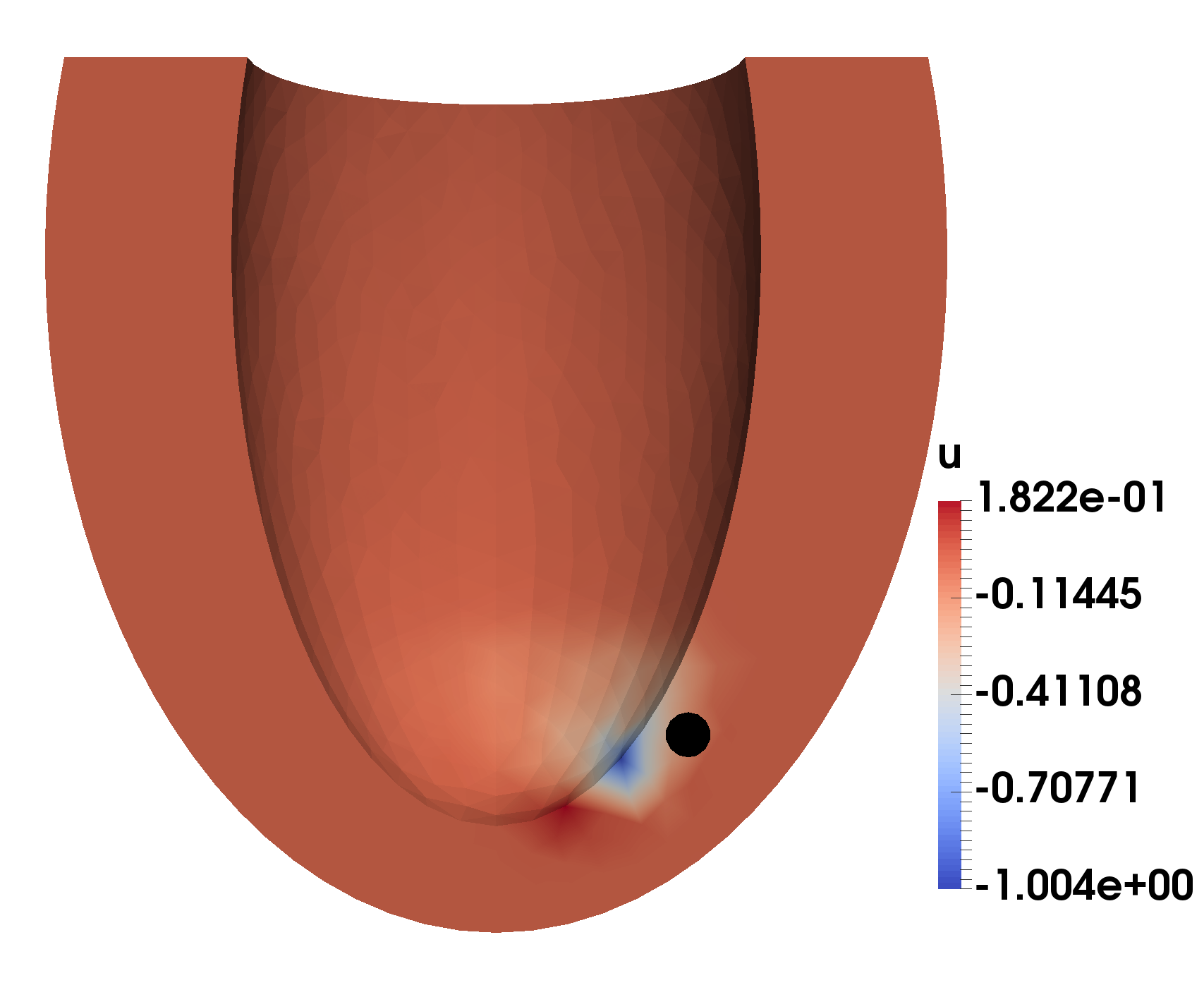}
			}
			\subfloat[$N_{p}=15$]{
		    	\includegraphics[width=0.3\textwidth]{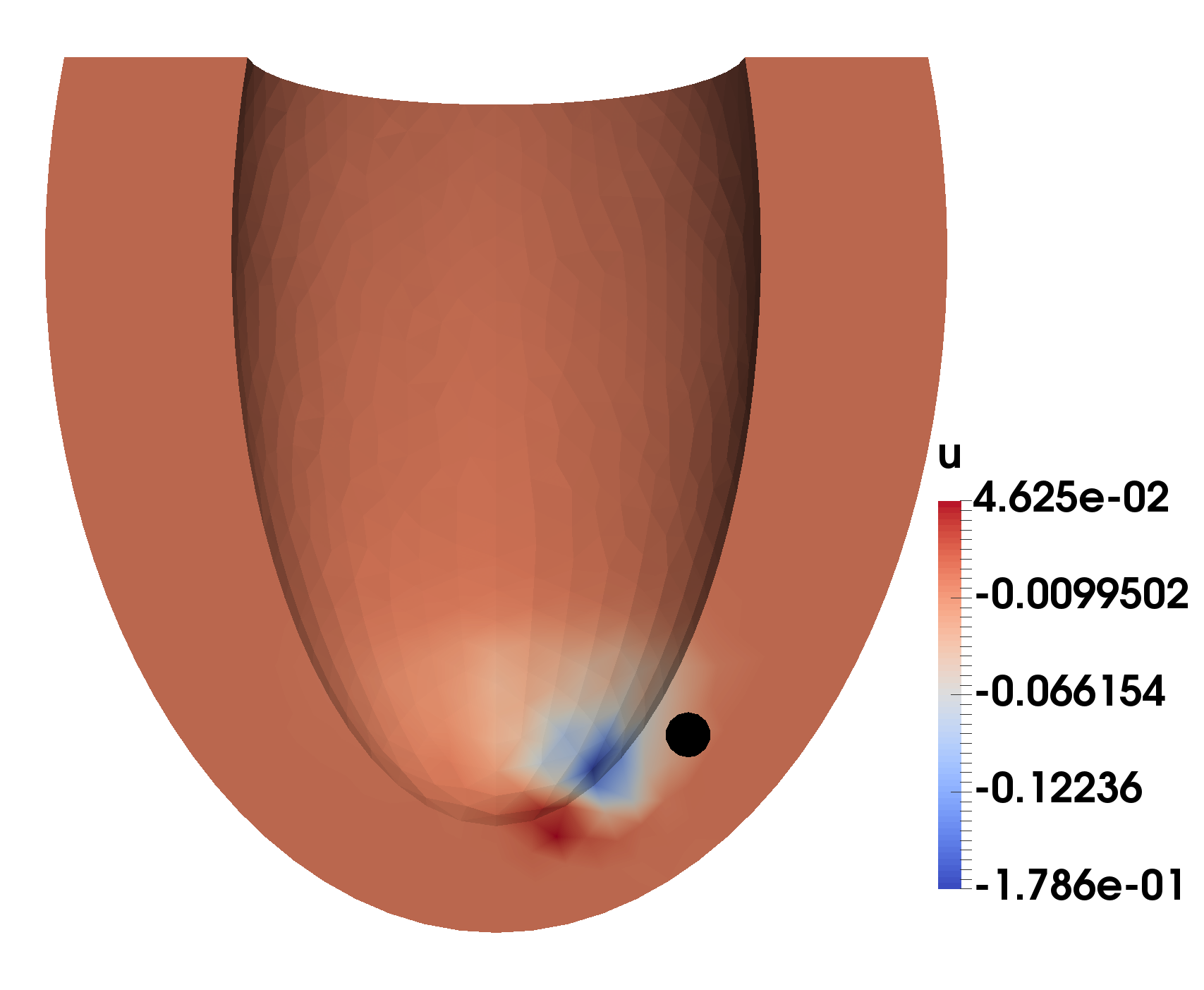}
			}
	\caption{Reconstruction results}
	\label{fig:partial} \vspace{-0.4cm}
	\end{figure}
	
For the same purpose, we have tested the capability of the reconstruction procedure to avoid false positives: the algorithm is able to distinguish the presence of a real ischemia from the case where no ischemic region is present, also in the case where the data are recovered only at a finite set of points, and are affected by noise. We compare the value of the cost functional and of the minimum of the topological gradient obtained through Algorithm \ref{al:top} on data generated when $(i)$ a small ischemia is present in the tissue or $(ii)$ no inclusion is considered. The measurement is performed on a set of $N_p = 100$ points and is affected by different noise levels. The results are reported in Table \ref{tab:falsepos}.

We observe that the presence of a small noise on the measured data causes a great increase of the cost functional $J$: with $5\%$ noise, e.g., the value of $J$ is two orders of magnitude greater than the value that $J$ assumes in presence of a small inclusion without noise.
Nevertheless, the topological gradient $G$ allows to distinguish the false positive cases, since (at least in case of a small noise level) the value attained by its minimum in presence of a small inclusion is considerably lower than the random oscillations of $G$ due to the noise.
\begin{table}[h!]
\centering
\subfloat[Results in presence of ischemia]{
\begin{tabular}[t]{c|c|c|c}
Error & $N_{p}$  & J & $min_\Omega G$ \\
\hline
  $0\%$ & 100 & 0.275 &-0.5793  \\
  $1\%$ & 100 & 1.235 &-0.589  \\
  $2\%$ & 100 & 4.128 &-0.530  \\
  $5\%$ & 100 & 24.429 &-0.589
\end{tabular}
} \hspace{0.2cm}
\subfloat[Results without ischemia]{
\begin{tabular}[t]{c|c|c|c}
Error & $N_{p}$  & J & $min_\Omega G$ \\
\hline
  $0\%$ & 100 & 0.000 & 0.000  \\
	$1\%$ & 100 & 0.964 &-0.044  \\
	$2\%$ & 100 & 3.864 &-0.105  \\
	$5\%$ & 100 & 24.148 &-0.189
\end{tabular}
}
\caption{False positive test. Comparison between reconstruction from data deriving from ischemic and healthy tissue. The null results in the first row of Table (b) are due to the usage of synthetic data.}
\label{tab:falsepos}
\end{table}

\subsection{Reconstruction of larger inclusions}\label{sec:75}
We finally assess the performance of Algorithm 1, developed for the reconstruction of small inclusions well separated from the boundary, in detecting the position of extended inclusions, which may be of greater interest in view of the problem of detecting ischemic regions.
Indeed, total occlusion of a major coronary artery generally causes the entire thickness of the ventricular wall to become ischemic (transmural ischemia) or, alternatively, a significant ischemia only in the endocardium, that is, the inner layer of the myocardium (subendocardial ischemia). See, e.g., \cite{Pavarino2010} for a detailed investigation of the interaction between the presence of moderate or severe subendocardial ischemic regions and the anisotropic structure of the cardiac muscle.
\par
The most important assumption on which our \textit{one-shot} procedure relies is that the variation of the cost functional from the value $J(0)$ attained in the background case to the value related to the exact inclusion can be correctly described by the first order term of its asymptotic expansion, the topological gradient $G$. Removing the hypothesis of the small size, we cannot rigorously assess the accuracy of the algorithm, however  it still allows us to identify the location of the ischemic region.
\par
The results reported in Figure \ref{fig:bigisch} show that in presence of a  inclusion of larger size (and not even separated from the boundary), the minimum of the topological gradient is found to be close to the position of the inclusion, and attains lower values with respect to the previously reported cases.
\begin{figure}[h!]
			\centering \vspace{-0.5cm}
			\subfloat[Exact inclusion]{
		    	\includegraphics[width=0.3\textwidth]{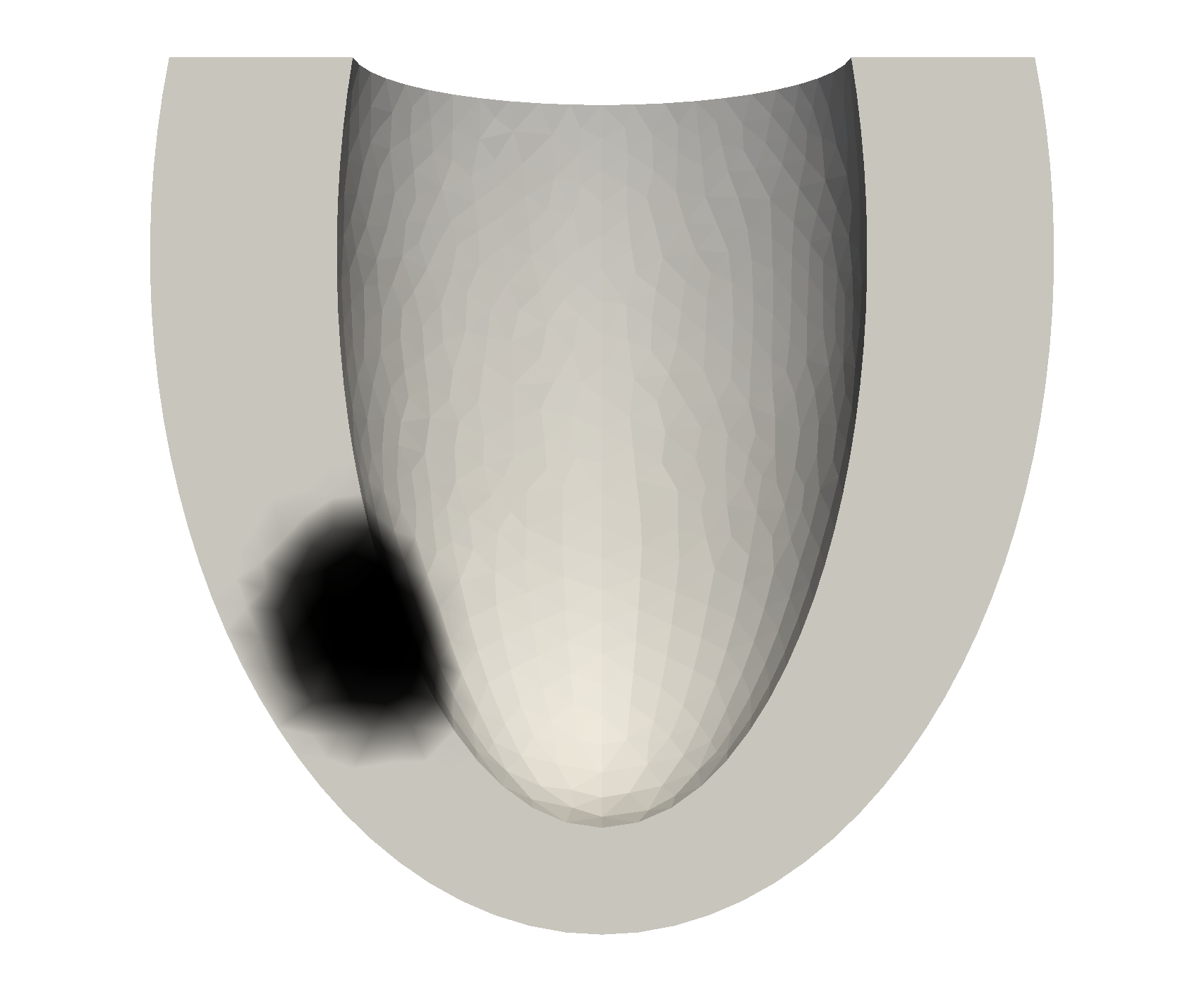}
			}
			\subfloat[Topological gradient]{
		    	\includegraphics[width=0.3\textwidth]{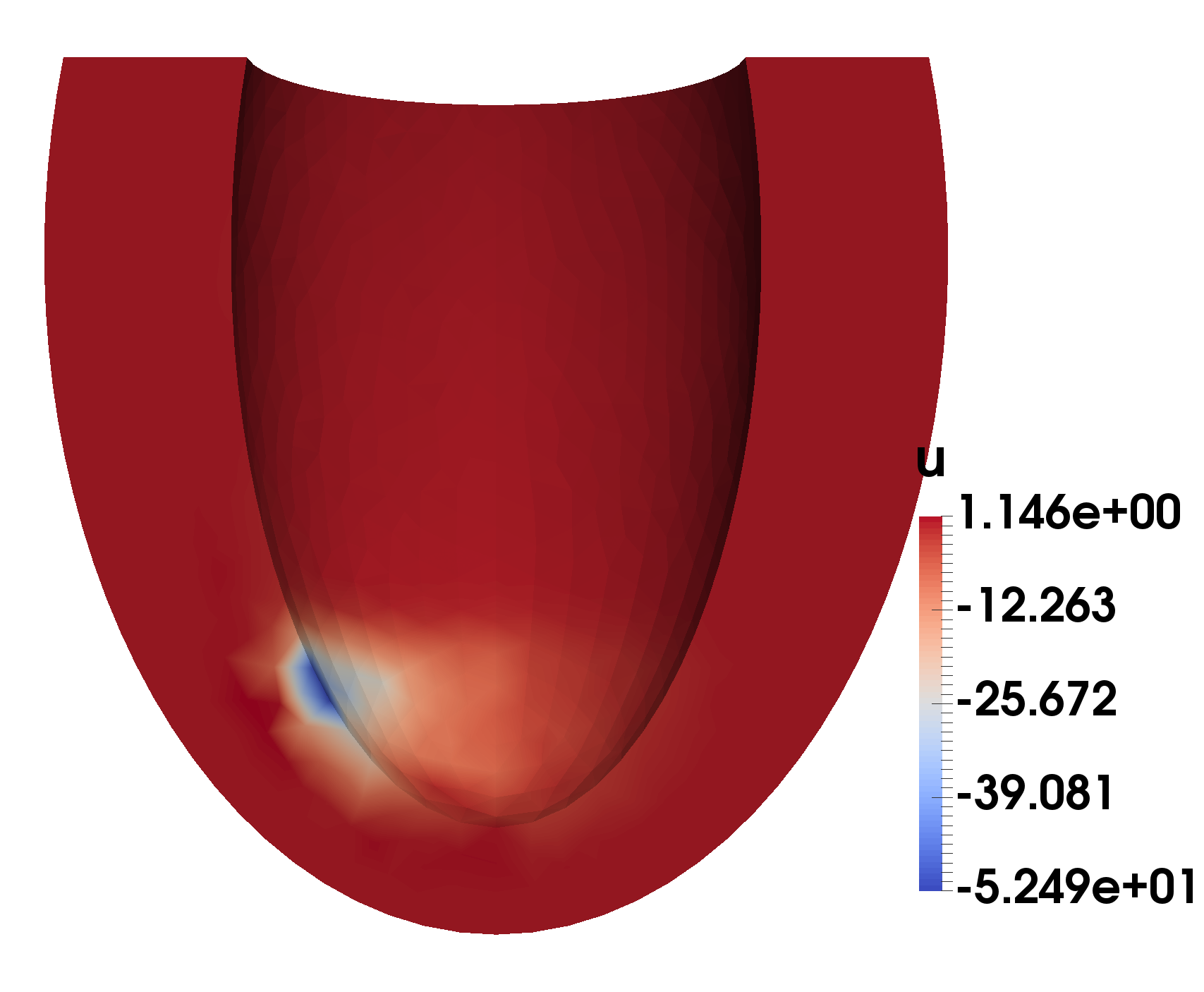}
			}
			\\
			\subfloat[Exact inclusion]{
		    	\includegraphics[width=0.3\textwidth]{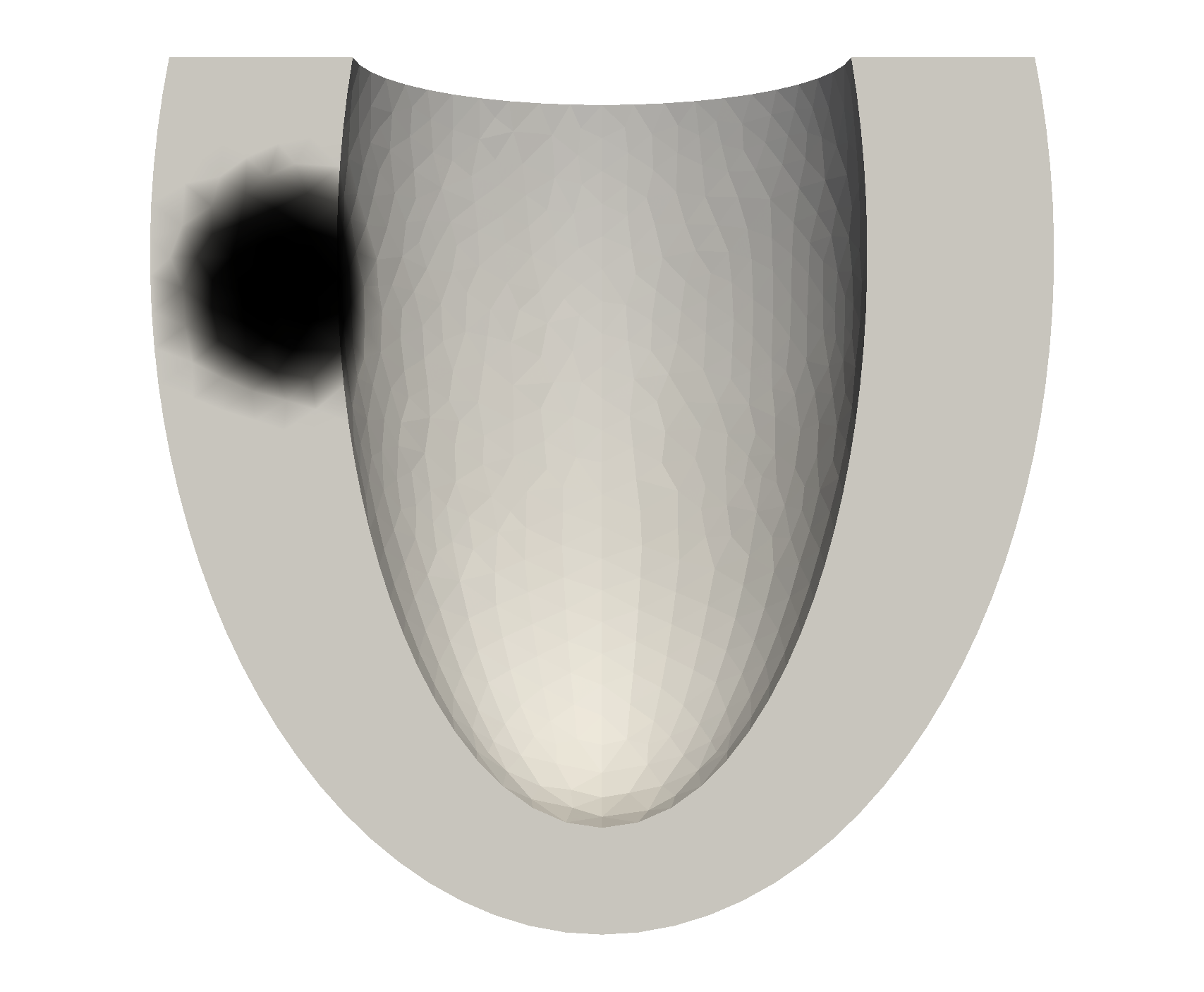}
			}
			\subfloat[Topological gradient]{
		    	\includegraphics[width=0.3\textwidth]{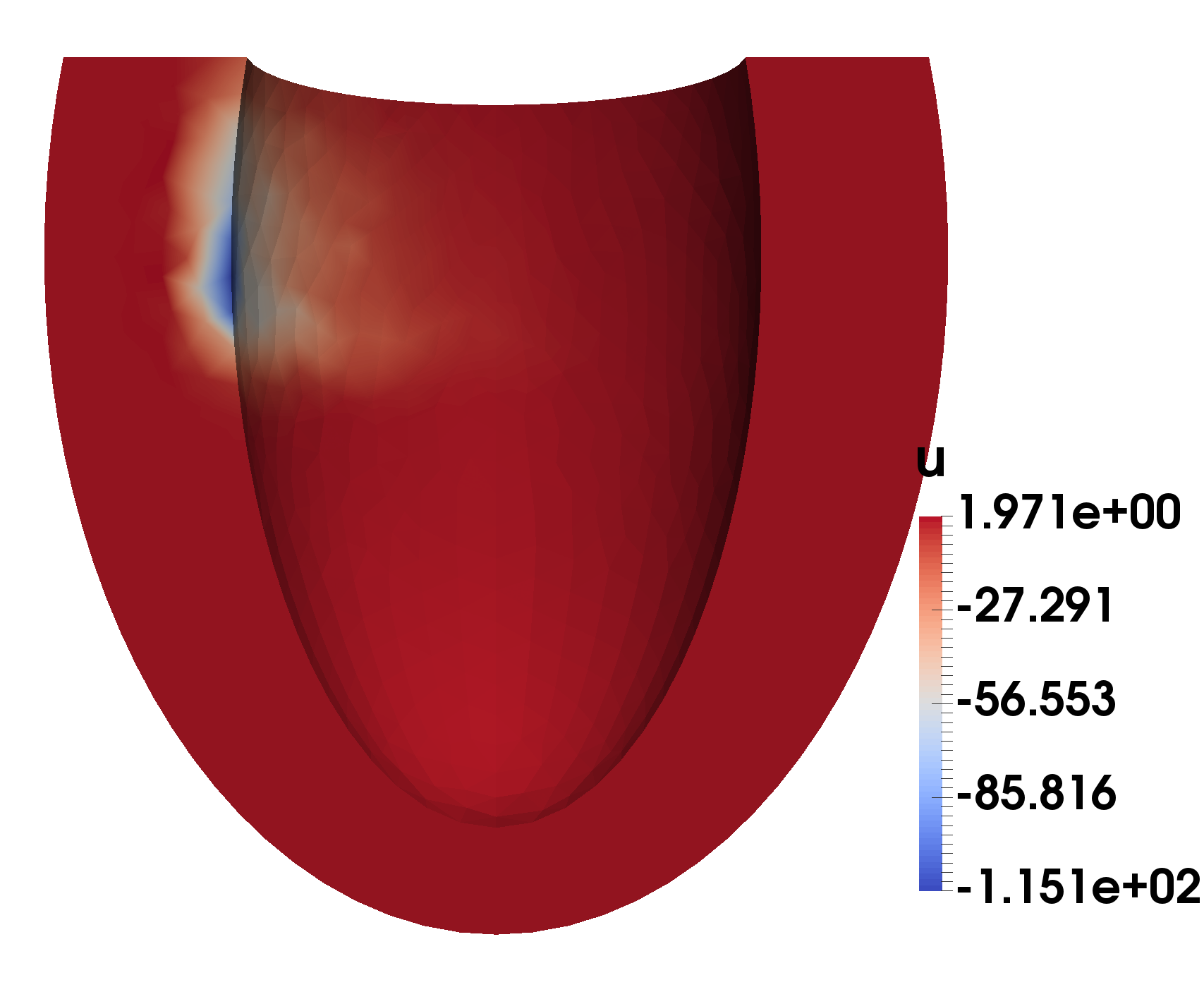}
			}
	\caption{Larger ischemic regions: reconstruction results}
	\label{fig:bigisch} \vspace{-0.4cm}
\end{figure}
\par
Moreover, in Figure \ref{fig:bigisch2} we also assess the stability of the reconstruction with respect to the presence of noisy data and partial measurements, as done in the case of small inclusions.
\begin{figure}[b!]
			\centering \vspace{-0.5cm}
			\subfloat[Exact inclusion]{
		    	\includegraphics[width=0.3\textwidth]{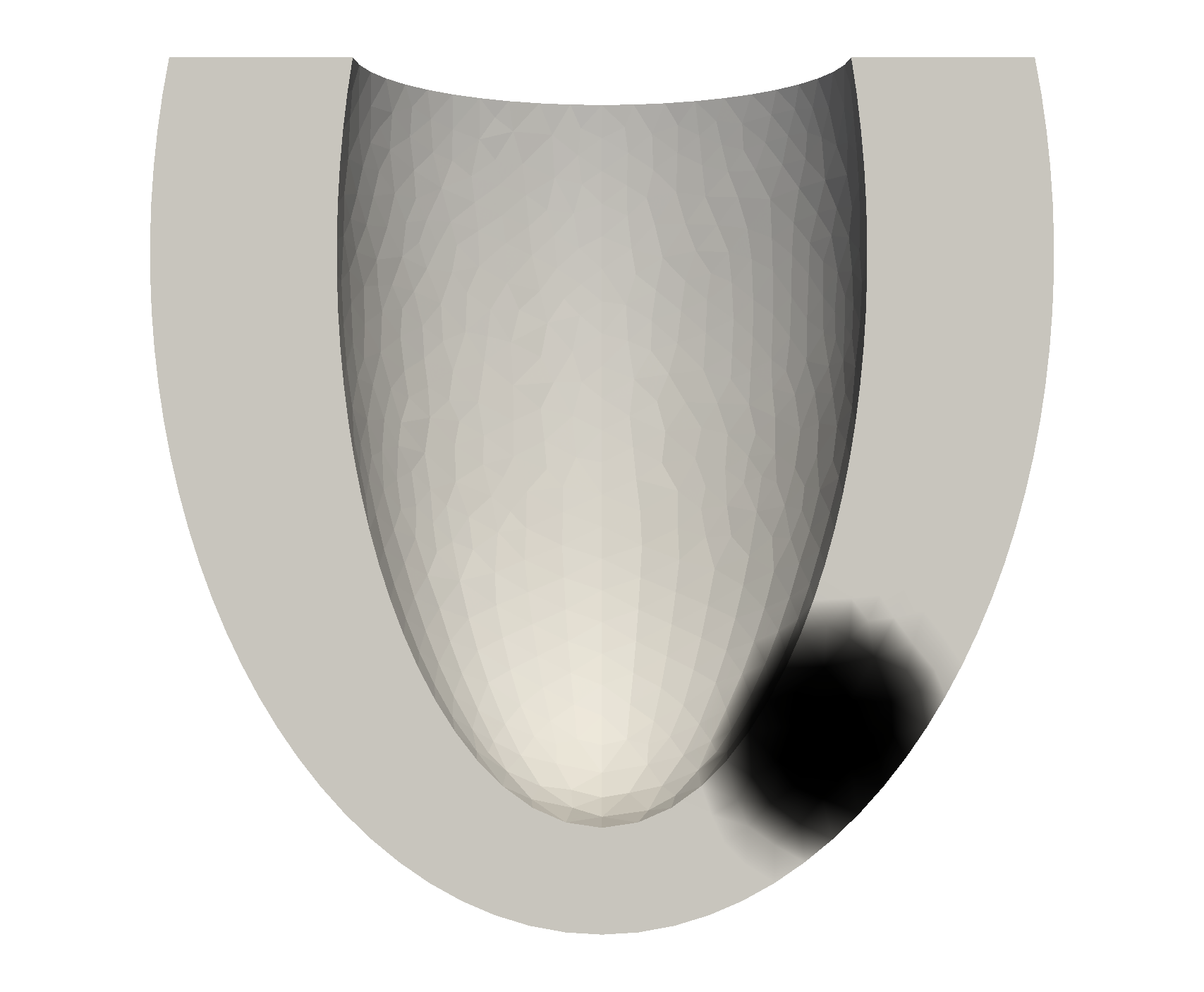}
			}
			\subfloat[Topological gradient, \newline $2\%$ noise]{
		    	\includegraphics[width=0.3\textwidth]{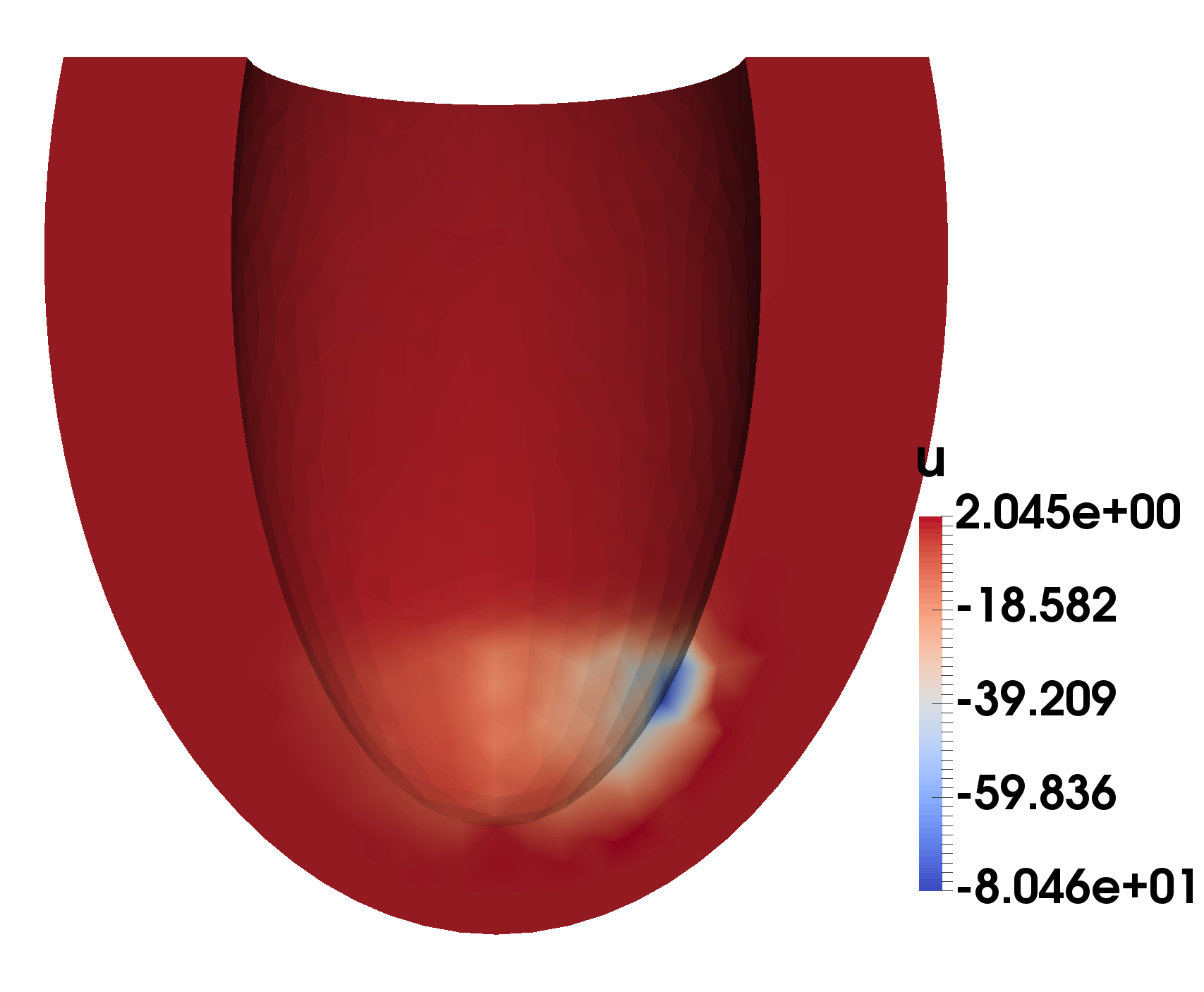}
			}
			\subfloat[Topological gradient, \newline  $2\%$ noise, measurements \newline on 100 points]{
		    	\includegraphics[width=0.3\textwidth]{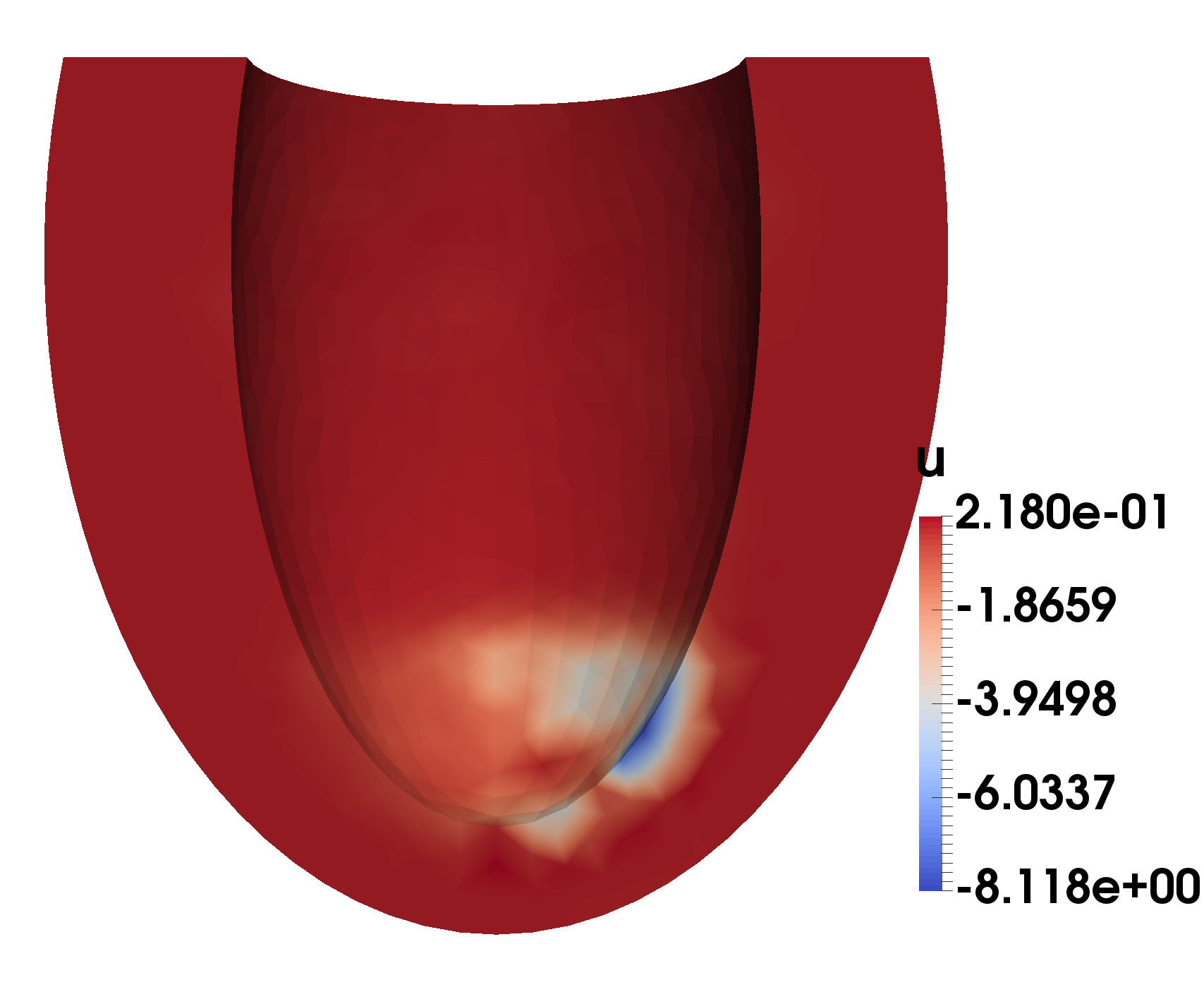}
			}
	\caption{Larger ischemic regions: stability of the reconstruction}
	\label{fig:bigisch2}\vspace{-0.1cm}
	\end{figure}

\section{Conclusions and perspectives}
A rigorous theoretical analysis of the inverse problem of detecting inhomogeneities in the monodomain equations has allowed us to set up a numerical reconstruction procedure, aiming at the detection of ischemic
regions in the myocardic tissue from a single measurement of the endocardial potential. The identification is made possible by evaluating the topological gradient  of a quadratic cost functional, requiring
the solution of two initial and boundary value problems, the background problem and the adjoint one. Numerical results are encouraging and allow to estimate the position of the inclusion,
although the identified inhomogeneity is nearly always detected on the boundary where the measurement is acquired. Nevertheless, provided a single measurement can be used for the sake of  identification,
and a {\em one-shot} procedure is performed, the obtained results give useful insights.

Many issues are still open. Concerning the mathematical model, an even more interesting case would be the one involving the heart-torso coupling is considered, so that more realistic (and noninvasive)
body surface measurements can be employed. Setting and analyzing the inverse problem in this context represents the natural continuation  of the present work.
To close the gap between the rigorous mathematical setting and the practice, the two assumptions made in this work about the size of the inclusion and its distance from the boundary should be relaxed.
Numerical results shown in Section \ref{sec:75}  provide a first insight on the detection of inclusions with larger size, as those corresponding to transmural or subendocardial ischemias.
From a mathematical standpoint, this problem is still open. Also in the case of a linear direct problem, very few results can be found in literature, see, e.g. \cite{Ammari2013}.
Estimating the size of the inclusion is another open question in the case of parabolic PDEs, also for linear equations. The case of multiple inclusions, addressed in \cite{BerettaManzoniRatti}
for a stationary nonlinear problem, could also be considered. Last, but not least, the topological optimization framework addressed in this paper could also be combined with an iterative algorithm,
such as the level set method, or with the solution of a successive shape optimization problem, to achieve a full reconstruction both of the dimension and the shape of the inclusion.

\section{Appendix - Proof of Proposition \ref{boundaryNorm}}
\setcounter{equation}{0}

\begin{proof}
Setting $Z(t) = \Phi(T-t), \, t \in (0,T)$, we get an equivalent problem to \eqref{IR1}
	\begin{equation}\label{IR1bis}
		\begin{cases}
			& Z_t - k_0\Delta Z + f^\prime(u)Z  = 0,\ \ \ {\rm in} \ \ \Omega\times (0,T),\\
			& \displaystyle\frac{\partial Z}{\partial n}  = u^\varepsilon - u, \ \ \  {\rm on } \ \ \partial\Omega \times (0,T), \\
			& Z(0) = 0, \ \ \ {\rm in} \ \ \Omega.
			\hskip2truecm
		\end{cases}
	\end{equation}
We will prove that $Z \in L^2(0,T;H^3(K)) \hookrightarrow L^1(0,T;W^{1,\infty}(K))$. To this aim we need to derive some energy estimates.
Multiplying the first equation in \eqref{IR1bis} by $Z$, an application of Young's inequality leads to
\begin{equation}\label{IR12bis}
\frac12 \frac{d}{dt}\|Z(t)\|^2_{L^2(\Omega)}  + \frac{k_0}{2}\|\nabla Z\|_{L^2(\Omega)}^2
\leq C (\|Z\|_{L^2(\Omega)}^2 + \|u^\varepsilon -u\|^2_{L^2(\partial\Omega)}),
\end{equation}
where $C= C(k_0, M_2, \Omega) >0$. An application of Gronwall's lemma gives
$$
\|Z(t)\|^2_{L^2(\Omega))} \leq C \|u^\varepsilon -u\|^2_{L^2(0,t; L^2(\partial\Omega))}, \quad \forall \, t \in [0,T],
$$
so that
\begin{equation}\label{IR13}
\|Z\|^2_{L^\infty(0,t; L^2(\Omega))} \leq C\|u^\varepsilon -u\|^2_{L^2(0,t; L^2(\partial\Omega))}, \quad \forall \, t \in [0,T].
\end{equation}
Instead, integrating \eqref{IR12bis} in time over $[0,t]$ we get
$$\int_0^t\|\nabla Z\|_{L^2(\Omega)}^2
\leq C\left ( \int_0^t\|Z\|_{L^2(\Omega)}^2 + \int_0^t\|u^\varepsilon -u\|^2_{L^2(\partial\Omega)}\right )$$
and finally
\begin{equation}\label{IR14}
\|\nabla Z\|_{L^2(0,t;L^2(\Omega))}^2
\leq C\|u^\varepsilon -u\|^2_{L^2(0,t;L^2(\partial\Omega))}, \quad \forall \, t \in [0,T],
\end{equation}
where $C$ is a positive constant depending on $k_0, M_2, \Omega, T$.
We remark that, by standard regularity results, $Z$ is smooth on $E \times [0,T]$, for any compact $E\subset \Omega$.

Consider now two compact sets $K_1$ and $K_2$ such that
$$K\subset K_2 \subset K_1 \subset \Omega, \quad d(k_0, \partial\Omega) \geq d_1 >0.$$
It is possible to construct two functions $\xi_1, \, \xi_2$ and two constants $b_1, \, b_2$ satisfying
$$\xi_i \in C^2(\overline\Omega), \quad 0\leq \xi_i \leq 1, \quad \xi_i(x) = 1 \quad \forall \, x \in K_i, \quad \xi_i(x) = 0
\quad \forall \,x \in B_i\quad i=1,2,$$
$$B_i = \{ x \in \Omega \, : \, d(x, \partial \Omega) \leq b_i\}, \quad 0 < b_1 < b_2 < d_1, \quad
K \subset\subset {\rm Supp}\, \xi_2 \subset\subset K_1 \subset {\rm Supp}\, \xi_1 \subset \Omega.$$
Let us multiply the first equation of \eqref{IR1} by $-\Delta Z$. Then the following holds
\begin{equation}\label{IR15}
\frac{d}{dt}\left (\frac12|\nabla Z|^2\right) + k_0(\Delta Z)^2 - f^\prime(u)Z\Delta Z = {\rm div}\,(Z_t \nabla Z).
\end{equation}
Multiplying \eqref{IR15} by $\xi_1$, integrating on $\Omega\times (0,T)$ and using the definitions of $Z$ we get
\begin{equation}\label{IR16}
\int_\Omega\left (\frac12|\nabla Z(T)|^2 \right)\xi_1 + k_0\int_0^T\int_\Omega(\Delta Z)^2\xi_1
= \int_0^T\int_\Omega f^\prime(u)Z\Delta Z \xi_1  - \int_0^T\int_\Omega Z_t \nabla Z \cdot \nabla\xi_1.
\end{equation}
Combining \eqref{IR16} and the first equation in \eqref{IR1bis}, applying Young's inequality and taking into account \eqref{M3} and the fact
that $0 \leq \xi\leq 1$, we obtain
$$
\int_\Omega\left (|\nabla Z(T)|^2 \right)\xi_1 + k_0\int_0^T\int_\Omega(\Delta Z)^2\xi_1
\leq 2M_2\int_0^T\int_\Omega Z^2- 2\int_0^T\int_\Omega (k_0\Delta Z - f^\prime(u)Z) \nabla Z \cdot \nabla\xi_1.
$$
Integrating by parts the term $\int_0^T\int_\Omega \Delta Z \nabla Z \cdot \nabla\xi_1$, we can easily deduce
\begin{equation}\label{IR18}
\int_\Omega\left (|\nabla  Z(T)|^2 \right)\xi_1 + \int_0^T\int_\Omega(\Delta  Z)^2\xi_1
\leq C\left (\| Z\|_{L^2(0,T;L^2(\Omega))}^2 + \|\nabla Z\|^2_{L^2(0,T;L^2(\Omega))} \right),
\end{equation}
where $C$ is a positive constant depending on $M_2$, $k_0$, $\xi_1$.
Hence, since $\xi_1 = 1$ in $k_0$, we get
\begin{equation}\label{IR19}
\|\Delta  Z \|^2_{L^2(0,T;L^2(K_1))} \leq C\left ( \| Z\|_{L^2(0,T;L^2(\Omega))}^2 + \|\nabla Z\|^2_{L^2(0,T;L^2(\Omega))} \right).
\end{equation}
Observe that, replacing $T$ by $t \in (0,T]$ in \eqref{IR18}, we deduce also
\begin{equation}\label{IR20}
\|\nabla  Z\|_{L^\infty(0,T;L^2(K_1))} \leq C\left (\| Z\|_{L^2(0,T;L^2(\Omega))}^2 + \|\nabla Z\|^2_{L^2(0,T;L^2(\Omega))} \right).
\end{equation}
Combining \eqref{IR13}, \eqref{IR19} and \eqref{IR20}, we obtain
\begin{equation}\label{IR21}
\|  Z \|^2_{L^2(0,T;H^2(K_1))} \leq C \|u^\varepsilon - u\|_{L^2(0,T;L^2(\partial\Omega))}^2,
\end{equation}
where $C$ is a positive constant depending on $k_0, M_2, \Omega, T, \xi_1$.

On account of the first equation in \eqref{IR1bis} and the previous estimates, we get
\begin{equation}\label{IR22}
\| Z_t\|_{L^2(0,T;L^2(K_1))}^2  \leq C (\| Z\|_{L^2(0,T;L^2(\Omega))}^2
+ \|\nabla  Z\|_{L^2(0,T;L^2(\Omega))}^2) \leq  C \|u^\varepsilon - u\|_{L^2(0,T;L^2(\partial\Omega))}^2,
\end{equation}
where $C$ is a positive constant depending on $k_0, M_2, \Omega, T, \xi_1$.

Now, let us multiply the first equation of \eqref{IR1bis} by  $-\Delta  Z_t$. We obtain
$$
-  Z_t \Delta  Z_t + \frac{k_0}{2}\frac{d}{dt}(\Delta Z)^2 - f^\prime(u) Z\Delta  Z_t = 0.
$$
Multiplying the previous equation by $\xi_2$ and integrating on $\Omega \times (0,T)$, then a
suitable integration by parts in space implies
$$
\int_0^T\int_\Omega|\nabla Z_t|^2\xi_2 + \frac{k_0}{2}\int_0^T\int_\Omega\frac{d}{dt}(\Delta Z)^2\xi_2
+ \int_0^T\int_\Omega\xi_2 Z f^{\prime\prime}(u)\nabla u \cdot\nabla Z_t + \int_0^T\int_\Omega \xi_2f^\prime(u)\nabla  Z\cdot\nabla Z_t
$$
$$
= \int_0^T\int_\Omega {\rm div}\,\left ( \frac12\nabla(( Z_t)^2)\right ) \xi_2
+ \int_0^T\int_\Omega  {\rm div}\, \Big (\nabla(f^\prime(u) Z  Z_t) -  Z_t\nabla(f^\prime(u) Z)\Big)\xi_2.
$$
Integrating by parts the second term of the left-hand side and by parts in space the terms in the right-hand side, by an application of Young's
inequality we finally get
$$
\int_0^T\int_{K_2}|\nabla Z_t|^2 \leq \int_0^T\int_\Omega|\nabla Z_t|^2\xi_2 \leq
C\left ( \int_0^T\int_\Omega| Z|^2  + \int_0^T\int_\Omega  |\nabla Z|^2 + \int_0^T\int_{K_1} ( Z_t)^2\right ),
$$
where the constant $C>0$ depends on $\xi_2, M_2$.
A combination with \eqref{IR13}, \eqref{IR14}, \eqref{IR22} gives
$$
\|\nabla Z_t\|_{L^2(0,T;L^2(K_2))}^2 \leq
C \|u^\varepsilon - u \|_{L^2(0,T;L^2(\partial\Omega))}^2,
$$
where the constant $C>0$ depends on $k_0, M_2, \Omega, T, \xi_1, \xi_2$.
In order to prove the desired regularity, we need to take into account also the third-order derivatives, in particular the operator $\nabla \Delta  Z$.
Observe that from the first equation in \eqref{IR1bis} we get
\begin{equation}
	\nabla\Delta  Z = \frac{1}{k_0}\left ( \nabla Z_t +  Z f^{\prime\prime}(u)\nabla u + f^\prime(u)\nabla  Z \right ).
\label{thirdOrd}
\end{equation}
Hence, we can conclude
$$
\|\nabla\Delta Z\|_{L^2(0,T;L^2(K_2))}^2 \leq
C \|u^\varepsilon - u \|_{L^2(0,T;L^2(\partial\Omega))}^2,
$$
where $C$ is a positive constant depending on $k_0, \frac{1}{k_0}, M_2, \Omega, T, \xi_1, \xi_2$.

Recalling \eqref{IR21} and the fact that $K \subset K_2 \subset K_1$, standard regularity results imply
\begin{equation}\label{IR23}
\| Z\|_{L^2(0,T;H^3(K))}^2 \leq
C \|u^\varepsilon - u \|_{L^2(0,T;L^2(\partial\Omega))}^2.
\end{equation}
Finally, from \eqref{IR21} and \eqref{IR23}, by Sobolev immersion theorems, we get
\begin{equation}\label{IR24}
\| Z\|_{L^1(0,T;W^{1,\infty}(K))}^2  \leq C(T) \| Z\|_{L^2(0,T;W^{1,\infty}(K))}^2 \leq
C \|u^\varepsilon - u \|_{L^2(0,T;L^2(\partial\Omega))}^2,
\end{equation}
where $C$ is a positive constant depending on $k_0, \frac{1}{k_0}, M_2, \Omega, T, \xi_1, \xi_2$.

Recalling the relation between $Z$ and $\Phi$ we get \eqref{eq:infEst}.
\end{proof}


\section*{Acknowledgments}
E. Beretta, C. Cavaterra, M.C. Cerutti and L. Ratti thank the New York University in Abu Dhabi
for its kind hospitality that permitted a further development of the present research.
The work of C. Cavaterra was supported by the FP7-IDEAS-ERC-StG 256872 (EntroPhase)
and by GNAMPA (Gruppo Nazionale per l'Analisi Matematica, la Probabilit\`a
e le loro Applicazioni).




\end{document}